\documentclass[14pt,a4paper,titlepage,draft]{extarticle}

\usepackage[utf8]{inputenc}
\usepackage[english,russian]{babel}
\usepackage{amsfonts,amsthm,amsmath,amssymb,indentfirst,enumerate, amsopn,mmap,afterpage}

\theoremstyle{plain}
\newtheorem{lemma}{Лемма}
\newtheorem{proposition}{Предложение}
\newtheorem*{theorem}{Теорема}

\newtheorem{corollary}{Следствие}
\newtheorem*{problem}{Проблема}

\theoremstyle{definition}
\newtheorem{agreement}{Соглашение}
\newtheorem{definition}{Oпределение}
\newtheorem{definitions}[definition]{Определения}
\newtheorem{notation}{Обозначение}
\newtheorem{notations}[notation]{Обозначения}

\theoremstyle{remark}
\newtheorem{remark}{Замечание}

\newtheorem{example}{Пример}
\newtheorem{examples}[example]{Примеры}

\newcommand{\empha}[1]{\emph{\textbf{#1}}}

\newcommand{\Z}{\mathbf{Z}}
\newcommand{\Q}{\mathbf{Q}}
\newcommand{\Co}{\mathbf{C}}
\newcommand{\R}{\mathbf{R}}

\DeclareMathOperator{\Un}{Un}
\DeclareMathOperator{\I}{Int}
\DeclareMathOperator{\V}{V}
\DeclareMathOperator{\tr}{tr}
\DeclareMathOperator{\N}{Norm}
\DeclareMathOperator{\Gal}{Gal}
\DeclareMathOperator{\ldb}{\text{\upshape\textlbrackdbl}}
\DeclareMathOperator{\rdb}{\text{\upshape\textrbrackdbl}}

\newcommand{\zdb}[1]{\Z{\ldb}{#1}{\rdb}}
\newcommand{\qdb}[1]{\Q{\ldb}{#1}{\rdb}}

\title{Units of integral group rings of cyclic $2$-groups}

\author{Rifkhat~Zh.~Aleeev, Olga~V.~Mitina, Aleksandra~D.~Godova}

\begin{document}

\selectlanguage{english}

\maketitle

\begin{selectlanguage}{english}

\begin{abstract}
This paper is devoted to the units of integral group rings of cyclic $2$-groups of small orders, namely, the orders of $2^n$ for $n\leqslant7$.
Immediately we should note the issues our consideration describe in the introduction in more detail.

Here we will indicate the main directions of our research.
Previously, we proved that the normalized group of units of an integral group ring of a cyclic $2$-group of order $2^n$ contains a subgroup of finite index, which is
the direct product of the subgroup of units defined by the character with the largest character field
and the subgroup of units that is isomorphic to the subgroup of units of the integer group ring of the cyclic $2$-group of order $2^{n-1}$.
Because of this, it is very important to study the structure of the subgroup of units defined by the character with the largest field of characters, which is the cyclotomic field $\Q_{2^n}$
obtained by adjoining a primitive $2^n$th root of unity to $\Q$, the field of rational number.
That subgroup of units of an integral group ring of a cyclic $2$-group is isomorphic to the subgroup of the group of units of the integer ring of the specified cyclotomic field.
Therefore, the research of units of an integer group ring of a cyclic $2$-group is reduced to study the properties of the group of units of the integer ring of the cyclotomic field $\Q_{2^n}$.
In general, the group of units of the ring of integers of the circular field $\Q_{2^n}$ is not known completely.
However, the fundamental paper of Sinnott allows us to always find a subgroup of circular units of a finite index in this group of units,
and this index is equal to the class number of the maximum real subfield of the field $\Q_{2^n}$.
The classical Weber class number problem assumes the class number is equal to $1$, and this would give a coincidence of the subgroup of circular units and the group of all units of the integer ring of the cyclotomic field $\Q_{2^n}$.
As far as we know, the Weber problem is solved positively for all $2^n\leqslant256$, and under the generalized Riemann hypothesis for $2^n\leqslant512$.

Thus, we will study of groups of circular units of integer rings of cyclotomic fields $\Q_{2^n}$  in large part.
First of all, we note one significant point.
In previous works of the authors (Aleev and Mitina) \&\ co-authors, it was used $5$ generates a unit subgroup of the index $2$ modulo $2^n$.
As it turned out, Gauss noted $5$ can be replaced by $3$, that is $3$ generates a unut subgroup of the index $2$ modulo $2^n$.
This made it possible to replace the generators of groups of circular units by simpler ones.
However, this also led to the need to revise a rather extensive list of previously obtained results.
Since some of the results obtained earlier were proved by complicated calculations, it was impossible to simply say  they are obtained similarly.
This transfer of the results took up quite a lot of space in this work.
More precisely, the cases of $n\in\{0,1,2\}$ are trivial, the case of $n=3$ is well known, besides, in this case there is no difference what to consider $3$ or $5$.
For $5$ the cases of $n\in\{4,5,6\}$ were fully studied, so they were necessary to review these cases.
Also, many results were obtained for an arbitrary $n$, they needed to be revised  also.

Let's summarize the results.
We can divide our work into the following parts.
\begin{description}
\item[Preliminary information.]
This is described in section 2.
There we introduce three number sequences of cyclotomic fields  play an essential role in further considerations.
In particular, one of these sequences allows us to describe circular units.
\item[Circular units.]
This is the subject of section 3.
First, we consider the reduction of unit group of circular unit group.
Then we obtain a description of circular unit group by one of the sequences introduced earlier.
\item[Properties of units of group rings.]
More precisely, in section 4 we produce certain units of an integral group ring of a cyclic group of order $2^n$.
These units are determined by a character whose character field is a cyclotomic field $\Q_{2^n}$.
Due to the standard isomorphism, we reduce the study of such units to the study of the units of the integer ring of the cyclotomic field $\Q_{2^n}$.
Namely, we prove necessary to study the units of the integer ring of the cyclotomic field $\Q_{2^n}$ congruent to $1$ modulo $2$.
\item[Congruence of circular units.]
In sections 5--7 we study the congruence of circular units modulo $2$, which plays an important role in finding the units of integral group rings of cyclic $2$-groups.
More specifically, in section 5 a subgroup of circular units is constructed, which claims to describe all circular units that congruent to $1$ modulo $2$.
In section 6 we introduce a subgroup of circular units whose quotient group is an elementary Abelian $2$-group by the subgroup from section 5.
Section 7 is devoted to proving the coincidence of subgroups from sections 5 and 6, which can be done for $2^n\leqslant128$.
As noted earlier, Weber's problem about the number of classes for such cases has a positive solution.
Therefore, for $2^n\leqslant128$ we can state that obtain the description of all units of integer group rings of cyclic $2$-groups determined by the character with the character field $\Q_{2^n}$.
\end{description}
\end{abstract}

\end{selectlanguage}

\begin{selectlanguage}{russian}
\begin{abstract}

\setcounter{page}{4}

Эта работа посвящена единицам целочисленных групповых колец циклических $2$-групп небольших порядков, а именно, порядков $2^n$ для $n\leqslant7$.
Сразу же отметим, что более подробно рассматриваемые вопросы описаны во введении.

Здесь мы укажем основные направления наших исследований в этой работе.
Ранее было доказано, что нормализованная группа единиц целочисленного группового кольца циклической $2$-группы порядка $2^n$ содержит 
в качестве подгруппы конечного индекса прямое произведение подгруппы единиц, определяемых характером с наибольшим полем характеров,
и подгруппы единиц, которая изоморфна подгруппе единиц целочисленного группового кольца циклической $2$-группы порядка $2^{n-1}$.
В силу этого очень важно изучить строение подгруппы единиц, определяемых характером с наибольшим полем характеров, которым является круговое поле $\Q_{2^n}$,
полученное присоединением к полю рациональных чисел примитивного корня из $1$ степени $2^n$.
Упомянутая подгруппы единиц целочисленного группового кольца циклической $2$-группы стандартно изоморфна подгруппе группе единиц кольца целых указанного кругового поля.
Поэтому задача о подгруппе единиц целочисленного группового кольца циклической $2$-группы сводится к изучению свойств группе единиц  кольца целых кругового поля $\Q_{2^n}$.
В общем случае группа единиц кольца целых кругового поля $\Q_{2^n}$ не известна.
Однако, фундаментальная работа Синнотта позволяет находить всегда подгруппу круговых единиц конечного индекса в этой группе единиц, 
причём этот индекс равен числу классов максимального действительного подполя поля $\Q_{2^n}$.
Классическая проблема Вебера о числе классов предполагает, что это число классов равно $1$, а это давало бы совпадение упомянутой подгруппы круговых единиц со всей группой единиц кольца целых кругового поля $\Q_{2^n}$.
Насколько нам известно, проблема Вебера решена положительно для всех $2^n\leqslant256$, а при допущении обобщённой гипотезы Римана также для $2^n\leqslant512$.

Таким образом, мы сосредоточимся, в основном, на исследовании групп круговых единиц колец целых круговых полей $\Q_{2^n}$.
Прежде всего отметим один существенный момент.
В предыдущих работах авторов (Алеев и Митина) с соавторами существенно использовалось, что $5$ порождает по модулю $2^n$ подгруппу единиц индекса $2$.
Как оказалось, ещё Гаусс отметил, что $5$ можно заменить на $3$, то есть $3$ порождает по модулю $2^n$ подгруппу единиц индекса $2$.
Это позволило заменить ранее рассматриваемые порождающие группы круговых единиц на более простые.
Однако, это также повлекло за собой, что необходимо пересмотреть довольно обширный список полученных  ранее результатов.
Поскольку некоторые полученные ранее результаты были доказаны путём довольно сложных вычислений, то нельзя было просто сказать, что  они получаются аналогично.
Этот перенос результатов занял весьма много места в этой работе.
Более точно, случаи $n\in\{0,1,2\}$ тривиальны, случай $n=3$ хорошо известен, кроме того в этом случае нет разницы, что рассматривать $3$ или $5$.
Для $5$ были полностью изучены случаи $n\in\{4,5,6\}$, поэтому нужно было пересмотреть эти случаи.
Также было получено много результатов для произвольного $n$, они тоже нуждались в пересмотре.

Подведём итоги.
Можно разделить всю работу на следующие части.
\begin{description}
\item[Предварительные сведения.] 
Это изложено в разделе 2.
Там вводятся три последовательности чисел из круговых полей, которые играют существенную роль в дальнейших рассмотрениях.
В частности, одна из этих последовательностей позволяет описывать круговые единицы.
\item[Круговые единицы.]
Этому посвящён раздел 3.
Сначала рассматривается редукция группы единиц к круговым единицам.
Потом описание группы круговых единиц в терминах одной из последовательностей, введённых ранее.
\item[Свойства единиц групповых колец.]
Более точно, в разделе 4 производится определённых единиц целочисленного группового кольца циклической группы порядка $2^n$.
Эти единицы определяются характером, у которого полем характеров  является круговое поле $\Q_{2^n}$.
В силу стандартного изоморфизма изучение таких единиц сводится к изучению единиц кольца целых кругового поля $\Q_{2^n}$.
А именно, доказывается, что нужно изучать единицы кольца целых кругового поля $\Q_{2^n}$, сравнимые с $1$ по модулю $2$.
\item[Сравнимость круговых единиц.]
В разделах 5--7  изучается сравнимость круговых единиц  по модулю $2$, которая играет важную роль в нахождении единиц целочисленных групповых колец циклических $2$-групп.
Более определённо, в разделе 5 строится подгруппа круговых единиц, которая претендует на описание всех круговых единиц, сравнимых с $1$  по модулю $2$.
В разделе 6 вводится подгруппа круговых единиц, фактор-группа которой по подгруппе из раздела 5 является элементарной абелевой $2$-группой.
Раздел 7 посвящён доказательству совпадения подгрупп из разделов 5 и 6, что удаётся сделать для $2^n\leqslant128$.
Как отмечалось ранее, проблема Вебера о числе классов для таких случаев имеет положительное решение.
Поэтому можем утверждать, что для $2^n\leqslant128$ описаны все единицы целочисленных групповых колец циклических $2$-групп, которые определяются характером с полем характера $\Q_{2^n}$.
\end{description}
\end{abstract}

\setcounter{page}{7}

\tableofcontents

\allowdisplaybreaks

\section{Введение}

Разобьём цель  работы на 3 части.
\begin{description}
\item[Общая.]
Изучение единиц целочисленных групповых колец циклических $2$-групп.

Это непосредственно связано с изучением единиц колец целых $2$-круговых полей $\Q_{2^n}$
(круговых полей, полученных присоединением к полю рациональных чисел примитивного корня из $1$ степени $2^n$, где $n$ --- неотрицательное целое число).
Классическая проблема Вебера в качестве следствия даёт, что все единицы колец целых $2$-круговых полей являются  круговыми.
Однако проблема Вебера решена $2^n\leqslant512$, причём для $512$ при дополнительном условии.
Тем не менее круговые единицы образуют подгруппу конечного (для  $2^n>512$ неизвестного) индекса в группе всех единиц колец целых $2$-круговых полей.
Поэтому круговые единицы крайне важны.

Используя круговые единицы можно построить подгруппы конечного (возможно даже тривиального) индекса в группе единиц целочисленных групповых колец циклических $2$-групп,
а при допущении справедливости проблемы Вебера можно построить всю группу единиц целочисленного группового кольца любой циклической $2$-группы.
\item[Конкретная.]
Для каждого натурального числа $n$ круговые единицы кольца целых $2$-кругового поля $\Q_{2^n}$ позволяют строить единицы целочисленного группового кольца  циклической группы порядка $2^n$.

Постепенно увеличивая $n$, можно таким способом легко построить подгруппу конечного индекса в группе единиц единиц целочисленного группового кольца любой циклической $2$-группы.
Более точно, опишем, что имеется ввиду.
Пусть $\zeta_{2^n}$ --- примитивный примитивный корень из $1$ степени $2^n$.
\begin{description}
\item[Значения $n\in\{0,1,2\}$.]
Этот случай самый начальный и самый неинтересный.
Так как $-1\in\Q$ --- примитивный корень степени $2$ из $1$, то
\[
\Q_2=\Q_1=\Q.
\]
Поэтому $\Q_2$ \empha{никогда} особо не рассматривают.
Как круговые единицы кольца целых $\Q$ (а это будет кольцо целых чисел $\Z$) возникают только
\[
1\text{ и }-1.
\]
В случае $n=2$ имеем
\[
\Q_4=\Q(i)=\left\{a+bi\mid a,b\in\Q\right\}
\]
--- \emph{поле гауссовых чисел}, его кольцо целых ---
\[
\Z[i]=\left\{a+bi\mid a,b\in\Z\right\}.
\]
Как круговые единицы кольца $\Z[i]$ возникают только
\[
1,-1,i\text{ и }-i.
\]
Во всех этих случаях из круговые единиц в целочисленных групповых кольцах циклических $2$-групп будут возникать только элементы групп и к ним противоположные.
А самое главное состоит в том, что \emph{других единиц нет} в целочисленных групповых кольцах циклических $2$-групп порядков $2^n$ для $n\in\{0,1,2\}$.
\item[Значения $n\in\{3,4\}$.]
Случай $2^3=8$ --- первый нетривиальный.
Он давно и хорошо изучен.
Информацию о нём можно найти либо в диссертации \cite{aleev3+}, либо  в \cite{aleev3}.

Случай $2^4=16$ полностью изучен в работах \cite{amkhan} и \cite{amkhan1}.
\item[Значения $n\in\{5,6,7\}$.]
Много информации в  случае $2^5=32$  найти в \cite{amkhan} и в  случае $2^6=64$ найти в \cite{amkhan2}.

Случай $2^7=128$ не опубликован, только анонсирован в \cite{amgod}.
\item[Значения $n>7$.]
До конца не изучено!
\item[Построение подгруппы конечного индекса.]
Ясно, что
\[
\Q_{2^3}\subset\Q_{2^4}\subset\dots\subset\Q_{2^n}.
\]
Для каждого $k\in\{3,\dots,n\}$ можно построить подгруппу круговых единиц, которая является прямым произведением бесконечных циклических групп,
а потом по этой подгруппе строится подгруппа единиц $U_k$ в целочисленном групповом кольце циклической $2$-группы $G$ порядка $2^n$,
которая тоже является прямым произведением бесконечных циклических групп.
Наконец, получим
\[
\langle-1\rangle\times G\times U_3\times\dots\times U_n.
\]
Это будет подгруппа конечного индекса в группе единиц целочисленного группового кольца данной циклической $2$-групп.

На деле этот метод состоит в решении нескольких задач.
\begin{description}
\item[Первая задача] состоит в построении подходящей подгруппы круговых единиц максимального ранга, которая является прямым произведением бесконечных циклических групп.

Эта задача решена даже двумя способами!
\item[Вторая задача] состоит в построении особой подгруппы единиц в целочисленном групповом кольце циклической $2$-группы, которая также является прямым произведением бесконечных циклических групп.
Эта подгруппа сразу получается из подгруппы, построенной во второй задаче.

Поэтому эта задача решена тоже двумя способами!
\item[Третья задача.] Однако нас будет интересовать вопрос:
\begin{quote}
\emph{как построить максимальную возможную\\ подгруппу во второй задаче?}
\end{quote}
Эта задача решена одним способом для $n\in\{3,4,5,6\}$ в \cite{aleev3}, \cite{amkhan1} и \cite{amkhan2}.
Другим способом данная задача решена для $n=7$.
\end{description}
\end{description}
\item[Частная.]
Наконец подобрались к тому, что нужно сделать!
\begin{quote}
Возникла лакуна, которую \empha{надо} закрыть!
\end{quote}
Более определённо.
Как уже отмечали третья задача решена одним способом для $n\in\{3,4,5,6\}$ и другим для  $n=7$.
Естественно возникает вопрос:
\begin{quote}
Можно ли решить упомянутую задачу новым способом, что применялся для  $n=7$?
\end{quote}
На первый взгляд ответ очевиден --- ДА, КОНЕЧНО!
Однако, так как вычисления для $n\in\{3,4,5,6\}$ были весьма трудные, то <<верхоглядное>> уверение не может удовлетворить!

Теперь перейдём к более точным постановкам.
Будем использовать обозначения, которые будут объяснены в дальнейшем.
\begin{enumerate}
\item
Для $n\in\{3,4,5,6\}$ использовались следующие порождающие группы круговых единиц
\[
t_j=1+s_j+s_{2j},\ j\in\{1,3,5,\dots\}
\]
Для $n=7$ использовались другие порождающие группы круговых единиц
\[
d_j=1+s_j,\ j\in\{1,3,5,\dots\}
\]
\item
Для $n=3$ нужен только $t_1$, но в этом случае
\[
s_2=0\longrightarrow t_1=d_1.
\]
\item
Поэтому надо рассматривать только
\[
n\in\{4,5,6\}
\]
и по возможности значения $n>6$.
\end{enumerate}
\end{description}

\section{Круговое поле $\Q_{2^n}$}

Круговое поле, полученное присоединением первообразного (примитивного) корня из $1$ степени $2^n$, будем  обозначать как $\Q_{2^n}$  или $\Q(\zeta_{2^n})$, где $\zeta_{2^n}$ ---  примитивный корень из $1$ степени $2^n$.

Пусть $K$ --- подполе поля комплексных чисел $\Co$ и $\overline{\Z}$ --- кольцо всех целых алгебраических чисел.
Тогда обозначим через $\I(K)=K\cap\overline{\Z}$ --- кольцо целых поля $K$ и также через $\Un(\I(K))$ ---  группу единиц кольца $\I(K)$.

\subsection{Общие сведения}

\begin{notations}\label{n:m&a}
Удобно ввести следующие обозначения.
\begin{enumerate}
\item
Положим $\zeta_{2^n}=\alpha$  и не ограничивая общности можем считать, что
\[
\alpha=e^{i\frac{2\pi}{2^n}}=\cos\frac{2\pi}{2^n}+i\sin\frac{2\pi}{2^n}.
\]
\item
Тогда, в частности,
\[
\alpha^{2^{n-1}}=-1,\qquad\alpha^{2^{n-2}}=i\qquad\text{и}\text\qquad\alpha^{2^{n-3}}=\frac{\sqrt{2}}{2}(1+i).
\]
\end{enumerate}
\end{notations}

Хорошо известен следующий результат.
\begin{lemma}\label{l:cyc}{\ }
\begin{enumerate}[{\rm1.}]
\item
Круговое поле $\Q(\alpha)$ равно
\[
\Q(\alpha)=\left\{f(\alpha)\mid f\in\Q[t]\right\}=\left\{f(\alpha)\mid f\in\qdb{2^{n-1}-1}[t]\right\},
\]
где $\qdb{2^{n-1}-1}[t]$ --- множество (точнее подпространство) всех многочленов с рациональными коэффициентами степени не выше $2^{n-1}-1$.
Иными словами,
\[
1,\alpha,\alpha^2, \dots,\alpha^{2^{n-1}-1}
\]
--- базис $\Q(\alpha)$ как векторного пространства над полем рациональных чисел $\Q$.
\item
Группа Галуа $\Gal(\Q_{2^n})$ кругового поля $\Q_{2^n}$, или равносильно группа автоморфизмов  поля $\Q_{2^n}$, имеет порядок, равный степени расширения
\[
|\Gal(\Q_{2^n})|=[\Q(\alpha):\Q]=2^{n-1}.
\]
\item
Более того, всякий автоморфизм $\sigma$ поля $\Q_{2^n}$ является расширением по линейности отображения
\[
\sigma:\alpha\mapsto\alpha^k\text{ для подходящего нечётного }k\in\{1,3,\dots,2^n-1\},
\]
то есть, для любого
\begin{multline*}
\beta=b_{2^{n-1}-1}\alpha^{2^{n-1}-1}+\dots+b_1\alpha+b_0\in\Q_{2^n},\\ 
\text{ где }\{b_{2^{n-1}-1},\dots,b_1,b_0\}\subset\Q.
\end{multline*}
имеем
\begin{align*}
\sigma(\beta)&=b_{2^{n-1}-1}\sigma(\alpha^{2^{n-1}-1})+\dots+b_1\sigma(\alpha)+b_0=\\
&=b_{2^{n-1}-1}\alpha^{k(2^{n-1}-1)}+\dots+b_1\alpha^k+b_0.
\end{align*}
Таким образом
\[
\Gal(\Q_{2^n})=\left\{\sigma_k\mid\sigma_k(\alpha)=\alpha^k,\ k\in\{1,3,\dots,2^n-1\}\right\}.
\]
\item
Кольцом целых кругового поля $\Q(\alpha)=\Q_{2^n}$ является
\[
\Z[\alpha]=\left\{f(\alpha)\mid f\in\Z[t]\right\}=\left\{f(\alpha)\mid f\in\zdb{2^{n-1}-1}[t]\right\},
\]
где $\zdb{2^{n-1}-1}[t]$ --- множество (точнее, конечно-порождённая подгруппа) всех многочленов с целыми коэффициентами степени не выше $2^{n-1}-1$.
\item\label{cib}
Из {\rm 1} и {\rm 4} получаем, что
\[
1,\alpha,\alpha^2, \dots,\alpha^{2^{n-1}-1}
\]
--- целый базис расширения $\Q(\alpha)/\Q$.
\end{enumerate}
\end{lemma}

\begin{definition}
Для любого $2^n=2^n$ множество
\[
2\Z[\alpha]=\left\{2\rho\mid\rho\in\Z[\alpha]\right\}
\]
является идеалом в $\Z[\alpha]$.
Поэтому возникает \empha{сравнимость элементов из $\Z[\alpha]$ по модулю этого идеала}.
Для краткости будем писать для $\rho,\sigma\in\Z[\alpha]$
\[
\rho\equiv\sigma\pmod{2},
\]
если  $\rho\equiv\sigma\pmod{2\Z[\alpha]}$, то есть $\rho\in\sigma+2\Z[\zeta_{2^n}]$.
\end{definition}

\begin{remark}\label{r:int}
В силу утверждения \ref{cib} леммы \ref{l:cyc} имеем, что
\[
2\Z[\alpha]=\left\{2\sum_{j=0}^{2^{n-1}-1}b_j\alpha^j\mid \{b_0,b_1,\dots,b_{2^{n-1}-1}\}\subset\Z\right\}.
\]
В частности, получим, что для $b=\sum_{j=0}^{2^{n-1}-1}b_j\alpha^j\in\Z[\alpha]$ сравнение
\[
b\equiv1\pmod{2}
\]
выполняется тогда и только тогда, когда
\[
b_0\text{ --- нечётное число, }b_j\text{ --- чётное число для }j\in\{1,2,\dots,2^{n-1}-1\}.
\]
\end{remark}

\subsection{Три полезных последовательности}

\subsubsection{Последовательность $\left\{s_j\right\}_{j\in\Z}$}

\begin{notation}
Для любого  целого $j$ положим
\[
s_j=\alpha^j+\alpha^{-j}=2\cos\frac{2\pi}{2^n}j=2\cos\frac{\pi}{2^{n-1}}j.
\]
\end{notation}

\begin{remark}
Для любого  целого $j$ число $s_j$ инвариантно относительно действия автоморфизма кругового поля, индуцированного отображением
\[
\alpha\longmapsto\alpha^{-1}=\alpha^{2^n-1}=-\alpha^{2^{n-1}-1}
\]
и, очевидно, состоит из действительных чисел, то есть
\[
\left\{s_j\right\}_{j\in\Z}\subset\R.
\]
\end{remark}

Свойства последовательности $\left\{s_j\right\}_{j\in\Z}$ изучены в лемме 2 работы \cite{amkhr}, поскольку там использовались другие обозначения.
\begin{lemma}\label{l:ms}
Последовательность чисел $\left\{s_j\right\}_{j\in\Z}$  обладает следующими свойствами.
\begin{enumerate}[{\rm1.}]
\item
Последовательность $\left\{s_j\right\}_{j\in\Z}$ периодична с периодом $2^n$ и
\begin{gather*}
s_0=2,\  s_{2^{n-1}}=-2,\  s_{2^{n-2}}=s_{3\cdot2^{n-2}}=0,\\
 s_{2^{n-3}}=s_{7\cdot2^{n-3}}=\sqrt{2},\ s_{3\cdot2^{n-3}}=s_{5\cdot2^{n-3}}=-\sqrt{2}.
\end{gather*}
\item
Набор $(s_0,s_1,\dots,s_{2^n})$ симметричен относительно центра, то есть для $j\in\{0,1,2,\dots,2^{n-1}\}$
\[
s_{2^n-j}=s_j\longleftrightarrow s_{2^{n-1}+j}=s_{2^{n-1}-j}.
\]
Набор $(s_0,s_1,\dots,s_{2^{n-1}})$ антисимметричен относительно цент\-ра, то
есть для $j\in\{0,1,2,\dots,2^{n-2}\}$
\[
s_{2^{n-1}-j}=-s_j\longleftrightarrow s_{2^{n-2}+j}=-s_{2^{n-2}-j}.
\]
Таким образом, для любого $j\in\{1,2,\dots,2^{n-2}-1\}$
\[
s_j=-s_{2^{n-1}-j}=-s_{2^{n-1}+j}=s_{2^n-j}.
\]
Также
\[
\begin{array}{|c||c|c|c|c|}\hline
j& 0 & 1  &  \dots  & 2^{n-3}-1 \\ \hline
s_j&2   & s_1   & \dots & s_{2^{n-3}-1}    \\ \hline
j& 2^{n-3} & 2^{n-3}+1 & \dots &2^{n-2}-1 \\ \hline
s_j& \sqrt{2}  & s_{2^{n-3}+1} & \dots & s_{2^{n-2}-1}   \\ \hline\hline
j& 2^{n-2}& 2^{n-2}+1 & \dots & 3\cdot2^{n-3}-1   \\ \hline
s_j & 0& -s_{2^{n-2}-1} & \dots &-s_{2^{n-3}+1}  \\ \hline
j&  3\cdot2^{n-3} & 3\cdot2^{n-3}+1 & \dots & 2^{n-1}-1  \\ \hline
s_j &- \sqrt{2} &-s_{2^{n-3}-1} & \dots & -s_1 \\ \hline\hline
j&2^{n-1}& 2^{n-1}+1 & \dots & 5\cdot2^{n-3}-1 \\ \hline
s_j&-2&-s_1            & \dots &-s_{2^{n-3}-1} \\ \hline
j& 5\cdot2^{n-3} & 5\cdot2^{n-3}+1 & \dots & 4\cdot2^{n-2}-1 \\ \hline
s_j&-\sqrt{2} & -s_{2^{n-3}+1}& \dots &-s_{2^{n-2}-1}  \\ \hline\hline
j&4\cdot2^{n-2}&4\cdot2^{n-2}+1& \dots &7\cdot2^{n-3}-1 \\ \hline
s_j&0&s_{2^{n-2}-1}  & \dots & s_{2^{n-3}+1} \\ \hline
j&7\cdot2^{n-3}&7\cdot2^{n-3}+1& \dots &s_{2^n-1}   \\ \hline
s_j&  \sqrt{2} &s_{2^{n-3}-1}&\dots & s_1              \\ \hline
\end{array}
\]
\item
Для любых целых $j$ и $l$
\[
s_js_l=s_{j+l}+s_{l-j},
\]
в частности, $s_j^2=s_{2j}+2$ и
\begin{gather*}
s_0^2=s_{2^{n-1}}^2=4,\ s_{2^{n-2}}^2=s_{3\cdot2^{n-2}}^2=0\text{ и }\\
s_{2^{n-3}}^2=s_{3\cdot2^{n-3}}^2=s_{5\cdot2^{n-3}}^2=s_{7\cdot2^{n-3}}^2=2.
\end{gather*}
\end{enumerate}
\end{lemma}

Отсюда для последующих применений извлечём очевидное, но весьма полезное следствие.
\begin{lemma}\label{l:ms2}
Последовательность $\left\{s_j\right\}_{j\in\Z}$ по модулю $2$ периодична с периодом $2^{n-1}$ и имеет следующие свойства.
\begin{enumerate}[{\rm 1.}]
\item
$s_0\equiv s_{2^{n-2}}\equiv0\pmod{2},\  s_{2^{n-3}}\equiv s_{3\cdot2^{n-3}}\equiv\sqrt{2}\pmod{2}$.
\item
Набор $\left(s_0,s_1,\dots,s_{2^{n-1}}\right)$ симметричен относительно центра, то есть для $j\in\{0,1,2,\dots,2^{n-2}\}$
\[
s_{2^{n-1}-j}\equiv s_j\pmod{2}\longleftrightarrow s_{2^{n-2}+j}=s_{2^{n-2}-j}\pmod{2}.
\]
Таким образом
\begin{align*}
\left(s_0,s_1,\dots,s_{2^{n-1}}\right)\equiv(&s_0\equiv0,s_1,\dots,s_{2^{n-3}}\equiv\sqrt{2},\dots,s_{2^{n-2}-1},\\
&s_{2^{n-2}}\equiv0,s_{2^{n-2}+1}\equiv s_{2^{n-2}-1},\dots,\\
&s_{3\cdot2^{n-3}}\equiv s_{2^{n-3}}\equiv\sqrt{2},\dots,\\
&s_{2^{n-1}-1}\equiv s_1,s_{2^{n-1}}\equiv0)\pmod{2},
\end{align*}
или таблица по модулю $2$
\[
\begin{array}{|c||c|c|c|c|}\hline
j& 0 & 1  &  \dots  & 2^{n-3}-1\\ \hline
s_j&0   & s_1   & \dots & s_{2^{n-3}-1}  \\ \hline
j& 2^{n-3} & 2^{n-3}+1 & \dots &2^{n-2}-1 \\ \hline
s_j& \sqrt{2}  & s_{2^{n-3}+1} & \dots & s_{2^{n-2}-1}   \\ \hline\hline
j& 2^{n-2}& 2^{n-2}+1 & \dots & 3\cdot2^{n-3}-1   \\ \hline
s_j & 0& s_{2^{n-2}-1} & \dots &s_{2^{n-3}+1} \\ \hline
j&3\cdot2^{n-3} & 3\cdot2^{n-3}+1 & \dots & 2^{n-1}-1  \\ \hline
s_j &\sqrt{2} &s_{2^{n-3}-1} & \dots &s_1\\ \hline
\end{array}
\]
\item
Для любых целых $j$ и $k$
\[
s_js_k=s_{j+k}+s_{k-j},
\]
в частности, $s_j^2\equiv s_{2j}\pmod{2}$ и
\[
s_0^2\equiv s_{2^{n-2}}^2\equiv s_{2^{n-3}}^2\equiv s_{3\cdot2^{n-3}}^2\equiv0\pmod{2}.
\]
\end{enumerate}
\end{lemma}

\subsubsection{Таблицы $\left\{s_j\right\}_{j\in\{0,\dots,2^{n-1}-1\}}$ по модулю $2$\\ для $2^n\in\{16,32,64,128\}$ $\longleftrightarrow$ $m\in\{1,2,4,8\}$}\label{sss:ts}

Основные вычисления будут проводиться по модулю $2$, поэтому удобно иметь соответствующие таблицы.

\begin{table}[h]
\caption{Таблица для $16$}
\[
\begin{array}{|c||c|c|c|c|c|c|c|c|}\hline
j& 0 & 1  & 2 &3 \\ \hline
s_j&0  & s_1& \sqrt{2}& s_3   \\ \hline\hline
j& 4& 5 & 6 & 7  \\ \hline
s_j & 0&s_3&\sqrt{2}&s_1\\ \hline
\end{array}
\]
\end{table}

\begin{table}[!h]
\caption{Таблица для $32$}
\[
\begin{array}{|c||c|c|c|c|c|c|c|c|}\hline
j& 0 & 1  &  2  & 3 & 4 & 5 & 6 &7 \\ \hline
s_j&0&s_1&s_2&s_3&\sqrt{2}&s_5&s_6& s_7 \\ \hline\hline
j&8&9&10&11&12&13&14&15 \\ \hline
s_j& 0& s_7&s_6 &s_5&\sqrt{2} &s_3&s_2&s_1\\ \hline
\end{array}
\]
\end{table}

\begin{table}[!h]
\caption{Таблица для $64$}
\[
\hspace*{-21pt}
\begin{array}{|c||c|c|c|c|c|c|c|c|c|c|c|c|c|c|c|c|}\hline
j& 0 & 1  &  2  & 3 & 4 & 5 & 6 &7 &8 &9 &10 &11 &12 &13 &14 &15\\ \hline
s_j&0&s_1&s_2&s_3&s_4 &s_5&s_6& s_7 &\sqrt{2} &s_9 &s_{10} &s_{11} &s_{12} &s_{13} &s_{14} &s_{15} \\ \hline\hline
j&16 &17 &18 &19 &20 &21 &22 &23 &24 &25 &26 &27 &28 &29 &30 &31 \\ \hline
s_j& 0& s_{15}&s_{14} &s_{13} &s_{12} &s_{11}&s_{10}&s_9&\sqrt{2} &s_7 &s_6 &s_5 &s_4 &s_3 &s_2 &s_1 \\ \hline
\end{array}
\]
\end{table}

\begin{table}[!h]
\caption{Таблица для $128$}
\[
\hspace*{-32pt}
\begin{array}{|c||c|c|c|c|c|c|c|c|c|c|c|c|c|c|c|c|}\hline
j& 0 & 1  &  2  & 3 & 4 & 5 & 6 &7 &8 &9 &10 &11 &12 &13 &14 &15\\ \hline
s_j&0&s_1&s_2&s_3&s_4 &s_5&s_6& s_7 &s_8 &s_9 &s_{10} &s_{11} &s_{12} &s_{13} &s_{14} &s_{15} \\ \hline\hline
j&16 &17 &18 &19 &20 &21 &22 &23 &24 &25 &26 &27 &28 &29 &30 &31 \\ \hline
s_j&\sqrt{2}& s_{17}&s_{18} &s_{19} &s_{20} &s_{21}&s_{22}&s_{23}&s_{24}&s_{25} &s_{26} &s_{27} &s_{28} &s_{29} &s_{30} &s_{31} \\ \hline\hline
j&32 &33 &34 &35 &36 &37 &38 &39 &40 &41 &42 &43 &44 &45 &46 &47 \\ \hline
s_j& 0& s_{31}&s_{30} &s_{29} &s_{28} &s_{27}&s_{26}&s_{25}&s_{24} &s_{23} &s_{22} &s_{21} &s_{20} &s_{19} &s_{18} &s_{17} \\ \hline\hline
j&48 &49 &50 &51 &52 &53 &54 &55 &56 &57 &58 &59 &60 &61 &62 &63 \\ \hline
s_j&\sqrt{2}& s_{15}&s_{14} &s_{13} &s_{12} &s_{11}&s_{10}&s_9&\sqrt{2} &s_7 &s_6 &s_5 &s_4 &s_3 &s_2 &s_1 \\ \hline
\end{array}
\]
\end{table}

\subsubsection{Последовательность $\left\{d_j\right\}_{j\in\Z}$}\label{sss:u}

\begin{remark}
Для любого целого $j$ положим
\[
t_j=1+s_j+s_{2j}=1+\alpha^j+\alpha^{-j}+\alpha^{2j}+\alpha^{-2j}=1+2\cos\frac{\pi}{2^{n-2}}j+2\cos\frac{\pi}{2^{n-1}}j.
\]
Свойства последовательности $\left\{t_j\right\}_{j\in\Z}$ изучены в лемме 3 работы \cite{amkhr}.
В прежних работах и вариантах рассматривались такие числа.

Сделаем попытку отказаться от них, заменив их на числа более простого вида.
\end{remark}

\begin{notation}
Для любого целого $j$ положим
\[
d_j=1+s_j=1+\alpha^j+\alpha^{-j}=1+2\cos\frac{\pi}{2^{n-1}}j.
\]
Обозначение $d$ связано с немецким словом
\begin{quote}
\empha{drei} --- три.
\end{quote}
\end{notation}

\begin{lemma}\label{l:md}
Последовательность чисел $\left\{d_j\right\}_{j\in\Z}$  обладает следующими свойствами.
\begin{enumerate}[{\rm 1.}]
\item
Последовательность $\left\{d_j\right\}_{j\in\Z}$ периодична с периодом $2^n$ и
\begin{gather*}
d_0=3,\ d_{2^{n-1}}=-1,\ d_{2^{n-2}}=d_{3\cdot2^{n-2}}=1,\\
d_{2^{n-3}}=d_{7\cdot2^{n-3}}=1+\sqrt{2}\text{ и } d_{3\cdot2^{n-3}}=d_{5\cdot2^{n-3}}=1-\sqrt{2}.
\end{gather*}
\item
Набор $(d_0,d_1,\dots,d_{2^n})$ симметричен относительно центра, то есть для $j\in\{0,1,2,\dots,2^{n-1}\}$
\[
d_{2^n-j}=d_j\longleftrightarrow d_{2^{n-1}+j}=d_{2^{n-1}-j}.
\]
Также $d_{2^{n-1}\pm j}=d_j-2s_j$ для любого $j\in\{0,1,2,\dots,2^{n-1}-1\}$, то есть $d_{2^{n-1}\pm j}\equiv d_j\pmod{2}$.
\item
Для любого целого $j$
\[
d_j^2=d_{2j}+2d_j,
\]
в частности, $d_j^2\equiv d_{2j} \pmod{2}$ и
\begin{gather*}
d_0^2=9,\ d_{2^{n-1}}^2=1,\ d_{2^{n-2}}^2=d_{3\cdot2^{n-2}}^2=1,\\
d_{2^{n-3}}^2=d_{7\cdot2^{n-3}}^2=3+2\sqrt{2},\ d_{3\cdot2^{n-3}}^2=d_{5\cdot2^{n-3}}^2=3-2\sqrt{2}.
\end{gather*}
\item
Для любых целых $j$ и $k$
\[
d_jd_k=1+s_j+s_k+s_{k+j}+s_{k-j}=-3+d_j+d_k+d_{k+j}+d_{k-j}.
\]
\end{enumerate}
\end{lemma}
\begin{proof}
Будем использовать лемму \ref{l:ms}.
\begin{enumerate}
\item
Из пункта 1 леммы \ref{l:ms}
\begin{gather*}
d_0=1+2=3,\ d_{2^{n-1}}=1-2=-1,\ d_{2^{n-2}}=d_{3\cdot2^{n-2}}=1+0=1,\\
 d_{2^{n-3}}=d_{7\cdot2^{n-3}}=1+\sqrt{2},\ d_{3\cdot2^{n-3}}=d_{5\cdot2^{n-3}}=1-\sqrt{2}.
\end{gather*}
\item
По пункту 2 из леммы \ref{l:ms} для любого $j\in\{1,2,\dots,2^{n-1}-1\}$
\[
d_{2^{n-1}-j}=1-s_j=1+s_j-2s_j=d_j-2s_j.
\]
\item
В самом деле, из пункта 3 леммы \ref{l:ms} для любого $j$
\[
d_j^2=(1+s_j)^2=1+(s_{2j}+2)+2s_j=d_{2j}+2d_j.
\]
Из пункта 1 следует остальное.
\item
Имеем
\begin{align*}
d_jd_k&=(1+s_j)(1+s_k)=1+s_j+s_k+s_{k+j}+s_{k-j}=\\
&=(1+s_j)+(1+s_k)+(1+s_{k+j})+(1+s_{k-j})-3=\\
&=d_j+d_k+d_{k+j}+d_{k-j}-3.
\end{align*}
\end{enumerate}
\end{proof}

Отсюда для последующих применений извлечём очевидное, но весьма полезное следствие.
\begin{lemma}\label{l:md2}
Последовательность $\left\{d_j\right\}_{j\in\Z}$ по модулю $2$ периодична с периодом $2^{n-1}$ и имеет следующие свойства.
\begin{enumerate}[{\rm1.}]
\item
$d_0\equiv d_{2^{n-2}}\equiv1\pmod{2},\ d_{2^{n-3}}\equiv d_{3\cdot2^{n-3}}\equiv1+\sqrt{2}\pmod{2}$.

В частности, для любого целого $k$
\[
d_{2^{n-2}k}\equiv1\pmod{2}.
\]
\item
Для любого $j\in\{0,1,2,\dots,2^{n-2}\}$
\[
d_{2^{n-1}-j}\equiv d_j\pmod{2}\longleftrightarrow d_{2^{n-2}+j}\equiv d_{2^{n-2}-j}\pmod{2}.
\]
Таким образом
\begin{align*}
\left(d_0,d_1,\dots,d_{2^{n-1}}\right)\equiv(&d_0\equiv1,d_1,\dots,d_{2^{n-3}}\equiv1+\sqrt{2},\dots,d_{2^{n-2}-1},\\
&d_{2^{n-2}}\equiv1,d_{2^{n-2}+1}\equiv d_{2^{n-2}-1},\dots,\\
&d_{3\cdot2^{n-3}}\equiv d_{2^{n-3}}\equiv1+\sqrt{2},\dots,\\
&d_{2^{n-1}-1}\equiv d_1,d_{2^{n-1}}\equiv1)\pmod{2}.
\end{align*}
\item
Для любого целого $j$
\[
d_j^2\equiv d_{2j}\pmod{2}
\]
и
\[
d_0^2\equiv d_{2^{n-2}}^2\equiv d_{2^{n-3}}^2\equiv d_{3\cdot2^{n-3}}^2\equiv1\pmod{2}.
\]
\item
Для любых целых $j$ и $k$
\[
d_jd_k=1+s_j+s_k+s_{k+j}+s_{k-j}\equiv1+d_j+d_k+d_{k+j}+d_{k-j}\pmod{2}.
\]
\end{enumerate}
\end{lemma}

\subsubsection{Последовательность $\left\{r_j\right\}_{j\in\Z}$}

\begin{notation}
Для любого целого $j$ положим
\begin{align*}
r_j&=s_j+s_{2^{n-2}-j}=(1+s_j)+(1+s_{2^{n-2}-j})-2=d_j+d_{2^{n-2}-j}-2\equiv\\
&\equiv d_j+d_{2^{n-2}-j}\pmod{2}.
\end{align*}
\end{notation}
\begin{remark}
Ясно, что
\begin{align*}
r_j&=s_j+s_{2^{n-2}-j}=2\cos\frac{\pi}{2^{n-1}}j+2\cos\frac{\pi}{2^{n-1}}(2^{n-2}-j)=\\
&=2\cos\frac{\pi}{2^{n-1}}j+2\cos\left(\frac{\pi}{2}-\frac{\pi}{2^{n-1}}j\right)=2\cos\frac{\pi}{2^{n-1}}j+2\sin\frac{\pi}{2^{n-1}}j=\\
&=4\cos\frac{\pi}{4}\cos\left(\frac{\pi}{4}-\frac{\pi}{2^{n-1}}j\right)=2\sqrt{2}\cos\left(\frac{\pi}{4}-\frac{\pi}{2^{n-1}}j\right)=\\
&=2\sqrt{2}\cos\left(\frac{\pi}{2^{n-1}}(2^{n-3}-j)\right)=\sqrt{2}s_{2^{n-3}-j}=s_{2^{n-3}}s_{2^{n-3}-j}.
\end{align*}
\end{remark}

\begin{lemma}\label{l:r1}
Последовательность чисел $\left\{r_j\right\}_{j\in\Z}$ периодична  с периодом $2^n$.
Кроме того, отрезок $\left(r_0,\dots,r_{2^n-1}\right)$  последовательности $\left\{r_j\right\}_{j\in\Z}$ разбивается на части:
\begin{align*}
\left\{r_0=2\right\}&\cup R_0=\left(r_1,\dots,r_{2^{n-3}}=2\sqrt{2},\dots,r_{2^{n-2}-1}\right),\\
 \left\{r_{2^{n-2}}=2\right\}&\cup R_1=\left(r_{2^{n-2}+1},\dots,r_{3\cdot2^{n-3}}=0,\dots,r_{2^{n-1}-1}\right), \\
\left\{r_{2^{n-1}}=-2\right\}&\cup R_2=\left(r_{2^{n-1}+1},\dots,r_{5\cdot2^{n-3}}=-2\sqrt{2},\dots,r_{12^{n-3}-1}\right),\\
\left\{r_{12^{n-3}}=-2\right\}&\cup R_3=\left(r_{12^{n-3}+1},\dots,r_{7\cdot2^{n-3}}=0,\dots,r_{2^n-1}\right).
\end{align*}
Упорядоченные наборы $R_0$, $R_1$, $R_2$ и $R_3$ имеют следующие свойства.
\begin{enumerate}[{\rm 1.}]
\item
$R_2=-R_0$ и $R_3=- R_1$, то есть состоят из противоположных чисел.
\item
Для любого целого числа $j$
\[
r_{2^{n-2}+j}=r_j-2s_{2^{n-2}-j}=r_{-j}.
\]
\item
Каждый из наборов $R_0$ и $R_2$ центрально симметричен, то есть для любого $k\in\{0,2\}$
\begin{gather*}
\hspace*{-7pt}R_k{=}\left(r_{2^{n-2}k+1},\dots,r_{2^{n-3}(2k+1)-1},r_{2^{n-3}(2k+1)},r_{2^{n-3}(2k+1)-1},\dots,r_{2^{n-2}k+1}\right),
\intertext{то есть для любого $j\in\{0,\dots,2^{n-3}-1\}$ имеем}
r_{2^{n-3}(2k+1)+j}=r_{2^{n-3}(2k+1)-j}.
\end{gather*}
Каждый из наборов $R_1$ и $R_3$ центрально антисимметричен, то есть для любого $k\in\{1,3\}$
\begin{gather*}
R_k=\left(r_{2^{n-2}k+1},\dots,r_{2^{n-3}(2k+1)-1},0,-r_{2^{n-3}(2k+1)-1},\dots,-r_{2^{n-2}k+1}\right),
\intertext{то есть для любого $j\in\{0,\dots,2^n-1\}$ имеем}
r_{2^{n-3}(2k+1)+j}=-r_{2^{n-3}(2k+1)-j}.
\end{gather*}
\end{enumerate}
\end{lemma}
\begin{proof}
По лемме \ref{l:ms} получаем, что последовательность чисел $\left\{r_j\right\}_{j\in\Z}$ периодична  с периодом $2^n$.

Сначала вычислим $r_{2^{n-3}k}$ для $k\in\{0,1,\dots,7\}$.
В самом деле, из леммы \ref{l:ms} получаем
\begin{align*}
r_0&=s_0+s_{2^{n-2}}=2+0=2,\\
r_{2^{n-1}}&=s_{2^{n-1}}+s_{2^{n-2}}=-2+0=-2,\\
r_{2^{n-2}}&=s_{2^{n-2}}+s_0=0+2=2,\\
r_{3\cdot2^{n-2}}&=s_{12^{n-3}}+s_{2^{n-1}}=0-2=-2,\\
r_{2^{n-3}}&=s_{2^{n-3}}+s_{2^{n-3}}=2s_{2^{n-3}}=2\sqrt{2},\\
r_{3\cdot2^{n-3}}&=s_{3\cdot2^{n-3}}+s_{2^{n-3}}=-\sqrt{2}+\sqrt{2}=0,\\
r_{5\cdot2^{n-3}}&=s_{5\cdot2^{n-3}}+s_{3\cdot2^{n-3}}=s_{3\cdot2^{n-3}}+s_{3\cdot2^{n-3}}=-\sqrt{2}-\sqrt{2}=-2\sqrt{2},\\
r_{7\cdot2^{n-3}}&=s_{7\cdot2^{n-3}}+s_{5\cdot2^{n-3}}=s_{2^{n-3}}+s_{3\cdot2^{n-3}}=\sqrt{2}-\sqrt{2}=0.
\end{align*}

Теперь исследуем свойства наборов $R_0$, $R_1$, $R_2$ и $R_3$.
\begin{enumerate}
\item
Рассмотрим произвольный элемент $r$ из $R_2\cup R_3$.
Тогда найдётся такой номер 
\begin{gather*}
j\in\{1,2,\dots,2^{n-2}-1\}\cup\{2^{n-2}+1,2^{n-2}+2,\dots,2^{n-1}-1\},\text{ что}\\
r=r_{j+2^{n-1}}.
\end{gather*}
Поэтому из леммы \ref{l:ms} следует
\[
r=s_{j+2^{n-1}}+s_{2^{n-2}-j-2^{n-1}}=-s_j-s_{2^{n-2}-j}=-r_j,\text{ где }r_j\in R_0\cup R_1.
\]
\item
Имеем
\begin{align*}
r_{2^{n-2}+j}&=s_{2^{n-2}+j}+s_{2^{n-1}-(2^{n-2}+j)}=-s_{-2^{n-1}+(2^{n-2}+j)}+s_{-j}=\\
&=s_j-s_{-2^{n-2}+j}=s_j-s_{2^{n-2}-j}=r_j-2s_{2^{n-2}-j}.
\end{align*}
Также
\begin{align*}
r_{-j}&=s_{-j}+s_{2^{n-2}+j}=s_j-s_{-2^{n-1}+(2^{n-2}+j)}+s_{-j}=\\
&=s_j-s_{-2^{n-2}+j}=s_j-s_{2^{n-2}-j}=r_j-2s_{2^{n-2}-j}.
\end{align*}
\item
Теперь в силу утверждения 1 достаточно рассмотреть $R_0$ и $R_1$.
Пусть $j\in\{1,2,\dots,2^{n-3}-1\}$.
Элемент $r_j\in R_0$ в качестве центрально симметричного имеет элемент
\[
r_{2^{n-2}-j}=s_{2^{n-2}-j}+s_{2^{n-2}-2^{n-2}+j}=s_{2^{n-2}-j}+s_j=r_j.
\]
Элемент $r_{2^{n-2}+j}\in R_1$ в качестве центрально симметричного имеет элемент
\[
r_{2^{n-1}-j}.
\]
Применим утверждение 2 и получим
\begin{align*}
r_{2^{n-2}+j}&=r_j-2s_{2^{n-2}-j};\\
r_{2^{n-1}-j}&=r_{2^{n-2}+(2^{n-2}-j)}=r_{2^{n-2}-j}-2s_{2^{n-2}-(2^{n-2}-j)}=r_j-2s_j,\\
r_{2^{n-2}+j}&+r_{2^{n-1}-j}=r_j-2s_{2^{n-2}-j}+r_j-2s_j=2r_j-2r_j=0.
\end{align*}
\end{enumerate}
Лемма доказана.
\end{proof}

Рассмотрим последовательность $\left\{r_j\right\}_{j\in\Z}$, приведённую по модулю $2$.
\begin{lemma}\label{l:r2}
Последовательность $\left\{r_j\right\}_{j\in\Z}$ по модулю $2$ периодична с периодом $2^{n-2}$.
Более точно, в обозначениях леммы {\rm\ref{l:r1}} последовательность $\left\{r_j\right\}_{j\in\Z}$ по модулю $2$ имеет следующие свойства.
\begin{enumerate}[{\rm1.}]
\item
$r_0\equiv r_{2^{n-3}}\equiv r_{2^{n-2}}\equiv0\pmod{2}$.
\item
Для любого целого числа $j$
\[
r_j\equiv r_{-j}\pmod{2}.
\]
\item
$R_0\equiv R_1\equiv R_2\equiv R_3\pmod{2}$ --- здесь имеется в виду поэлементная сравнимость по модулю $2$ упорядоченных наборов $R_0$, $R_1$, $R_2$ и $R_3$.
\item
Набор $R_0$ центрально симметричен по модулю $2$, то есть для любого $j\in\{0,\dots,2^{n-3}-1\}$ имеем
\[
r_{2^{n-3}+j}=r_{2^{n-3}-j}.
\]
Поэтому
\begin{multline*}
R_0\equiv\left(r_1,\dots,r_{2^{n-3}-1},r_{2^{n-3}}\equiv0,\right.\\
\left.r_{2^{n-3}+1}\equiv r_{2^{n-3}-1},\dots,r_{2^{n-2}-1}\equiv r_1\right)\pmod{2},
\end{multline*}
\end{enumerate}
\end{lemma}
\begin{proof}
Всё очевидно следует из леммы \ref{l:r1}.

Лемма доказана.
\end{proof}

\subsubsection{Таблицы $\left\{r_j\right\}_{j\in\{0,\dots,2^{n-2}-1\}}$ по модулю $2$\\ для $2^n\in\{16,32,64,128\}$ $\longleftrightarrow$ $m\in\{1,2,4,8\}$}\label{sss:tr}

Основные вычисления будут проводиться по модулю $2$, поэтому удобно иметь соответствующие таблицы.
\begin{table}[h]
\caption{Таблица для $16$}
\[
\begin{array}{|c||c|c|c|c|c|c|c|c|}\hline
j& 0 & 1  & 2 &3 \\ \hline
r_j&0  & r_1&0& r_1   \\ \hline
\end{array}
\]
\end{table}

\begin{table}[!h]
\caption{Таблица для $32$}
\[
\begin{array}{|c||c|c|c|c|c|c|c|c|}\hline
j& 0 & 1  &  2  & 3 & 4 & 5 & 6 &7 \\ \hline
r_j&0&r_1&r_2&r_3&0&r_3&r_2& r_1 \\ \hline
\end{array}
\]
\end{table}

\begin{table}[!h]
\caption{Таблица для $64$}
\[
\begin{array}{|c||c|c|c|c|c|c|c|c|}\hline
j& 0 & 1  &  2  & 3 & 4 & 5 & 6 &7 \\ \hline
r_j&0&r_1&r_2&r_3&r_4 &r_5&r_6&r_7 \\ \hline\hline
j&8 &9 &10 &11 &12 &13 &14 &15 \\ \hline
r_j&0 &r_7 &r_6 &r_5 &r_4 &r_3 &r_2 &r_1 \\ \hline
\end{array}
\]
\end{table}

\begin{table}[!h]
\caption{Таблица для $128$}
\[
\hspace*{-13pt}
\begin{array}{|c||c|c|c|c|c|c|c|c|c|c|c|c|c|c|c|c|}\hline
j& 0 & 1  &  2  & 3 & 4 & 5 & 6 &7 &8 &9 &10 &11 &12 &13 &14 &15\\ \hline
r_j&0&r_1&r_2&r_3&r_4 &r_5&r_6&r_7 &r_8 &r_9 &r_{10} &r_{11} &r_{12} &r_{13} &r_{14} &r_{15} \\ \hline\hline
j&16 &17 &18 &19 &20 &21 &22 &23 &24 &25 &26 &27 &28 &29 &30 &31 \\ \hline
r_j& 0&r_{15}&r_{14} &r_{13} &r_{12} &r_{11}&r_{10}&r_9&r_8 &r_7 &r_6 &r_5 &r_4 &r_3 &r_2 &r_1 \\ \hline
\end{array}
\]
\end{table}

\begin{table}[p]
\caption{Таблица для $256$}
\[
\hspace*{-24pt}
\begin{array}{|c||c|c|c|c|c|c|c|c|c|c|c|c|c|c|c|c|}\hline
j& 0 & 1  &  2  & 3 & 4 & 5 & 6 &7 &8 &9 &10 &11 &12 &13 &14 &15\\ \hline
r_j&0&r_1&r_2&r_3&r_4 &r_5&r_6&r_7 &r_8 &r_9 &r_{10} &r_{11} &r_{12} &r_{13} &r_{14} &r_{15} \\ \hline\hline
j&16 &17 &18 &19 &20 &21 &22 &23 &24 &25 &26 &27 &28 &29 &30 &31 \\ \hline
r_j&r_{16}&r_{17}&r_{18} &r_{19} &r_{20} &r_{21}&r_{22}&r_{23}&r_{24}&r_{25} &r_{26} &r_{27} &r_{28} &r_{29} &r_{30} &r_{31} \\ \hline\hline
j&32 &33 &34 &35 &36 &37 &38 &39 &40 &41 &42 &43 &44 &45 &46 &47 \\ \hline
r_j& 0&r_{31}&r_{30} &r_{29} &r_{28} &r_{27}&r_{26}&r_{25}&r_{24} &r_{23} &r_{22} &r_{21} &r_{20} &r_{19} &r_{18} &r_{17} \\ \hline\hline
j&48 &49 &50 &51 &52 &53 &54 &55 &56 &57 &58 &59 &60 &61 &62 &63 \\ \hline
r_j&r_{16}&r_{15}&r_{14} &r_{13} &r_{12} &r_{11}&r_{10}&r_9&r_8 &r_7 &r_6 &r_5 &r_4 &r_3 &r_2 &r_1 \\ \hline
\end{array}
\]
\end{table}

\newpage

\subsubsection[Перемножение элементов последовательностей по модулю $2$]{Перемножение элементов последовательностей\\ по модулю $2$}

\begin{lemma}\label{l:sdr}
Элементы последовательностей $\left\{s_j\right\}_{j\in\Z}$, $\left\{d_j\right\}_{j\in\Z}$ и\\ 
$\left\{r_j\right\}_{j\in\Z}$ перемножаются по модулю $2$ следующим образом.
\begin{enumerate}[{\rm1.}]
\item
Для любых целых $j$ и $k$
\begin{align*}
(SS)\quad&
\left\{\begin{aligned}
s_js_k&=s_{k+j}+s_{k-j},\text{ в частности, }\\
s_k^2&\equiv s_{2k}\pmod{2}\text{ и }s_0^2\equiv s_{2^n}^2\equiv0\pmod{2};
\end{aligned}\right.\\
(SD)\quad&
\left\{\begin{aligned}
s_jd_k&=s_j+s_{k-j}+s_{k+j},\text{ в частности,}\\
s_kd_k&=d_k+d_{2k}\equiv s_{2k}+s_{3k}\pmod{2};
\end{aligned}\right.\\
(SR)\quad&
\left\{\begin{aligned}
s_jr_k&=r_{k-j}+r_{k+j},\text{ в частности,}\\
s_kr_k&\equiv r_{2k}\pmod{2}.
\end{aligned}\right.
\end{align*}
\item
Для любых целых $j$ и $k$
\begin{align*}
(DD)\quad&
\left\{\begin{aligned}
d_jd_k&\equiv1+d_j+d_k+d_{k+j}+d_{k-j}\pmod{2},\\
&\text{в частности,}\\
d_k^2&\equiv d_{2k}\pmod{2}\text{ и для любого натурального }l\\
d_k^{2^l}&\equiv d_{2^l k}\pmod{2},\\
d_k^{2^{n-2}}&\equiv d_{2^{n-2}k}\equiv1\pmod{2};
\end{aligned}\right.\\
(DR)\quad&
\left\{\begin{aligned}
d_jr_k&=r_k+r_{k+j}+r_{k-j},\text{ в частности,}\\
d_kr_k&=2+r_k+r_{2k}\equiv r_k+r_{2k}\pmod{2},\\
d_{2k}r_k&=r_k+r_{3k}+r_{-k}\equiv r_{3k}\pmod{2},\\
d_kr_{2k}&=r_{2k}+r_{3k}+r_{-k}\equiv r_k+r_{2k}+r_{3k}\pmod{2}.
\end{aligned}\right.\\
\end{align*}
\item
Для любых целых $j$ и $k$
\[
(RR)\quad r_jr_k\equiv0\pmod{2}.
\]
\end{enumerate}
\end{lemma}
\begin{proof}{\ }
\begin{enumerate}
\item
Произведение $s_js_k$ и его свойства описаны в лемме \ref{l:ms2}.

Отсюда
\[
s_jd_k=s_j(1+s_k)=s_j+s_{k-j}+s_{k+j}.
\]
Если $j=k$, то
\begin{align*}
s_kd_k&=s_k+s_0+s_{2k}=2+s_k+s_{2k}=d_k+d_{2k}\equiv\\
&\equiv s_k+s_{2k}\pmod{2}.
\end{align*}

Далее
\begin{align*}
s_jr_k&=s_j(s_k+s_{2^{n-2}-k})=s_{k-j}+s_{k+j}+s_{2^{n-2}-k+j}+s_{2^{n-2}-k-j}=\\
&=r_{k-j}+r_{k+j}.
\end{align*}
Если $j=k$, то
\[
s_kr_k=r_0+r_{2k}=2+r_{2k}\equiv r_{2k}\pmod{2}.
\]
\item
Произведение $d_jd_k$ и его свойства описаны в лемме \ref{l:md2}.

Применим утверждение 1 и получим
\[
d_jr_k=(1+s_j)r_k=r_k+r_{k+j}+r_{k-j}.
\]
Поэтому при $j=k$ получим
\begin{align*}
d_kr_k&=r_k+r_{k+k}+r_{k-k}=r_k+r_{2k}+r_0=2+r_k+r_{2k}\equiv\\
&\equiv r_k+r_{2k}\pmod{2}.
\end{align*}
\item
По утверждению 1 имеем
\begin{align*}
r_jr_k&=(s_j+s_{2^{n-2}-j})r_k=r_{k-j}+r_{k+j}+r_{k-2^{n-2}+j}+r_{k+2^{n-2}-j}=\\
&=r_{k-j}+r_{2^{n-2}+(k-j)}+r_{k+j}+r_{-2^{n-2}+(k+j)}.
\end{align*}
Так как по лемме \ref{l:r2} последовательность $\left\{r_j\right\}_{j\in\Z}$ по модулю $2$ периодична с периодом $2^{n-2}$, то
\begin{align*}
r_{k-j}\equiv r_{2^{n-2}+(k-j)}\pmod{2},
r_{k+j}\equiv r_{-2^{n-2}+(k+j)}\pmod{2}.
\end{align*}
Следовательно
\[
r_jr_k\equiv2r_{k-j}+2r_{k+j}\equiv0\pmod{2}.
\]
\end{enumerate}
Лемма доказана.
\end{proof}

\subsection[Максимальное действительное подполе $\Q_{2^n}\cap\R$ кругового поля $\Q_{2^n}$]{Максимальное действительное подполе\\ $\Q_{2^n}\cap\R$ кругового поля $\Q_{2^n}$}

\subsubsection{Известные и стандартные результаты}

\begin{lemma}\label{l:su}
Для любого целого $j$
\begin{align*}
s_{2j}&=s_1^{2j}+\sum_{k=0}^{j-1}(-1)^{j-k}\left(C_{j+k}^{j-k}+C_{j+k-1}^{j-k-1}\right)s_1^{2k},\\
s_{2j+1}&=s_1^{2j+1}+\sum_{k=0}^{j-1}(-1)^{j-k}\left(C_{j+1+k}^{j-k}+C_{j+k}^{j-n-1}\right)s_1^{2k+1},\\
s_1^{2j}&=C_{2j}^j +s_{2j}+\sum_{k=1}^{j-1}C_{2j}^k s_{2(j-k)},\\
s_1^{2j+1}&=s_{2j+1}+\sum_{k=1}^{j-1}C_{2j+1}^ns_{2(j-k)+1}.
\end{align*}
\end{lemma}
\begin{proof}
Первые 2 формулы непосредственно следуют из предложения 1 в \cite{amukh}.

По формуле бинома Ньютона
\begin{align*}
(\alpha+\alpha^{-1})^{2j}&=\alpha^{2j}+\sum_{k=1}^{2j-1}C_{2j}^k\alpha^{2j-k}\alpha^{-k}+\alpha^{-2j}=\\
&=\alpha^{2j}+\alpha^{-2j}+\sum_{k=1}^{2j-1}C_{2j}^k\alpha^{2(j-k)}=\\
&=(\alpha^{2j}+\alpha^{-2j})+\sum_{n=1}^{j-1}C_{2j}^n\alpha^{2(j-k)}+C_{2j}^j+\sum_{k=j+1}^{2j-1}C_{2j}^k\alpha^{2(j-k)}=
\intertext{во второй сумме положим $j-k=l-j\longleftrightarrow k=2j-l$}
&=C_{2j}^j+(\alpha^{2j}+\alpha^{-2j})+\sum_{k=1}^{j-1}C_{2j}^k\alpha^{2(j-k)}+\\
&+\sum_{l=1}^{j-1}C_{2j}^{2j-l}\alpha^{-2(j-l)}=
\intertext{так как $C_{2j}^k=C_{2j}^{2j-k}$, то}
&=C_{2j}^j+(\alpha^{2j}+\alpha^{-2j})+\sum_{k=1}^{j-1}C_{2j}^k(\alpha^{2(j-k)}+\alpha^{-2(j-k)}),
\end{align*}
что и надо.

Для нечётной степени аналогично.

Лемма доказана.
\end{proof}

\begin{lemma}\label{l:ib} 
Для максимального действительное подполя $\Q(\alpha)\cap\R$ кругового поля $\Q(\alpha)$ выполняются следующие утверждения.
\begin{enumerate}[{\rm1.}]
\item
Максимальное действительное подполе $\Q(\alpha)\cap\R$ кругового поля $\Q(\alpha)$ равно
\[
\Q(\alpha+\alpha^{-1})=\left\{f(\alpha+\alpha^{-1})\mid f\in\Q[t]\right\}.
\]
\item
Поле $\Q(\alpha)\cap\R$ абелево, $[\Q(\alpha)\cap\R:\Q]=2^{n-2}$ и любой автоморфизм из группы $\Gal(\Q(\alpha)\cap\R)$ индуцируется отображением $\alpha\mapsto\alpha^k$ для  нечётного $k$.
В частности,
\begin{align*}
\Q(\alpha+\alpha^{-1})&=\left\{f(\alpha+\alpha^{-1})\mid f\in\Q[t]\right\}=\\
&=\left\{f(s_1)\mid f\in\qdb{2^{n-2}-1}[t]\right\},
\end{align*}
где $\qdb{2^{n-2}-1}[t]$ --- множество (точнее, подпространство) всех многочленов с рациональными коэффициентами степени не выше $2^{n-2}-1$.
\item\label{intbr}
Максимальное действительное подполе $\Q(\alpha)\cap\R$ кругового поля $\Q(\alpha)$ имеет два следующих целых базиса.
\begin{enumerate}
\item[{\rm а.}]
$1, s_1=\alpha+\alpha^{-1}, s_1^2=(\alpha+\alpha^{-1})^2,\dots, s_1^{2^{n-2}-1}=(\alpha+\alpha^{-1})^{2^{n-2}-1}.$
\item[{\rm б.}]
$1, s_1=\alpha+\alpha^{-1}, s_2=\alpha^2+\alpha^{-2},\dots, s_{2^{n-2}-1}=\alpha^{2^{n-2}-1}+\alpha^{-2^{n-2}+1}.$
\end{enumerate}
В частности, получаем, что кольцо целых $\I\left(\Q(\alpha)\cap\R\right)$ максимального действительного подполя $\Q(\alpha)\cap\R$ кругового поля $\Q(\alpha)$ равно
\begin{align*}
\Z[\alpha+\alpha^{-1}]&=\left\{f(\alpha+\alpha^{-1})\mid f\in\Z[t]\right\}=\\
&=\left\{f(s_1)\mid f\in\zdb{2^{n-2}-1}[t]\right\},
\end{align*}
где $\zdb{2^{n-2}-1}[t]$ --- множество (точнее, конечно-порождённая подгруппа) всех многочленов с целыми коэффициентами степени не выше $2^{n-2}-1$.
\end{enumerate}
\end{lemma}
\begin{proof}{\ }
\begin{enumerate}
\item
 Ясно, что $\Q(\alpha^{-1}+\alpha)\subseteq\Q(\alpha)\cap\R$.
 Пусть $\beta=\sum_ia_i\alpha^i\in\Q(\zeta)\cap\R$, где для любого $i$ число $a_i\in\Q$.
 Тогда $\overline{\beta}=\beta$, но $\overline{\beta}=\sum_ia_i\overline{\alpha}^i=\sum_ia_i\alpha^{-i}$.
 Поэтому $\beta=\sum_i\frac{a_i}{2}(\alpha^i+\alpha^{-i})$.
 Поскольку $\alpha^i+\alpha^{-i}$ выражается с целыми коэффициентами через $\alpha+\alpha^{-1}$ по лемме \ref{l:su}, то $\beta\in\Q(\alpha^{-1}+\alpha)$, то есть, $\Q(\alpha^{-1}+\alpha)\supseteq\Q(\alpha)\cap\R$.
 Таким образом, $\Q(\alpha^{-1}+\alpha)=\Q(\alpha)\cap\R$.
\item
Ясно, что $\Q(\alpha^{-1}+\alpha)=\Q(\alpha)\cap\R$ --- неподвижное относительно комплексного сопряжения подполе поля $\Q(\alpha)$.
Всё следует из леммы~\ref{l:cyc} по основной теореме теории Галуа \cite[\S~58]{vdv} и теореме из \cite[\S~59, с.~202]{vdv}.
\item
\begin{enumerate}
\item[{\rm а.}]
Это результат из \cite{li}.
\item[{\rm б.}]
Данное утверждение сразу следует из утверждения 3.а и леммы \ref{l:su}.
\end{enumerate}
\end{enumerate}
\end{proof}
\begin{remark}\label{r:intr}
Как в замечании \ref{r:int} на с.~\pageref{r:int} в силу утверждения \ref{intbr} леммы \ref{l:ib} имеем, что
\[
2\Z[\alpha+\alpha^{-1}]=\left\{2\sum_{j=0}^{2^{n-2}-1}b_js_j\mid \{b_0,b_1,\dots,b_{2^{n-2}-1}\}\subset\Z\right\}.
\]
В частности, получим, что для $b=\sum_{j=0}^{2^{n-2}-1}b_js_j\in\Z[\alpha+\alpha^{-1}]$ сравнение
\[
b\equiv1\pmod{2}
\]
выполняется тогда и только тогда, когда
\[
b_0\text{ --- нечётное число, }b_j\text{ --- чётное число для }j\in\{1,2,\dots,2^{n-2}-1\}.
\]
\end{remark}

\subsubsection{Особый базис кольца $\Z[\alpha+\alpha^{-1}]$ и связанный с ним идеал}

\begin{definitions}
Обозначим для удобства
\[
\vec{R}=(r_1,\dots,r_{2^{n-3}-1}).
\]
Определим в поле $\Q(\alpha)\cap\R$ новый особый (упорядоченный) базис
\begin{align*}
\vec{B}&=(1,s_1,\dots,s_{2^{n-3}-1},s_{2^{n-3}},\vec{R})=\\
&=(1,s_1,\dots,s_{2^{n-3}-1},s_{2^{n-3}},r_1,\dots,r_{2^{n-3}-1}).
\end{align*}
Также пусть $R_{\Z}$ --- подгруппа (по сложению), порождённая $\vec{R}$, то есть
\[
R_{\Z}=\left\{a_1r_1+\dots+a_{2^{n-3}-1}r_{2^{n-3}-1}\mid\{a_1,\dots,a_{2^{n-3}-1}\}\subset\Z\right\}.
\]
Наконец, определим подгруппу (по сложению)
\[
\widetilde{R}=R_{\Z}+2\Z[s_1]=R_{\Z}+2\Z[\alpha+\alpha^{-1}].
\]
\end{definitions}

\begin{lemma}\label{l:bsp} 
Определённый выше базис $\vec{B}$ поля $\Q(\alpha)\cap\R$ имеет следующие свойства.
\begin{enumerate}[{\rm1.}]
\item
Для любого $j\in\{1,2,\dots,2^{n-3}-1\}$
\begin{align*}
r_j&=s_j+s_{2^{n-2}-j},\\
s_{2^{n-3}+j}&=r_{2^{n-3}-j}-s_{2^{n-3}-j}.
\end{align*}
\item
$\vec{B}$ --- целый базис поля $\Q(\alpha)\cap\R$.
\item
Элементы базиса $\vec{B}$ перемножаются по модулю $2$ так:
\begin{description}
\item[$2(SS\leqslant2^{n-3})$]
если $j\leqslant k<2^{n-3}$ и $k+j\leqslant2^{n-3}$, то
\[
\left\{\begin{aligned}
s_js_k&=s_{k-j}+s_{k+j},\text{ в частности, для $j\leqslant m$}\\
&s_j^2=s_{2j};
\end{aligned}\right.
\]
\item[$2(SS>2^{n-3})$]
если $j\leqslant k<2^{n-3}$ и $k+j>2^{n-3}$, то
\begin{gather*}
\left\{\begin{aligned}
s_js_k&=s_{k-j}+r_{2^{n-3}-(k+j)}-s_{2^{n-3}-(k+j)},\\
&\text{ в частности, для $j>m$}\\
s_j^2&=r_{2^{n-2}-2j}-s_{2^{n-2}-2j};
\end{aligned}\right.
\end{gather*}
\item[$2(SS2^{n-3})$]
если  $j<2^{n-3}$,то
\[
\left\{\begin{aligned}
s_js_{2^{n-3}}&=r_{2^{n-3}-j},\\
s_{2^n}^2&=0;
\end{aligned}\right.
\]
\item[$2(SR<2^{n-3})$]
если $j\leqslant k$ и $k+j<2^{n-3}$, то
\[
\left\{\begin{aligned}
s_jr_k&=r_{k-j}+r_{k+j},\text{ в частности, для $j<m$}\\
s_jr_j&=r_{2j};
\end{aligned}\right.
\]
\item[$2(SR=2^{n-3})$]
если $j\leqslant k$ и $k+j=2^{n-3}$, то
\[
\left\{\begin{aligned}
s_jr_k&=r_{k-j},\text{ в частности,}\\
s_{2^{n-4}} r_{2^{n-4}}&=0;
\end{aligned}\right.
\]
\item[$2(SR>2^{n-3})$]
если $j\leqslant k$ и $k+j>2^{n-3}$, то
\[
\left\{\begin{aligned}
s_jr_k&=r_{k-j}+r_{2^{n-2}-(k+j)},\text{ в частности, для $j>m$}\\
&s_jr_j=r_{2^{n-2}-2j};
\end{aligned}\right.
\]
\item[$2(S2^{n-3}R)$]
\[
s_{2^{n-3}}r_k=0;
\]
\item[$(RS<2^{n-3})$]
если $k<j<2^{n-3}$ и $k+j<2^{n-3}$, то
\[
s_jr_k=r_{j-k}+r_{k+j};
\]
\item[$2(RS=2^{n-3})$]
если $k<j<2^{n-3}$ и $k+j=2^{n-3}$, то
\[
s_jr_k=r_{j-k};
\]
\item[$2(RS>2^{n-3})$]
если $k<j<2^{n-3}$ и $k+j>2^{n-3}$, то
\[
s_jr_k=r_{j-k}+r_{2^{n-2}-(k+j)};
\]
\item[$2(RR)$]
\[
r_jr_k=0.
\]
\end{description}
\end{enumerate}
\end{lemma}
\begin{proof}{\ }
\begin{enumerate}
\item
Для любого $j\in\{1,2,\dots,2^{n-3}-1\}$ по определению $r_j=s_j+s_{2^{n-2}-j}$.
Так как
\[
r_{2^{n-3}-j}=s_{2^{n-3}-j}+s_{2^{n-2}-(2^{n-3}-j)}=s_{2^{n-3}-j}+s_{2^{n-3}+j},
\]
то
\[
s_{2^{n-3}+j}=r_{2^{n-3}-j}-s_{2^{n-3}-j}.
\]
\item
Это утверждение сразу следует из леммы \ref{l:ib} и утверждения 1 данной леммы.
\item
Это утверждение --- непосредственное следствие леммы \ref{l:sdr}.
\end{enumerate}
Лемма доказана.
\end{proof}

\begin{lemma}\label{l:idR}
Подгруппа $\widetilde{R}$ является идеалом кольца $\Z[\alpha+\alpha^{-1}]$, причём
\[
R_\Z^2\subseteq2\Z[\alpha+\alpha^{-1}].
\]
Иными словами $\widetilde{R}/2\Z[\alpha+\alpha^{-1}]$ является идеалом с нулевым умножением фактор-кольца $\Z[\alpha+\alpha^{-1}]/2\Z[\alpha+\alpha^{-1}]$.
\end{lemma}
\begin{proof}
Это непосредственное следствие леммы \ref{l:bsp}.
\end{proof}

\section{Группа единиц кольца $\Z[\alpha]$}

\subsection{Редукция к круговым единицам}

\begin{lemma}\label{l:u}
\[
\Un(\Z[\alpha])=\langle\alpha\rangle\times K,
\]
где $K\subset\Un(\Z[\alpha+\alpha^{-1}])\subset\R$.
\end{lemma}
\begin{proof}
По-видимому, этот результат хорошо известен, и носит фо\-льклорный характер.
Он является обобщением леммы Куммера \cite[Лемма 4, с. 178]{bor}.

В явном виде на с. 119 статьи \cite{sin} и на с. 561 книги \cite{has} указано, что можно считать, что $K\subset\R$, но тогда $K\subset\Z[\alpha+\alpha^{-1}]$, и потому
\[
K\subset\Un(\Z[\alpha+\alpha^{-1}]).
\]
\end{proof}

\begin{definition}
Пусть
\[
P=\left\langle1-\alpha^k\mid k\in\{1,2\dots,2^n-1\}\right\rangle\leqslant\Q_{2^n}^*
\]
--- подгруппа по умножению мультипликативной группы $\Q_{2^n}^*$ кругового поля $\Q_{2^n}$.
Тогда назовём \emph{группой круговых единиц} поля $\Q_{2^n}$
\[
K(\alpha)=P\cap\Un\left(\Z[\alpha]\right).
\]
\end{definition}

\begin{notation}
Обозначим через $h(n)$ --- число классов поля $\Q_{2^n}\cap\R$.
\end{notation}

\begin{lemma}\label{l:sin}
Индекс
\[
\left|\Un\left(\Z[\alpha]\right):K(\alpha)\right|=h(n).
\]
\end{lemma}
\begin{proof}
Всё следует из \cite[Теорема на с. 107 и результаты на с.~119]{sin}.

Лемма доказана.
\end{proof}

Таким образом, группа круговых единиц $K$ является подгруппой индекса $h(n)$ в $\Un\left(\Z[\alpha]\right)$,  и известна следующая проблема.
\begin{problem}[Вебера о числе классов (Weber class number problem)]{\ }\\
Для любого натурального $n$ число классов $h(n)=1$,
\end{problem}

О (почти) современном состоянии этой проблемы можно понять из следующего результата.
\begin{lemma}\label{l:m&fk}{\ }
\begin{enumerate}[{\rm1.}]
\item
$h(n)=1$ для $n\leqslant8$ и при допущении обобщённой гипотезы Римана $h(9)=1$.
\item
Пусть $\ell$ --- простое число.
\begin{enumerate}
\item[{\rm a.}]
Если $\ell$ меньше чем $10^9$, то $\ell$ не делит $h(n)$ для всех $n\geqslant1$.
\item[{\rm б.}]
Если $\ell\not\equiv\pm1\pmod{32}$, то $\ell$ не делит $h(n)$ для всех $n\geqslant1$.
\end{enumerate}
\end{enumerate}
\end{lemma}
\begin{proof}
Это результаты Миллера и Фукуды с Комацу.
Более точно.
\begin{enumerate}
\item
Первый результат получен Миллером в \cite{mil1} и \cite{mil2}.
Следует отметить, что  случаи для $n\leqslant7$ изучались в \cite{mas} и \cite{lin}.
\item
Это доказано в \cite{fk}.
\end{enumerate}
Лемма доказана.
\end{proof}

\subsection{Описание группы круговых единиц}

В пунктах 1.3.1 и 1.3.2 работы \cite{amkhr} была изучена группа $K(\alpha)$.
Однако, в заключении этих исследований использовались числа $t_j$, упоминавшиеся в начале пункта\ref{sss:u}.
Поэтому для полноты изложения приведём результаты из пунктов 1.3.1 и 1.3.2 работы \cite{amkhr}, которые не меняются,
а которые нужно изменить переформулируем и докажем.

Сначала изучим упоминавшуюся ранее подгруппу
\[
P = \left\langle1-\alpha^k\mid k\in\{1,2,\dots,2^n-1\}\right\rangle\leqslant\Q_{2^n}^*.
\]

\begin{lemma}[Лемма 4 в \cite{amkhr}]\label{l:P}
При введённых выше обозначениях
\begin{gather*}
P =\left\langle1-\alpha^j\mid j\in\{1,3,\dots,2^n-1\}\right\rangle\text{ и}\\
\langle\alpha \rangle\leqslant P.
\end{gather*}
\end{lemma}

\begin{lemma}[Лемма 5 в \cite{amkhr}]\label{l:P+}
Для любых $\varepsilon_0\in\{-1,1\}$,\dots, $\varepsilon_{2^{n-2}-1}\in \{-1,1\}$
\[
P=\langle\alpha\rangle\left\langle1-\alpha^{\varepsilon_l(2l+1)}\mid\in\{0,1,\dots,2^{n-2}-1\}\right\rangle.
\]
\end{lemma}

В следующей лемме $3$ заменяет $5$  в лемме 6 в \cite{amkhr}.

\begin{lemma}\label{l:3}
\[
P=\langle\alpha\rangle\left\langle1-\alpha^{3^k}\mid\in\{0,1,\dots,2^{n-2}-1\}\right\rangle.
\]
Таким образом, если $\mu\in P$, то
\[
\mu = \alpha^k \prod_{l=0}^{2^{n-2}-1} (1-\alpha^{3^l})^{k_l}
\]
для подходящих $k\in\{0,1,\dots,2^{n-3}-1\}$ и $k_l, l\in\{0,1,\dots,2^{n-2}-1\}$.
\end{lemma}
\begin{proof}
Согласно \cite[\textbf{Модули, являющиеся степенями числа $2$ (90, 91)}, с.~90--92]{gauss} множество
\[
\left\{\pm 3^k\mid k\in\{0,\dots,2^{n-2}-1=2^{n-2}-1\}\right\}
\]
образует приведённую систему вычетов по модулю $2^n$.
Это означает, что для любого нечётного $j\in\{1,3,\dots,2^n-1\}$ существуют такие $\varepsilon\in\{1,-1\}$ и $k\in\{0,\dots,2^{n-2}-1=2^{n-2}-1\}$, что
\[
j\equiv\varepsilon3^k\pmod{2^n}.
\]
Для каждого $l\in\{0,\dots,2^{n-2}-1=2^{n-2}-1\}$ найдём такие числа $\varepsilon_l\in\{1,-1\}$ и $k_l\in\{0,\dots,2^{n-2}-1=2^{n-2}-1\}$, что
\[
2l+1\equiv\varepsilon_l3^{k_l}\pmod{2^n}.
\]
Покажем, что
\[
\left\{k_l\mid l\in\{0,\dots,2^{n-2}-1=2^{n-2}-1\}\right\}=\left\{0,1,\dots,2^{n-2}-1\right\}.
\]
Допустим, что $k_l=k_p$ для 
\[
\{k_l,k_p\}\subseteq\left\{k_l\mid l\in\{0,\dots,32^{n-3}-1=2^{n-2}-1\}\right\}.
\]
Тогда
\[
\begin{aligned}
2l+1&\equiv\varepsilon_l3^{k_l}\pmod{2^n},\\
2p+1&\equiv\varepsilon_p3^{k_l}\pmod{2^n},
\end{aligned}
\longleftrightarrow
\begin{aligned}
\varepsilon_l(2l+1)&\equiv3^{k_l}\pmod{2^n},\\
\varepsilon_p(2p+1)&\equiv3^{k_l}\pmod{2^n}.
\end{aligned}
\]
Следовательно,
\begin{multline*}
\varepsilon_l(2l+1)\equiv\varepsilon_p(2p+1)\pmod{2^n} \longleftrightarrow\\
\longleftrightarrow 2l+1\equiv\pm(2p+1)\pmod{2^n}.
\end{multline*}
Если $2l+1\equiv 2p+1\pmod{2^n}$, то по выбору $l$ и $p$ имеем $l=p$.

Пусть $2l+1\equiv-(2p+1)\pmod{2^n}$.
Тогда  имеем
\begin{align*}
2(l+p)+2\equiv0\pmod{2^n}&\longleftrightarrow l+p+1\pmod{2^{n-1}}\longleftrightarrow\\
&\longleftrightarrow l\equiv -p-1\pmod{2^{n-1}}\longleftrightarrow\\
&\longleftrightarrow l=2^{n-1}-p-1.
\end{align*}
Получим систему
\[
\left\{
\begin{aligned}
0&\leqslant p\leqslant2^{n-2}-1\\
0&\leqslant2^{n-1}-p-1\leqslant2^{n-2}-1
\end{aligned}
\right.
\longrightarrow
0\leqslant2^{n-1}-1\leqslant2^{n-1}-2,
\]
что невозможно.

Таким образом,
\[
k_l=k_p\longleftrightarrow l=p.
\]
Теперь для любого значения $k\in\{0,\dots,2^{n-2}-1=2^{n-2}-1\}$ найдётся  такое число $l_k\in\{0,\dots,2^{n-2}-1=2^{n-2}-1\}$, что
\[
2l_k+1\equiv\varepsilon_{l_k}3^k\pmod{2^n}\longleftrightarrow\varepsilon_{l_k}(2l_k+1)\equiv3^k\pmod{2^n}.
\]
Откуда по лемме \ref{l:P+}
\[
P=\langle\alpha\rangle\left\langle 1-\alpha^{3^k}\mid k\in\{0,\dots,2^{n-2}-1=2^{n-2}-1\}\right\rangle.
\]
\end{proof}

\begin{lemma}\label{l:norm}
Для любого нечётного $j\in\{1, 3, \dots,2^n - 1\}$ норма в $\Q_{2^n}$
\[
\N(1 - \alpha^j) = 2.
\]
Если $\mu\in P$ и
\[
\mu = \alpha^k \prod_{l=0}^{2^{n-2}-1} (1-\alpha^{3^l})^{k_l}
\]
для подходящих $k\in\{0,1,\dots,2^n-1\}$ и $k_l, l\in\{0,1,\dots,2^{n-2}-1\}$, то
\[
\mu\in K(\alpha)\longleftrightarrow\sum_{l=0}^{2^{n-2}-1} k_l=0.
\]
\end{lemma}
\begin{proof}
Алгебраически сопряжёнными с элементом $\alpha^j$ являются все первообразные корни из $1$ степени $2^n=2^n$
\[
\alpha,\alpha^3,\dots,\alpha^{2^n-1}.
\]
Они являются корнями кругового многочлена $\Phi_{2^n}(y)=y^{2^{n-1}}+1$, поэтому
\[
\Phi_{2^n}(y)=y^{2^{n-1}}+1=\prod_{l=0}^{2^{n-1}-1}(y-\alpha^{2l+1}).
\]
Отсюда
\[
\N(1-\alpha^j)=\prod_{l=0}^{62^{n-2}-1}(1-\alpha^{2l+1})=\Phi_{122^{n-1}}(1)=1+1=2.
\]
Так как $\alpha^k\in K(\alpha)=P\cap\Un(\Z[\alpha])$, то достаточно найти условие обратимости элемента $\prod_{l=0}^{32^{n-3}-1}(1-\alpha^{3^l})^{k_l}$.
Из мультипликативности нормы следует, что
\[
\N\left(\prod_{l=0}^{2^{n-2}-1}(1-\alpha^{3^l})^{k_l}\right)=\prod_{l=0}^{32^{n-3}-1}\left(\N(1-\alpha^{3^l})\right)^{k_l}=2^{\sum_{l=0}^{2^{n-2}-1}k_l}.
\]
Элементы группы $K(\alpha)$ являются единицами в кольце $\Z[\alpha]$, поэтому должны иметь норму $\pm1$.
Это даёт нужное.
\end{proof}

\begin{notations}
Для любого $l\in\{0, 1, \dots,2^{n-2}-1\}$ положим
\[
\beta_l=1+\alpha^{3^l}+\alpha^{2\cdot 3^l}=\alpha^{3^l}\left(\alpha^{- 3^l}+1+\alpha^{3^l}\right)=\alpha^{3^l}d_{3^l}.
\]
\end{notations}

\begin{lemma}\label{l:beta}
Для любого $l\in \{0,1,2,\dots,2^{n-2}-1\}$ элементы  $\beta_l$ лежат в $K(\alpha)$.
\end{lemma}
\begin{proof}
Так как $\N\left(1-\alpha^{3^{l+1}}\right)=2$ и $\N(1-\alpha^{3^l})=2$, то
\[
\N(\beta_l)=\N\left(\frac{1-\alpha^{3^{l+1}}}{1-\alpha^{3^l}}\right)=1.
\]
Поэтому $\beta_l$ будет единицей кольца $\Z[\alpha]$.
\end{proof}

\begin{lemma}\label{l:Kb}
При введённых обозначениях
\[
K=\langle\alpha\rangle\left\langle\beta_l\mid l\in \{0,1,2,\dots,2^{n-2}-1\}\right\rangle.
\]
\end{lemma}
\begin{proof}
Пусть $\mu\in K(\alpha)$.
Тогда по лемме \ref{l:norm}
\[
\mu=\alpha^k\prod_{l=0}^{2^{n-2}-1}(1-\alpha^{3^l})^{k_l}
\]
для подходящих $k\in\{0,1,\dots,2^n-1\}$  и целых $k_l$ ($l\in\{0,\dots,2^{n-2}-1\}$), причём $\sum_{l=0}^{2^{n-2}-1}k_l=0$.

Достаточно найти такие целые $f_0$,\dots, $f_{2^{n-2}-1}$, что
\[
\prod_{l=0}^{2^{n-2}-1}(1-\alpha^{3^l})^{k_l}=\prod_{l=0}^{2^{n-2}-1}\beta_l^{f_l}=\prod_{l=0}^{2^{n-2}-1}\left((1-\alpha^{3^l})^{-f_l}(1-\alpha^{3^{l+1}})^{f_l}\right).
\]
Так как $3^{2^{n-2}}\equiv 1\pmod{2^n}$, то $\alpha^{3^{2^{n-2}}}=\alpha$ и можно рассмотреть систему
\begin{multline*}
\hspace{-3pt}\left\{
\begin{aligned}
k_0&=-f_0+f_{2^{n-2}-1},\\
k_1&=f_0-f_1,\\
k_2&=f_1-f_2,\\
&\vdots\\
k_{2^{n-2}-2}&=f_{2^{n-2}-3}-f_{2^{n-2}-2},\\
k_{2^{n-2}-1}&=f_{2^{n-2}-2}-f_{2^{n-2}-1},
\end{aligned}
\right.
\longleftrightarrow
\left\{
\begin{aligned}
f_{2^{n-2}-1}&=k_0+0,\\
f_1&=f_0-k_1,\\
f_2&=f_1-k_2,\\
&\vdots\\
f_{2^{n-2}-2}&=f_{2^{n-2}-3}-k_{2^{n-2}-2},\\
f_{2^{n-2}-1}&=f_{2^{n-2}-2}-k_{2^{n-2}-1},
\end{aligned}
\right.\longleftrightarrow\\
\longleftrightarrow
\left\{
\begin{aligned}
f_{32^{n-3}-1}&=k_0+f_0,\\
f_1&=f_0-k_1,\\
f_2&=f_0-k_1-k_2,\\
&\vdots\\
f_{2^{n-2}-2}&=f_0-k_1-\dots-k_{2^{n-2}-2},\\
f_{2^{n-2}-1}&=f_0-k_1-\dots-k_{2^{n-2}-1}.
\end{aligned}
\right.
\end{multline*}
Так как
\[
\sum_{l=0}^{2^{n-2}-1}k_l=0\longleftrightarrow k_0=-\sum_{l=1}^{2^{n-2}-1}k_l,
\]
то получим систему
\[
\left\{
\begin{aligned}
f_1&=f_0-k_1,\\
f_2&=f_0-k_1-k_2,\\
&\vdots\\
f_{2^{n-2}-2}&=f_0-k_1-\dots-k_{2^{n-2}-2},\\
f_{2^{n-2}-1}&=f_0-k_1-\dots-k_{2^{n-2}-1},
\end{aligned}
\right.,
\]
позволяющую находить решения  $f_1$,\dots, $f_{2^{n-2}-1}$ при произвольном $f_0$.
\end{proof}

\begin{notations}
Для любого $j\in \{0, 1, \dots, 2^{n-2}-1\}$ положим
\[
L_j=\{0,1,\dots, 2^{n-2}-1\}\setminus\{j\}=\{0,1,\dots,j-1,j+1,\dots,2^{n-2}-1\}.
\]
\end{notations}

\begin{lemma}\label{l:bLj}
Для любого $j\in\{0,1,\dots, 2^{n-2}-1\}$
\[
K(\alpha)=\langle\alpha\rangle\times\prod_{l\in L_j}\langle\beta_l\rangle.
\]
\end{lemma}
\begin{proof}
Поскольку $3^{2^{n-2}}\equiv 1\pmod{2^n}$, то $\alpha^{3^{2^{n-2}}}=\alpha$ и поэтому
\begin{align*}
\prod_{l=0}^{2^{n-2}-1}\beta_l&=\prod_{l=0}^{2^{n-2}-1}\frac{1-\alpha^{3^{l+1}}}{1-\alpha^{3^l}}=\\
&=\frac{1-\alpha^3}{1-\alpha}\cdot\frac{1-\alpha^{3^2}}{1-\alpha^3}\cdot\frac{1-\alpha^{3^3}}{1-\alpha^{3^2}}\cdot\dots\cdot\frac{1-\alpha^{3^{2^{n-2}-1}}}{1-\alpha^{3^{2^{n-2}-2}}}\cdot\frac{1-\alpha^{3^{2^{n-2}}}}{1-\alpha^{3^{2^{n-2}-1}}}=\\
&=\frac{1-\alpha^{3^{2^{n-2}}}}{1-\alpha}=1.
\end{align*}
Следовательно, из порождающих $\left\{\beta_l\mid l\in\{0,\dots,2^{n-2}-1\}\right\}$ можно удалить любой $\beta_l$.
Разложение группы $K$ в прямое произведение следует из того, что во-первых периодическая часть равна $\langle\alpha\rangle$, во-вторых, ранг группы единиц равен $2^{n-2}-1$.
\end{proof}

\begin{lemma}\label{l:tLj}
При введённых обозначениях группа круговых единиц
\[
K(\alpha)=\langle\alpha\rangle\times\prod_{l=0}^{2^{n-2}-2} \langle d_{2l+1} \rangle.
\]
\end{lemma}
\begin{proof}
Для каждого $l\in\{0,\dots,2^{n-2}-2=2^{n-2}-2\}$ найдём такие $\varepsilon_l\in\{1,-1\}$ и $k_l\in\{0,\dots,2^{n-2}-1=2^{n-2}-1\}$, что
\[
2l+1\equiv\varepsilon_l3^{k_l}\pmod{2^n}
\]
и поэтому
\begin{align*}
d_{2l+1}&=1+\alpha^{2l+1}+\alpha^{-(2l+1)}=\alpha^{-(2l+1)}\left(1+\alpha^{2l+1}+\alpha^{2(2l+1)}\right)=\\
&=\alpha^{-(2l+1)}\cdot\frac{1-\alpha^{3(2l+1)}}{1-\alpha^{2l+1}}=\alpha^{-2(2l+1)}\cdot\frac{1-\alpha^{3\varepsilon_l3^{k_l}}}{1-\alpha^{\varepsilon_l3^{k_l}}}=\\
&=\alpha^{-(2l+1)}\cdot\frac{1-\alpha^{\varepsilon_l3^{k_l+1}}}{1-\alpha^{\varepsilon_l3^{k_l}}}.
\end{align*}
Если $\varepsilon_l=1$, то
\[
d_{2l+1}=\alpha^{-(2l+1)}\cdot\frac{1-\alpha^{3^{k_l+1}}}{1-\alpha^{3^{k_l}}}=\alpha^{-(2l+1)}\cdot\beta_{k_l},
\]
и по предыдущей лемме среди $k_l$ не будет точно одного элемента  из
\[
\{0,1,\dots,2^{n-2}-1\}.
\]

Пусть $\varepsilon_l=-1$. Тогда имеем
\[
\frac{1-\alpha^{-3^{k_l+1}}}{1-\alpha^{-3^{k_l}}}=\frac{-\alpha^{-3^{k_l+1}}}{-\alpha^{-3^{k_l}}}\cdot\frac{1-\alpha^{3^{k_l+1}}}{1-\alpha^{3^{k_l}}}=
\alpha^{-3^{k_l+1}+3^{k_l}}\cdot\beta_{k_l}=\alpha^{-2\cdot 3^{k_l}}\cdot\beta_{k_l},
\]
и снова всё доказано по предыдущей лемме.
\end{proof}

\begin{notation}
Положим
\[
D=\prod_{l=0}^{2^{n-2}-2}\left\langle d_{2l+1}\right\rangle=\left\langle d_1\right\rangle\times\left\langle d_3\right\rangle\times\dots\times\left\langle d_{2^{n-1}-3}\right\rangle.
\]
\end{notation}

Как следствие леммы \ref{l:tLj} получим.
\begin{lemma}\label{l:ur}{\ }
\begin{enumerate}[{\rm1.}]
\item
$K(\alpha)=\langle\alpha\rangle\times D$.
\item
$K(\alpha)\cap\R=\langle-1\rangle\times D$.
\end{enumerate}
\end{lemma}
\begin{proof}{\ }
\begin{enumerate}
\item
Это непосредственное следствие леммы \ref{l:tLj}.
\item
Всё следует из утверждения 1 и леммы \ref{l:u}.
\end{enumerate}

Лемма доказана.
\end{proof}

\section{Подгруппа $W_1$}

\subsection{Свойства $W_1$}

\begin{definition}
Пусть $\chi_1$ --- характер группы $G$ с
\[
\chi_1(x)=\alpha.
\]
Определим подгруппу $W_1$ нормализованной группы единиц $\V(\Z G)$ целочисленного группового кольца $\Z G$ циклической группы $G$ следующим образом:
\[
W_1=\left\langle u_{\chi_1}(\beta_1)\in\V(\Z G)\mid \beta_1\in\Un(\I(\Q(\chi_1)))\right\rangle.
\]
\end{definition}

\begin{remark}
Из определения единиц $u_{\chi_1}(\beta_1)$  в работе \cite[Определение 1]{aleev1} следует мультипликативность таких единиц, что даёт
\[
W_1=\left\{ u_{\chi_1}(\beta_1)\in\V(\Z G)\mid \beta_1\in\Un(\I(\Q(\chi_1)))\right\}.
\]
Как уже отмечалось ранее
\[
\Un(\I(\Q(\chi_1)))=\Un(\Z[\alpha]).
\]
\end{remark}

\begin{notation}
Обозначим след элемента $c\in\Q_{2^n}$ как:
\[
\tr_{\Q_{2^n}}(c),
\]
опуская иногда индекс $\Q_{2^n}$, то есть будем писать просто
\[
\tr(c).
\]
\end{notation}

Как непосредственное следствие леммы 1 из \cite{aleev1}, получаем следующий результат.
\begin{lemma}\label{l:locun}
Пусть $\beta_1\in\Un(\Z[\alpha])$ и
\[
u_{\chi_1}(\beta_1)=\sum_{j=0}^{2^n-1}\gamma_jx^j.
\]
Тогда для любого $j\in\{0,1,\dots,2^n-1\}$:
\[
\gamma_j=\begin{cases}
 1+\dfrac{\tr_{\Q_{2^n}}(\beta_1-1)}{2^n},&\text{если }j=0,\\
 \dfrac{1}{2^n}\tr_{\Q_{2^n}}\left((\beta_1-1)\alpha^{-j}\right),&\text{если }j\in\{1,\dots,2^n-1\}.
\end{cases}
\]
\end{lemma}

\begin{remark}
\empha{Поскольку нужно получить единицу
\[
u_{\chi_1}(\beta_1)=\sum_{j=0}^{2^n-1}\gamma_jx^j\in\V(\Z G),
\]
то необходимо найти такие условия на $\beta_1\in\Un(\Z[\alpha])$, чтобы обеспечивалась целочисленность $\gamma_j$ для всех} 
\[
j\in\{0,1,\dots,2^n-1\}.
\]
\end{remark}

Следующая лемма очевидна, потому что $y^{2^{n-1}}+1$ --- минимальный многочлен числа $\alpha$.
\begin{lemma}\label{l:tra}{\ }
\begin{enumerate}[{\rm1.}]
\item
Для любого $j\in\{0,1,\dots,2^n-1\}$
\[
\tr_{\Q_{2^n}}(\alpha^{-j})=\begin{cases}
 2^{n-1},&\text{если }j=0,\\
-2^{n-1},&\text{если }j=2^{n-1},\\
0&\text{для всех остальных }j.
\end{cases}
\]
\item
Пусть $\beta_1=\sum_{k=0}^{2^{n-1}-1}b_k\alpha^k\in\Un(\Z[\alpha])$.
Тогда для любого 
\[
j\in\{0,1,\dots,2^n-1\}
\]
имеем
\[
\tr_{\Q_{2^n}}(\beta_1\alpha^{-j})=
\begin{cases}
 2^{n-1}b_j&\text{при }j\in\{0,1,\dots,2^{n-1}-1\},\\
-2^{n-1}b_{j-2^{n-1}}&\text{при }j\in\{2^{n-1},\dots,2^n-1\}.
\end{cases}
\]
\end{enumerate}
\end{lemma}

\begin{proposition}\label{pr:ub1z}
Пусть $\beta_1\in\Un(\Z[\alpha])$.
Локальная единица 
\[
u_{\chi_1}(\beta_1)\in\V(\Z G) 
\]
тогда и только тогда, когда
\begin{enumerate}[{\rm1)}]
\item
$\beta_1\in\Un(\Z[\alpha+\alpha^{-1}])=\langle-1\rangle\times K$, где $K$ как в лемме {\rm\ref{l:u}},
\item
причём $\beta_1\equiv1\pmod{2}$.
\end{enumerate}
\end{proposition}
\begin{proof}{\ }
\begin{description}
\item[Необходимость.]
Пусть $\beta_1=\sum_{k=0}^{2^{n-1}-1}b_k\alpha^k$.
Тогда  по лемме \ref{l:tra}
\begin{align*}
\gamma_0\in\Z&\longleftrightarrow\tr_{\Q_{2^n}}(\beta_1-1)\equiv0\pmod{2^n}\longleftrightarrow\\
&\longleftrightarrow2^{n-1}b_0\equiv2^{n-1}\pmod{2^n}\longleftrightarrow b_0\equiv1\pmod{2},\\
\gamma_{2^{n-1}}\in\Z&\longleftrightarrow\tr_{\Q_{2^n}}((\beta_1-1)(-1))\equiv0\pmod{2^n}\longleftrightarrow\\
&\longleftrightarrow2^{n-1}b_0\equiv2^{n-1}\pmod{2^n}\longleftrightarrow b_0\equiv1\pmod{2},
\intertext{для любого $j\in\{1,2,\dots,2^{n-1}-1\}$}
\gamma_j\in\Z&\longleftrightarrow\tr_{\Q_{2^n}}\left((\beta_1-1)\alpha^{-j}\right)\equiv0\pmod{2^n}\longleftrightarrow\\
&\longleftrightarrow2^{n-1}b_j\equiv0\pmod{2^n}\longleftrightarrow b_j\equiv0\pmod{2},
\intertext{для любого $j\in\{2^{n-1},2^{n-1}+1,\dots,2^n-1\}$}
\gamma_j\in\Z&\longleftrightarrow\tr_{\Q_{2^n}}\left((\beta_1-1)\alpha^{-j}\right)\equiv0\pmod{2^n}\longleftrightarrow\\
&\longleftrightarrow-2^{n-1}b_{j-2^{n-1}}\equiv0\pmod{2^n}\\
&\longleftrightarrow b_{j-2^{n-1}}\equiv0\pmod{2}.
\end{align*}
Это даёт нам, что $\beta_1\equiv1\pmod{2}$.

С другой стороны
\[
\beta_1\in\Un(\Z[\alpha])=\langle\alpha\rangle\times K.
\]
Поэтому
\[
\beta_1=\alpha^l\cdot k,
\]
где $l\in\{0,1,\dots,2^n-1\}$ и $k\in K$.

Предположим, что $l\notin\{0,62^{n-2}\}$.
В этом случае
\[
\mbox{\text{НОД}}(l,2^n)=2^s\in\{1,2,4,\dots,2^{n-2}\}\longleftrightarrow ll_1+2^n n_1=2^s
\]
для подходящих целых $l_1$ и $n_1$.
Отсюда
\[
2^{n-2}=l(l_1\cdot2^{n-2-s})+2^n(n_1\cdot2^{n-2-s}),
\]
что влечёт, что
\begin{align*}
i&=\alpha^{2^{n-2}}=\alpha^{l(l_1\cdot2^{n-2-s})+2^n(n_1\cdot2^{n-2-s})}=(\alpha^l)^{l_1\cdot2^{n-2-s}}(\alpha^{2^n})^{n_1\cdot2^{n-2-s}}=\\
&=(\alpha^l)^{l_1\cdot2^{n-2-s}}
\end{align*}
и также
\[
\beta=\beta_1^{l_1\cdot2^{n-2-s}}=(\alpha^k)^{l_1\cdot2^{n-2-s}}\cdot k^{l_1\cdot2^{n-2-s}}=i\cdot k^{l_1\cdot2^{n-2-s}}
\]
Для  подходящих $a_j\in\Z$ для всех $j\in\{0,1,\dots,2^{n-2}-1\}$ по лемме \ref{l:ib} имеем
\begin{align*}
\beta&=i\left(a_0+\sum_{j=1}^{2^{n-2}-1}a_j(\alpha^j+\alpha^{-j})\right)=\\
&=a_0\cdot i+\sum_{j=1}^{2^{n-2}-1}a_j(\alpha^{j+2^{n-2}}+\alpha^{-j+2^{n-2}}).
\end{align*}
Очевидно
\[
\left\{j+2^{n-2},-j+2^{n-2}\mid j\in\{1,2,\dots,2^{n-2}\}\right\}\cap\{0,2^{n-1}\}=\varnothing,
\]
что даёт противоречие с утверждением 1.
Таким образом
\[
\beta_1=\pm k.
\]
Поэтому по лемме \ref{l:ur} получим
\[
\beta_1\in\Un(\Z[\alpha+\alpha^{-1}])=\langle-1\rangle\times K.
\]
\item[Достаточность.]
Из леммы \ref{l:tra} нетрудно получается доказательство достаточности.
\end{description}

Предложение доказано.
\end{proof}

\begin{lemma}\label{l:-1}
$\langle u_{\chi_1}(-1)\rangle=\langle x^{2^{n-1}}\rangle$.
\end{lemma}
\begin{proof}
Пусть
\[
u_{\chi_1}(-1)=\sum_{j=0}^{2^n-1}\gamma_jx^j.
\]
Тогда для любого $j\in\{0,1,\dots,2^n-1\}$ по леммам \ref{l:locun} и \ref{l:tra}:
\[
\gamma_j=\begin{cases}
 1+\dfrac{\tr_{\Q_{2^n}}(-1-1)}{2^n}=1-\dfrac{2^n}{2^n}=0,&\text{если }j=0,\\
 \dfrac{1}{2^n}\tr_{\Q_{2^n}}\left((-1-1)(-1)\right)=\dfrac{2^n}{2^n}=1,&\text{если }j=2^{n-1},\\
 \dfrac{1}{2^n}\tr_{\Q_{2^n}}\left((-1-1)\alpha^{-j}\right)=-2\cdot0=0&\text{для всех остальных }j.
\end{cases}
\]
Лемма доказана.
\end{proof}

\subsection{Подгруппа $V_1$}

\begin{agreement}
Далее ограничимся рассмотрением \empha{только круговых единиц}.
Более точно, будут рассматриваться только элементы группы $D$.
\end{agreement}

\begin{notations}
Введём следующие обозначения.
\begin{enumerate}
\item
Положим
\[
E=\left\{\lambda\in D\mid\lambda\equiv1\pmod{2}\right\}.
\]
Иными словами,
\[
E=(1+2\Z[\alpha])\cap D.
\]
\item
$V_1=\left\{u_{\chi_1}(\lambda)\in W_1\mid \lambda\in E\right\}$.
\end{enumerate}
\end{notations}

\begin{lemma}\label{l:std}
Для любого $l\in\{0,\dots,2^{n-2}-1\}$
\[
d_{2l+1}^{2^{n-2}}\equiv1\pmod{2},\text{ то есть }d_{2l+1}^{2^{n-2}}\in E,
\]
и для любого $k\in\{0,1,\dots,n-3\}$
\[
d_{2l+1}^{2^k}\not\equiv1\pmod{2}, \text{ то есть }d_{2l+1}^{2^k}\notin E.
\]
\end{lemma}
\begin{proof}
Поскольку для любого $l\in\{0,\dots,2^{n-2}-1\}$ число  $u=d_{2l+1}$  алгебраически сопряжёно с числом
\[
d_1=1+s_1=1+\alpha+\alpha^{-1},
\]
то достаточно рассмотреть $d_1$.
По лемме \ref{l:md2}
\[
d_1^{2^{n-2}}\equiv d_{2^{n-2}}\equiv1\pmod{2}.
\]

Допустим, что для некоторого $k\in\{0,1,\dots,n-3\}$
\[
d_1^{2^k}\equiv d_{2^k}\equiv1\pmod{2}.
\]
Тогда по лемме  \ref{l:md2}
\[
d_1^{2^{n-3}}\equiv d_{2^{n-2}}=1+\sqrt{2}\equiv1\pmod{2}.
\]
Откуда
\[
\frac{\sqrt{2}}{2}\in\Z[\alpha],
\]
что невозможно, ибо $\frac{\sqrt{2}}{2}$ не является целым алгебраическим числом.

Лемма доказана.
\end{proof}

\begin{proposition}\label{p:D2}{\ }
\begin{enumerate}[{\rm 1.}]
\item
Индекс
\[
\left|D:D^{2^{n-2}}\right|=(2^{n-2})^{2^{n-2}-1}=2^{(n-2)(2^{n-2}-1)},
\]
равносильно порядок фактор-группы
\[
\left|D/D^{2^{n-2}}\right|=(2^{n-2})^{2^{n-2}-1}=2^{(n-2)(2^{n-2}-1)}.
\]
\item
$E$ является подгруппой $D$, содержащей $D^{2^{n-2}}$, то есть
\[
D^{2^{n-2}}\leqslant E<D.
\]
\end{enumerate}
\end{proposition}
\begin{proof}{\ }
\begin{enumerate}
\item
Сразу следует из определения группы $D$ и леммы \ref{l:std}.
\item
Из леммы \ref{l:std} получаем, что подгруппа $D^{2^{n-2}}$ содержится в множестве $E$.
Пусть $1+2\lambda$ и $1+2\mu$ --- произвольные элементы $E$, где $\lambda,\mu\in\Z[\alpha]$.
Так как
\[
(1+2\lambda)(1+2\mu)=1+2(\lambda+\mu)+4\lambda\mu\in1+2\Z[\alpha]\cap D=E,
\]
то есть множество $E$ замкнуто относительно умножения.
Так как $E$ состоит из смежных классов, которые являются элементами конечной фактор-группы $D/D^{2^{n-2}}$, то получается конечная подгруппа $E/D^{2^{n-2}}$.
Откуда следует, что $E$ является подгруппой $D$, содержащей $D^{2^{n-2}}$.
\end{enumerate}
Предложение доказано.
\end{proof}

Как непосредственное следствие  предложений \ref{pr:ub1z} и \ref{p:D2} и леммы \ref{l:-1} получим описание строения группы $W_1$.
\begin{corollary}\label{c:W1}{\ }
\begin{enumerate}[{\rm 1.}]
\item
$V_1$ --- подгруппа группы $W_1$.
\item
$W_1=\langle x^{2^{n-1}}\rangle\times V_1$.
\end{enumerate}
\end{corollary}

\section{Подгруппа $F$}

\subsection{Построение воронки}

Так как
\[
D=\prod_{l=0}^{2^{n-2}-1}\left\langle d_{2l+1}\right\rangle,
\]
множеством индексов в этом произведении является
\[
A=\{1,3,5,\dots,2^{n-1}-3\}.
\]
Всего в этом множестве $2^{n-2}-1$ элементов.

\begin{remark}
Построим разбиения (дизъюнктные объединения) множества $A$, которые будут определяться всё уменьшающимися частями множества $A$.
Поэтому этот процесс (для красного словца) назовём \empha{``воронкой''}.
\end{remark}

\begin{notations}\label{n:part}
Обозначим
\[
A_0=\{1,3,5,\dots,2^{n-2}-1\}=\left\{2l+1\mid l\in\{0,\dots,2^{n-2}-1\}\right\}{,}
\]
в этом множестве $2^{n-3}=2^{n-3}$ элементов.
Также обозначим
\[
B_0=A\setminus A_0=\left\{2^{n-1}-(2l+1)\mid 2l+1\in A_0\setminus\{1\}\right\}.
\]
Будем делить множество $A_0$ пополам.

А именно, для любого $k\in\{1,\dots,n-3\}$ положим
\[
A_k=\{1,3,5,\dots,2^{n-2-k}-1\}=\left\{2l+1\mid l\in\{0,\dots,2^{n-3-k}-1\}\right\},
\]
в этом множестве $2^{n-3-k}$ элементов.
Также будем иметь
\[
B_k=A_{k-1}\setminus A_k=\left\{2^{n-1-k}-(2l+1)\mid 2l+1\in A_k\right\}.
\]
\end{notations}

\begin{examples}
Рассмотрим последние члены семейства $\left\{A_k\right\}_{k=0}^{n-3}$.
\begin{description}
\item[$A_{n-3}.$]
Имеем
\begin{align*}
A_{n-3}&=\{1,\dots,2^{n-2-n-3}-1\}=\{1,2^1-1\}=\\
&=\left\{2l+1\mid l\in\{0,2^0-1\}\right\}=\{1\}.
\end{align*}
\item[$A_{n-4}.$]
Теперь
\begin{align*}
A_{n-4}&=\{1,\dots,2^{n-2-n-4}-1\}=\{1,\dots,2^2-1\}=\\
&=\left\{2l+1\mid l\in\{0,\dots,2^1-1\}\right\}=\{1,3\}.
\end{align*}
Отсюда
\[
B_{n-3}=A_{n-4}\setminus A_{n-3}=\{3\}.
\]
\item[$A_{n-5}.$]
Теперь
\begin{align*}
A_{n-5}&=\{1,\dots,2^{n-2-n-5}-1\}=\{1,\dots,2^3-1\}=\\
&=\left\{2l+1\mid l\in\{0,\dots,2^2-1\}\right\}=\{1,3,5,7\}.
\end{align*}
Отсюда
\[
B_{n-4}=A_{n-5}\setminus A_{n-4}=\{5,7\}.
\]
\end{description}
\end{examples}

\begin{lemma}\label{l:part}
При таких обозначениях  для любого $k\in\{0,1,\dots,n-4\}$ получается разбиение
\[
A=A_k\cup B_k\cup B_{k-1}\cup\dots\cup B_0,
\]
причём $A_{n-4}=\{1,3\}$ и при $n=4\longleftrightarrow2^n=16$ имеем $B_{n-4}=\{5\}$, а при $n\geqslant5\longleftrightarrow2^n\geqslant32$ имеем $B_{n-4}=\{5,7\}$.
\end{lemma}
\begin{proof}
По определению подмножеств $A_0$ и $B_0$ получим
\[
A=A_0\cup B_0\text{ и }A_0\cap B_0=\varnothing,
\]
то есть получили разбиение множества $A$.
Это будет нулевой шаг.

Первые $2^{n-4}$ элементов множества $A_0$ образуют подмножество
\[
A_1=\{1,3,5,\dots,2^{n-3}-1\}.
\]
Так как
\[
B_1=A_0\setminus A_1=\left\{2^{n-2}-(2l+1)\mid 2l+1\in A_1\right\},
\]
то
\[
A_0=A_1\cup B_1\text{ и }A_1\cap B_0=\varnothing,
\]
то есть получили разбиение множества $A_0$, а это даёт разбиение множества $A$
\[
A=A_0\cup B_0=A_1\cup B_1\cup B_0.
\]
Это будет первый шаг индукции.

На шаге индукции с номером $k-1$ получили множество
\[
A_{k-1}=\{1,3,5,\dots,2^{n-2-(k-1)}-1\},
\]
в котором $2^{n-3-(k-1)}$ элементов, и разбиение множества $A$
\[
A=A_{k-1}\cup B_{k-1}\cup\dots\cup B_0.
\]
Теперь сделаем шаг индукции с номером $k$.
Первые $2^{n-3-k}$ элементов множества  $A_0$ образуют подмножество
\[
A_k=\{1,3,5,\dots,2^{n-2-k}-1\}\text{ и }B_k=A_{k-1}\setminus A_k.
\]
Так как получили разбиение
\[
A_{k-1}=A_k\cup B_k,
\]
что даёт разбиение множества $A$
\[
A=A_k\cup B_k\cup B_{k-1}\cup\dots\cup B_0.
\]
Это будет шаг с номером $k$.

При $n=4\longleftrightarrow2^n=16$ имеем
\[
A=\{1,3,5\},\ A_0=\{1,3\}\text{ и }B_0=\{5\}.
\]
При $n\geqslant5\longleftrightarrow2^n\geqslant32$ имеем
\[
A_{n-5}=\{1,3,\dots,2^{n-2-(n-5)}-1\}=\{1,3,\dots,2^3-1\}=\{1,3,5,7\}.
\]
Поэтому имеем
\[
A_{n-4}=\{1,3\}\text{ и }B_{n-4}=\{5,7\}.
\]
Это будет последний шаг.

Лемма доказана.
\end{proof}

\subsubsection{Примеры  воронок для $2^n\in\{16,32,64,128\}$ $\longleftrightarrow$ $m\in\{1,2,4,8\}$}\label{sss:funcas}

\begin{description}
\item[Воронка для $16$]
\end{description}

Изобразим этот процесс таблично
\[
A=\begin{array}{|c|c|c|c|}\hline
l&0&1&2\\\hline
2l+1&1&3&5\\\hline
\end{array}
\]
\begin{description}
\item[Шаг $0$.]
Имеем
\begin{align*}
A_0=\{1, 3\}, \\
B_0=A\setminus A_0=\{5\}.
\end{align*}
Получим разбиение
\[
A=\{1,3\}\cup\{5\}.
\]
\end{description}

\begin{description}
\item[Воронка для $32$]
\end{description}

Изобразим этот процесс таблично
\[
A=\begin{array}{|c|c|c|c|c|c|c|c|}\hline
l&0&1&2&3&4&5&6\\\hline
2l+1&1&3&5&7&9&11&13\\\hline
\end{array}
\]
\begin{description}
\item[Шаг $0$.]
Имеем
\begin{align*}
A_0=\{1, 3, 5, 7\}, \\
B_0=A\setminus A_0=\{9, 11, 13\}.
\end{align*}
Получим разбиение
\[
A=\{1,3,5,7\}\cup\{9,11,13\}.
\]
\item[Шаг $1$.]
Теперь
\begin{align*}
A_1=\{1,3\}, \;\; B_1=\{5,7\}.
\end{align*}
Получили итоговое разбиение
\[
A=\{1,3\}\cup\{5,7\}\cup\{9,11,13\}.
\]
\end{description}

\begin{description}
\item[Воронка для $64$]
\end{description}

Изобразим этот процесс таблично
\[
A=\begin{array}{|c|c|c|c|c|c|c|c|c|c|c|c|c|c|c|c|}\hline
l&0&1&2&3&4&5&6&7&8&9&10&11&12&13&14\\\hline
2l+1&1&3&5&7&9&11&13&15&17&19&21&23&25&27&29\\\hline
\end{array}
\]
\begin{description}
\item[Шаг $0$.]
Имеем
\begin{align*}
A_0&=\{1, 3, 5, 7, 9, 11, 13, 15\}, \\
B_0&=A\setminus A_0=\{17, 19, 21, 23, 25, 27, 29\}.
\end{align*}
Получим разбиение
\[
A=\{1,3,5,7,9,11,13,15\}\cup\{17,19,21,23,25,27,29\}.
\]
\item[Шаг $1$.]
Теперь
\begin{align*}
A_1=\{1,3,5,7\}, \;\; B_1=\{9,11,13,15\}.
\end{align*}
Получили разбиение
\[
A=\{1,3,5,7\}\cup\{9,11,13,15\}\cup\{17,19,21,23,25,27,29\}.
\]
\item[Шаг $2$.]
Наконец,
\[
A_2=\{1,3\}, B_2=\{5,7\}.
\]
Получили последнее разбиение
\[
A=\{1,3\}\cup\{5,7\}\cup\{9,11,13,15\}\cup\{17,19,21,23,25,27,29\}.
\]
\end{description}

\begin{description}\label{pr:128}
\item[Воронка для $128$]
\end{description}

Изобразим этот процесс таблично
\[
A=\begin{array}{|c|c|c|c|c|c|c|c|c|}\hline
l&0&1&2&3&\dots&28&29&30\\\hline
2l+1&1&3&5&7&\dots&57&59&61\\\hline
\end{array}
\]
\begin{description}
\item[Шаг $0$.]
\end{description}
Имеем
\[
\hspace*{-17pt}
\begin{aligned}
A_0&=\begin{array}{|c|c|c|c|c|c|c|c|c|c|c|c|c|c|c|c|c|}\hline
l&0&1&2&3&4&5&6&7&8&9&10&11&12&13&14&15\\\hline
2l+1&1&3&5&7&9&11&13&15&17&19&21&24&25&27&29&31\\\hline
\end{array},\\
B_0&=\begin{array}{|c|c|c|c|c|c|c|c|c|c|c|c|c|c|c|c|}\hline
l&16&17&18&19&20&21&22&23&24&25&26&27&28&29&30\\\hline
2l+1&33&35&37&39&41&43&45&47&49&51&53&55&57&59&61\\\hline
\end{array}.
\end{aligned}
\]
Получим разбиение
\[
A=\{1,3,\dots,29,31\}\cup\{31,33,\dots,59,61\}.
\]
\begin{description}
\item[Шаг $1$.]
Теперь
\begin{align*}
A_1&=\begin{array}{|c|c|c|c|c|c|c|c|c|}\hline
l&0&1&2&3&4&5&6&7\\\hline
2l+1&1&3&5&7&9&11&13&15\\\hline
\end{array},\\
B_1&=\begin{array}{|c|c|c|c|c|c|c|c|c|}\hline
l&8&9&10&11&12&13&14&15\\\hline
2l+1&17&19&21&23&25&27&29&31\\\hline
\end{array}.
\end{align*}
На этом шаге получим разбиение
\[
A=\{1,3,\dots,13,15\}\cup\{17,19,\dots,29,31\}\cup\{33,35,\dots,59,61\}.
\]
\item[Шаг $2$.]
Получим
\[
A_2=\begin{array}{|c|c|c|c|c|}\hline
l&0&1&2&3\\\hline
2l+1&1&3&5&7\\\hline
\end{array},
B_2=\begin{array}{|c|c|c|c|c|}\hline
l&4&5&6&7\\\hline
2l+1&9&11&13&15\\\hline
\end{array}.
\]
Отсюда возникает разбиение
\begin{align*}
A=\{1,3,5,7\}&\cup\{9,11,13,15\}\cup\{17,19,\dots,29,31\}\cup\\
&\cup\{33,35,\dots,59,61\}.
\end{align*}
\item[Шаг $3$.]
Наконец,
\[
A_3=\begin{array}{|c|c|c|}\hline
l&0&1\\\hline
2l+1&1&3\\\hline
\end{array},
B_3=\begin{array}{|c|c|c|}\hline
l&2&3\\\hline
2l+1&5&7\\\hline
\end{array},
\]
и всё закончено.
Получили последнее разбиение
\begin{align*}
A=\{1,3\}&\cup\{5,7\}\cup\{9,11,13,15\}\cup\{17,19,\dots,31\}\cup\\
&\cup\{33,35,\dots,59,61\}.
\end{align*}
\end{description}

\subsection{Сравнимость элементов воронки}

Используя построенные в лемме \ref{l:part} разбиения, изучим сравнимость по модулю $2$ степеней порождающих группы $D$.

\begin{lemma}\label{l:m&D}
Для любых $l,r\in\{0,\dots,2^{n-2}-1\}$ имеем
\[
d_{2l+1}^{2^{n-3}}d_{2r+1}^{2^{n-3}}\in E.
\]
В частности, $d_1^{-2^{n-3}}d_3^{2^{n-3}}\in E$.
\end{lemma}
\begin{proof}
По лемме  \ref{l:sdr} для любого $l\in\{0,1,\dots,2^{n-2}-1\}$
\[
d_{2l+1}^{2^{n-3}}\equiv d_{2^{n-3}(2l+1)}\pmod{2}.
\]
Из леммы \ref{l:md2} следует, что
\[
d_{2^{n-3}(2l+1)}\equiv d_{2^{n-3}}\equiv1+\sqrt{2}\pmod{2}.
\]
Поэтому лемма \ref{l:md2} даёт
\[
d_{2l+1}^{2^{n-3}}d_{2r+1}^{2^{n-3}}\equiv d_{2^{n-3}}^2\equiv d_{2^{n-2}}\equiv1\pmod{2}.
\]
В частности по лемме \ref{l:std} получим, что
\[
d_1^{-2^{n-3}}d_3^{2^{n-3}}=d_1^{-2^{n-2}}d_1^{2^{n-3}}d_3^{2^{n-3}}\in E.
\]
Лемма доказана.
\end{proof}

\begin{notations}
Введём обозначения.

Для любого $2l+1\in\left\{3,\dots,2^{n-2}-1=2^{n-2}-1\right\}=A_0\setminus\{1\}$ положим
\[
q(0,2l+1)=d_{2l+1}^{-1}d_{2^{n-1}-(2l+1)}.
\]
Пусть $k\in\{1,\dots,n-3\}$.
Для любого $2l+1\in\left\{1,\dots,2^{n-2-k}-1\right\}=A_k$ положим
\[
q(k,2l+1)=d_{2l+1}^{-1}d_{2^{n-1-k}-(2l+1)}.
\]
В частности,
для  $k=n-3$ имеем
\[
A_{n-3}=\{1\}
\]
и
\[
q(n-3,1)=d_1^{-1}d_3;
\]
и для  $k=n-4$ имеем
\[
A_{n-4}=\{1,3\}
\]
и
\[
q(n-4,1)=d_1^{-1}d_{8-1}=d_1^{-1}d_7\text{ и }q(n-4,3)=d_3^{-1}d_{8-3}=d_3^{-1}d_5.
\]
\end{notations}

\begin{lemma}\label{l:m(k)}
\begin{align*}
D&=\langle d_1\rangle\times\langle d_1^{-1}d_3\rangle\times\prod_{k=0}^{n-4}\prod_{2^{n-1-k}-(2l+1)\in B_k}\langle q(k,2l+1)\rangle=\\
&=\langle d_1\rangle\times\prod_{2l+1\in A_0\setminus\{1\}}\langle q(0,2l+1)\rangle\times\prod_{k=1}^{n-3}\prod_{2l+1\in A_k}\langle q(k,2l+1)\rangle.
\end{align*}
\end{lemma}
\begin{proof}
Будем рассматривать по шагам.

На шаге $0$ для любого $2l+1\in\left\{3,\dots,2^{n-2}-1\right\}=A_0\setminus\{1\}$ получим
\[
d_{2^{n-1}-(2l+1)}=d_{2l+1}q(0,2l+1).
\]
Поэтому
\begin{align*}
D&=\prod_{2l+1\in A_0}\langle d_{2l+1}\rangle\times\prod_{2^{n-1}-(2l+1)\in B_0}\langle d_{2^{n-1}-(2l+1)}\rangle=\\
&=\prod_{2l+1\in A_0}\langle d_{2l+1}\rangle\times\prod_{2^{n-1}-(2l+1)\in B_0}\langle d_{2l+1}q(0,2l+1)\rangle=\\
&=\prod_{2l+1\in A_0}\langle d_{2l+1}\rangle\times\prod_{2^{n-1}-(2l+1)\in B_0}\langle q(0,2l+1)\rangle.
\end{align*}
Положим
\[
D_0=\prod_{2l+1\in A_0}\langle d_{2l+1}\rangle.
\]
Далее надо рассматривать $D_0$.

Для любого $2l+1\in\left\{1,\dots,2^{n-2}-1\right\}=A_1$ получим
\[
d_{2^{n-2}-(2l+1)}=d_{2l+1}q(1,2l+1).
\]
Поэтому
\begin{align*}
D_0&=\prod_{2l+1\in A_1}\langle d_{2l+1}\rangle\times\prod_{2^{n-2}-(2l+1)\in B_1}\langle d_{2^{n-2}-(2l+1)}\rangle=\\
&=\prod_{2l+1\in A_1}\langle d_{2l+1}\rangle\times\prod_{2^{n-2}-(2l+1)\in B_1}\langle d_{2l+1}q(1,2l+1)\rangle=\\
&=\prod_{2l+1\in A_1}\langle d_{2l+1}\rangle\times\prod_{2^{n-2}-(2l+1)\in B_1}\langle q(1,2l+1)\rangle.
\end{align*}
Положим
\[
D_1=\prod_{2l+1\in A_1}\langle d_{2l+1}\rangle.
\]
Далее надо рассматривать $D_1$.

Предположим, что сделан шаг $k\in\{0,1,\dots,n-5\}$.
Тогда получим, что нужно рассматривать только
\[
D_k=\prod_{2l+1\in A_k}\langle d_{2l+1}\rangle.
\]
Поэтому
\begin{align*}
D_k&=\prod_{2l+1\in A_{k+1}}\langle d_{2l+1}\rangle\times\prod_{2^{n-2-k}-(2l+1)\in B_{k+1}}\langle d_{2^{n-2-k}-(2l+1)}\rangle=\\
&=\prod_{2l+1\in A_{k+1}}\langle d_{2l+1}\rangle\times\prod_{2^{n-2-k}-(2l+1)\in B_{k+1}}\langle d_{2l+1}q(k+1,2l+1)\rangle=\\
&=\prod_{2l+1\in A_{k+1}}\langle d_{2l+1}\rangle\times\prod_{2^{n-2-k}-(2l+1)\in B_{k+1}}\langle q(k+1,2l+1)\rangle.
\end{align*}
Положим
\[
D_{k+1}=\prod_{2l+1\in A_{k+1}}\langle d_{2l+1}\rangle.
\]
Далее надо рассматривать $D_{k+1}$.

Для $n\geqslant5$ рассмотрим, что получится после шага $n-4$.
В самом деле, имеем
\begin{align*}
D_{n-4}&=\langle d_1\rangle\times \langle d_3\rangle\times\langle q(n-4,5)\rangle\times\langle q(n-4,7)\rangle.
\end{align*}
Сделаем завершающий шаг
\[
\langle d_1\rangle\times\langle d_3\rangle=\langle d_1\rangle\times\langle d_1(d_1^{-1}d_3)\rangle=\langle d_1\rangle\times\langle d_1^{-1}d_3\rangle=\langle d_1\rangle\times\langle q(n-3,1)\rangle.
\]

Лемма доказана.
\end{proof}

\begin{lemma}\label{l:fun}
На шаге $0$ воронки имеем
\[
d_{2^{n-1}-(2l+1)}\equiv d_{2l+1}\pmod{2}\longleftrightarrow q(0,2l+1)=d_{2l+1}^{-1}d_{2^{n-1}-(2l+1)}\in E
\]
для любого $2l+1\in\{3,5,\dots,32^{n-3}-1\}=A_0\setminus\{1\}$.

На шаге $k\in\{1,2,\dots,n-4\}$
\begin{align*}
d_{2^{n-1-k}-(2l+1)}^{2^k}&\equiv d_{2l+1}^{2^k}\pmod{2}\longleftrightarrow\\
&\longleftrightarrow q(k,2l+1)^{2^k}=d_{2l+1}^{-2^k}d_{2^{n-1-k}-(2l+1)}^{2^k}\in E\text{ и}\\
d_{2^{n-1-k}-(2l+1)}^{2^{k-1}}&\not\equiv d_{2l+1}^{2^{k-1}}\pmod{2}
\end{align*}
для любого $2l+1\in\{1,3,\dots,2^{n-2-k}-1\}=A_k$.
\end{lemma}
\begin{proof}
На нулевом шаге по лемме \ref{l:md2} (утверждение 2)  получаем
\[
d_{2^{n-1}-(2l+1)}\equiv d_{2l+1}\pmod{2}
\]
для любого $2l+1\in\{3,5,\dots,2^{n-2}-1\}$.

Посмотрим, что будет на шаге $k\in\{1,2,\dots,n-4\}$.
По леммам \ref{l:sdr} и \ref{l:md2} получим, что для любого $2l+1\in\{1,3,\dots,2^{n-2-k}-1\}=A_k$
\begin{align*}
d_{2l+1}^{2^k}&\equiv d_{2^k(2l+1)}\pmod{2}\text{ и}\\
d_{2^{n-1-k}-(2l+1)}^{2^k}&\equiv d_{2^{n-1}-2^k(2l+1)}\equiv d_{2^k(2l+1)}\pmod{2}.
\end{align*}

Предположим, что
\[
d_{2l+1}^{2^{k-1}}\equiv d_{2^{n-1-k}-(2l+1)}^{2^{k-1}}\pmod{2},
\]
то есть, снова по лемме \ref{l:sdr} (утверждение 2)
\[
1+s_{2^{k-1}(2l+1)}\equiv1+s_{2^{n-2}-2^{k-1}(2l+1)}\pmod{2}.
\]
Следовательно
\begin{align*}
s_{2^{k-1}(2l+1)}&\equiv s_{2^{n-2}-2^{k-1}(2l+1)}=r_{2^{k-1}(2l+1)}-s_{2^{k-1}(2l+1)}\equiv\\
&\equiv r_{2^{k-1}(2l+1)}+s_{2^{k-1}(2l+1)}\pmod{2}.
\end{align*}
Поэтому
\[
r_{2^{k-1}(2l+1)}\equiv0\pmod{2},
\]
что невозможно по лемме \ref{l:bsp}, ибо $r_{2^{k-1}(2l+1)}$ --- элемент целого базиса $\vec{B}$.

Лемма доказана.
\end{proof}

\subsection{Группа $F$}

\begin{notations}
Введём обозначения.

Для шага $0$ положим
\[
F_0=\prod_{2l+1\in A_0\setminus\{1\}}\langle d_{2l+1}^{-1}d_{2^{n-1}-(2l+1)}\rangle=\prod_{2l+1\in A_0\setminus\{1\}}\langle q(0,2l+1)\rangle.
\]

На шаге $k\in\{1,2,\dots,n-3\}$ положим
\[
F_k=\prod_{2l+1\in A_k}\langle d_{2l+1}^{-2^k}d_{2^{n-2}-(2l+1)}^{2^k}\rangle=\prod_{2l+1\in A_k}\langle q(k,2l+1)^{2^k}\rangle.
\]

Наконец,
\[
F=\langle d_1^{2^{n-2}}\rangle\times\prod_{k=0}^{n-3}F_k.
\]
\end{notations}

\begin{lemma}\label{l:F}
$F$ --- подгруппа в $E$.
Кроме того,
\[
D^{2^{n-2}}<F\leqslant E<D.
\]
\end{lemma}
\begin{proof}
Утверждение следует из лемм \ref{l:std}, \ref{l:m&D} и \ref{l:fun}.

По лемме \ref{l:m(k)} имеем
\[
D^{2^{n-2}}=\langle d_1^{2^{n-2}}\rangle\times\prod_{2l+1\in A_0\setminus\{1\}}\langle q(0,2l+1)^{2^{n-2}}\rangle\times\prod_{k=1}^{n-3}\prod_{2l+1\in A_k}\langle q(k,2l+1)^{2^{n-2}}\rangle.
\]
Поэтому $F>D^{2^{n-2}}$ и дополнительное утверждение следует из предложения \ref{p:D2}.

Лемма доказана.
\end{proof}

\section{Группа $\sqrt{F}$}

\subsection{Основная лемма о $\sqrt{F}$}

\begin{notation}
Обозначим
\[
\sqrt{F}=\left\{d\in D\mid d^2\in F\right\}.
\]
\end{notation}

\begin{lemma}\label{l:sqrt(F)}
$\sqrt{F}$ --- подгруппа группы $D$, имеющая следующие свойства.
\begin{enumerate}[{\rm1.}]
\item
Фактор-группа
\[
\sqrt{F}/F=\langle d_1^{2^{n-3}}F\rangle\times\prod_{k=1}^{n-3}\prod_{2l+1\in A_k}\langle q(k,2l+1)^{2^{k-1}}F\rangle
\]
и является элементарной абелевой $2$-группой, возможно, единичной.
\item
$D^{2^{n-2}}<F\leqslant\sqrt{F}\cap E\leqslant\sqrt{F}<D$.
\end{enumerate}
\end{lemma}
\begin{proof}
То, что $\sqrt{F}$ является подгруппой группы $D$, очевидно, так как $D$ абелева.

Так как $q(0,2l+1)\in F$ для любого числа $2l+1\in A_0\setminus\{1\}$, то надо рассматривать $k\geqslant1$.
Дополнительные утверждения непосредственно следуют из определения подгруппы $F$ и леммы \ref{l:F}.

Лемма доказана.
\end{proof}

\begin{corollary}\label{c:sqrF}
Пусть $2^n\in\{16,32,64,128\}\longleftrightarrow m\in\{1,2,4,8\}$. 
Тогда фактор-группа $\sqrt{F}/F$ имеет следующее строение.
\begin{description}
\item[Для $16$.]
\[
\sqrt{F}/F=\langle d_1^2\rangle F\times\langle q(1,1)\rangle F.
\] 
\item[Для $32$.]
\[
\sqrt{F}/F=\langle d_1^4\rangle F\times\langle q(2,1)^2\rangle F\times\langle q(1,1)\rangle F\times\langle q(1,3)\rangle F.
\]
\item[Для $64$.]
\begin{align*}
\sqrt{F}/F&=\langle d_1^8\rangle F\times\langle q(3,1)^4\rangle F\times\langle q(2,1)^2\rangle F\times\langle q(2,3)^2\rangle F\times\\
&\times\langle q(1,1)\rangle F\times\langle q(1,3)\rangle F\times\langle q(1,5)\rangle F\times\langle q(1,7)\rangle F.
\end{align*}
\item[Для $128$.]
\begin{align*}
\sqrt{F}/F&=\langle d_1^{16}\rangle F\times\langle q(4,1)^8\rangle F\times\langle q(3,1)^4\rangle F\times\langle q(3,3)^4\rangle F\times\\
&\times\langle q(2,1)^2\rangle F\times\langle q(2,3)^2\rangle F\times\langle q(2,5)^2\rangle F\times\langle q(2,7)^2\rangle F\times\\
&\times\langle q(1,1)\rangle F\times\langle q(1,3)\rangle F\times\langle q(1,5)\rangle F\times\langle q(1,7)\rangle F\times\\
&\times\langle q(1,9)\rangle F\times\langle q(1,11)\rangle F\times\langle q(1,13)\rangle F\times\langle q(1,15)\rangle F.
\end{align*}
\end{description}
\end{corollary}
\begin{proof}{\ }
\begin{description}
\item[Группа $\sqrt{F}/F$ для $16$.]
\end{description}
\begin{description}
\item[$(k=1).$]
\begin{align*}
A_1&=\{1\}.\\
q(1,1)^2&=d_1^{-2}d^2_3, \text{(по лемме 34)}\\
F_1&=\langle q(1,1)^{2}\rangle.
\end{align*}
\item[$(k=0).$]
\begin{align*}
A_0\setminus\{1\}&=\{3\},\\
q(0,3)&=d_3^{-1}d_{8-3}=d_3^{-1}d_5,\\
F_0&=\langle q(0,3)\rangle.
\end{align*}
\end{description}

По лемме \ref{l:m(k)}
\begin{align*}
D&=\langle d_1\rangle\times\prod_{2l+1\in A_0\setminus\{1\}}\left\langle q(0,2l+1)\right\rangle\times
\prod_{k=1}^{4-3}\prod_{2l+1\in A_k}\left\langle q(k,2l+1)\right\rangle=\\
&=\langle d_1\rangle\times\langle q(0,3)\rangle\times\langle q(1,1)\rangle;\\
F&=\langle d_1^{4}\rangle\times\prod_{k=0}^{4-3}F_k=\langle d_1^{4}\rangle\times F_0\times F_1=\\
&=\langle d_1^{4}\rangle\times\langle q(0,3)\rangle\times\langle q(1,1)^2\rangle;
\intertext{фактор-группа}
&\qquad\qquad\sqrt{F}/F=\langle d_1^{2}\rangle F\times\langle q(1,1)\rangle F.
\end{align*}

\begin{description}
\item[Группа $\sqrt{F}/F$ для $32$.]
\end{description}
\begin{description}
\item[$(k=2).$]
\begin{align*}
A_2&=\{1\}.\\
q(2,1)^4&=d_1^{-4}d^4_3,\\
F_2&=\langle q(2,1)^{4}\rangle.
\end{align*}
\item[$(k=1).$]
\begin{align*}
A_1&=\{1,3\}.\\
q(1,1)^2&=d_1^{-2}d^2_7 \;\; \text{ и } \;\; q(1,3)^2=d_3^{-2}d^2_5,\\
F_1&=\langle q(1,1)^{2}\rangle\times\langle q(1,3)^{2}\rangle.
\end{align*}
\item[$(k=0).$]
\begin{align*}
A_0\setminus\{1\}&=\{3, 5, 7\},\\
q(0,3)&=d_3^{-1}d_{16-3}=d_3^{-1}d_{13},\\
q(0,5)&=d_5^{-1}d_{11},\\
q(0,7)&=d_7^{-1}d_9,\\
F_0&=\langle q(0,3)\rangle\times\langle q(0,5)\rangle\times\langle q(0,7)\rangle.
\end{align*}
\end{description}

По лемме \ref{l:m(k)}
\begin{align*}
D&=\langle d_1\rangle\times\prod_{2l+1\in A_0\setminus\{1\}}\left\langle q(0,2l+1)\right\rangle\times
\prod_{k=1}^{5-3}\prod_{2l+1\in A_k}\left\langle q(k,2l+1)\right\rangle=\\
&=\langle d_1\rangle\times\langle q(0,3)\rangle\times\langle q(0,5)\rangle\times\langle q(0,7)\rangle\times\\
&\times\langle q(1,1)\rangle\times\langle q(1,3)\rangle\times\langle q(2,1)\rangle;\\
F&=\langle d_1^8\rangle\times\prod_{k=0}^{5-3}F_k=\langle d_1^8\rangle\times F_0\times F_1\times F_2=\\
&=\langle d_1^8\rangle\times\langle q(0,3)\rangle\times\langle q(0,5)\rangle\times\langle q(0,7)\rangle\times\\
&\times\langle q(1,1)^2\rangle\times\langle q(1,3)^2\rangle\times\langle q(2,1)^4\rangle;
\intertext{фактор-группа}
&\qquad\sqrt{F}/F=\langle d_1^4\rangle F\times\langle q(1,1)\rangle F\times\langle q(1,3)\rangle F\times\langle q(2,1)^2\rangle F.
\end{align*}

\begin{description}
\item[Группа $\sqrt{F}/F$ для $64$.]
\end{description}
\begin{description}
\item[$(k=3).$]
\begin{align*}
A_3&=\{1\}.\\
q(3,1)^8&=d_1^{-8}d_3^8,\\
F_3&=\langle q(3,1)^{8}\rangle.
\end{align*}
\item[$(k=2).$]
\begin{align*}
A_2&=\{1,3\}.\\
q(2,1)^4&=d_1^{-4}d^4_7 \; \text{ и }\; q(2,3)^4=d_3^{-4}d^4_5,\\
F_2&=\langle q(2,1)^{4}\rangle\times\langle q(2,3)^{4}\rangle.
\end{align*}
\item[$(k=1).$]
\begin{align*}
A_1&=\{1,3,5,7\}.\\
q(1,1)^2&=d_1^{-2}d^2_{15}, \\
q(1,3)^2&=d_3^{-2}d^2_{13},\\
q(1,5)^2&=d_5^{-2}d^2_{11},\\
q(1,7)^2&=d_7^{-2}d^2_9,\\
F_1=\langle q(1,1)^{2}\rangle&\times\langle q(1,3)^{2}\rangle\times\langle q(1,5)^{2}\rangle\times\langle q(1,7)^{2}\rangle.
\end{align*}
\item[$(k=0).$]
\[
A_0\setminus\{1\}=\{3, 5, 7, 9, 11, 13, 15\},
\]
\begin{align*}
q(0,3)&=d_3^{-1}d_{29},&q(0,5)&=d_5^{-1}d_{27},\\
q(0,7)&=d_7^{-1}d_{25},&q(0,9)&=d_9^{-1}d_{23},\\
q(0,11)&=d_{11}^{-1}d_{21},&q(0,13)&=d_{13}^{-1}d_{19},\\
q(0,15)&=d_{15}^{-1}d_{17},
\end{align*}
\begin{align*}
F_0=\langle q(0,3)\rangle&\times\langle q(0,5)\rangle\times\langle q(0,7)\rangle\times\langle q(0,9)\rangle\times\\
&\times\langle q(0,11)\rangle\times\langle q(0,13)\rangle\times\langle q(0,15)\rangle.
\end{align*}
\end{description}

По лемме \ref{l:m(k)}
\begin{align*}
D&=\langle d_1\rangle\times\prod_{2l+1\in A_0\setminus\{1\}}\left\langle q(0,2l+1)\right\rangle\times
\prod_{k=1}^{6-3}\prod_{2l+1\in A_k}\left\langle q(k,2l+1)\right\rangle=\\
&=\langle d_1\rangle\times\langle q(0,3)\rangle\times\langle q(0,5)\rangle\times\langle q(0,7)\rangle\times\langle q(0,9)\rangle\times\langle q(0,11)\rangle\times\\
&\times\langle q(0,13)\rangle\times\langle q(0,15)\rangle\times\langle q(1,1)\rangle\times\langle q(1,3)\rangle\times\langle q(1,5)\rangle\times\\
&\times\langle q(1,7)\rangle\times\langle q(2,1)\rangle\times\langle q(2,3)\rangle\times\langle q(3,1)\rangle;\\
F&=\langle d_1^{16}\rangle\times\prod_{k=0}^{6-3}F_k=\langle d_1^{16}\rangle\times F_0\times F_1\times F_2\times F_3=\\
&=\langle d_1^{16}\rangle\times\langle q(0,3)\rangle\times\langle q(0,5)\rangle\times\langle q(0,7)\rangle\times\langle q(0,9)\rangle\times\langle q(0,11)\rangle\times\\
&\times\langle q(0,13)\rangle\times\langle q(0,15)\rangle\times\langle q(1,1)^2\rangle\times\langle q(1,3)^2\rangle\times\langle q(1,5)^2\rangle\times\\
&\times\langle q(1,7)^2\rangle\times\langle q(2,1)^4\rangle\times\langle q(2,3)^4\rangle\times\langle q(3,1)^8\rangle;
\intertext{фактор-группа}
&\sqrt{F}/F=\langle d_1^8\rangle F\times\langle q(1,1)\rangle F\times\langle q(1,3)\rangle F\times\langle q(1,5)\rangle F\times\\
&\times\langle q(1,7)\rangle F\times\langle q(2,1)^2\rangle F\times\\
&\times\langle q(2,3)^2\rangle F\times\langle q(3,1)^4\rangle F.
\end{align*}

\begin{description}
\item[Группа $\sqrt{F}/F$ для $128$.]
\end{description}

Для любого $2l+1\in\left\{3,\dots,31\right\}=A_0\setminus\{1\}$ имеем
\[
q(0,2l+1)=d_{2l+1}^{-1}d_{64-(2l+1)}.
\]
Пусть $k\in\{1,\dots,4\}$.
Для любого $2l+1\in\left\{1,\dots,2^{5-k}-1\right\}=A_k$ имеем
\[
q(k,2l+1)=d_{2l+1}^{-1}d_{2^{6-k}-(2l+1)}.
\]
\begin{description}
\item[$(k=4).$]
\begin{align*}
A_4&=\{1\},\\
q(4,1)&=d_1^{-1}d_3,\\
F_4&=\langle q(4,1)^{16}\rangle.
\end{align*}
\item[$(k=3).$]
\begin{align*}
A_3&=\{1,3\},\\
q(3,1)&=d_1^{-1}d_{8-1}=d_1^{-1}d_7\text{ и }q(3,3)=d_3^{-1}d_{8-3}=d_3^{-1}d_5,\\
F_3&=\langle q(3,1)^8\rangle\times\langle q(3,3)^8\rangle.
\end{align*}
\item[$(k=2).$]
\begin{align*}
A_2&=\{1,3,5,7\},\\
q(2,1)&=d_1^{-1}d_{16-1}=d_1^{-1}d_{15},& q(2,3)&=d_3^{-1}d_{16-3}=d_3^{-1}d_{13},\\
q(2,5)&=d_5^{-1}d_{16-5}=d_5^{-1}d_{11},& q(2,7)&=d_7^{-1}d_{16-7}=d_7^{-1}d_9;\\
\end{align*}
\[
F_2=\langle q(2,1)^4\rangle\times\langle q(2,3)^4\rangle\times\langle q(2,5)^4\rangle\times\langle q(2,7)^4\rangle.
\]
\item[$(k=1).$]
\[
A_1=\{1,3,5,7,9,11,13,15\}
\]
и
\begin{align*}
q(1,1)&=d_1^{-1}d_{32-1}=d_1^{-1}d_{31},& q(1,3)&=d_3^{-1}d_{32-3}=d_3^{-1}d_{29},\\
q(1,5)&=d_5^{-1}d_{32-5}=d_5^{-1}d_{27},& q(1,7)&=d_7^{-1}d_{32-7}=d_7^{-1}d_{25},\\
q(1,9)&=d_9^{-1}d_{32-9}=d_9^{-1}d_{23},& q(1,11)&=d_{11}^{-1}d_{32-11}=d_{11}^{-1}d_{21},\\
q(1,13)&=d_{13}^{-1}d_{32-13}=d_{13}^{-1}d_{19},& q(1,15)&=d_{15}^{-1}d_{32-15}=d_{15}^{-1}d_{17};
\end{align*}
\[
F_1=\prod_{l=0}^7\langle q(1,2l+1)^2\rangle.
\]

\item[$(k=0).$]
Наконец,
\[
A_0\setminus\{1\}=\{3,5,7,9,11,13,15,17,19,21,23,25,27,29,31\}
\]
и
\begin{align*}
q(0,3)&=d_3^{-1}d_{64-3}=d_3^{-1}d_{61},& q(0,5)&=d_5^{-1}d_{64-5}=d_3^{-1}d_{59},\\
q(0,7)&=d_7^{-1}d_{64-7}=d_7^{-1}d_{57},& q(0,9)&=d_9^{-1}d_{64-9}=d_7^{-1}d_{55},\\
q(0,11)&=d_{11}^{-1}d_{64-11}=d_9^{-1}d_{53},& q(0,13)&=d_{13}^{-1}d_{64-13}=d_3^{-1}d_{51},\\
q(0,15)&=d_{15}^{-1}d_{64-15}=d_{13}^{-1}d_{49},& q(0,17)&=d_{17}^{-1}d_{64-17}=d_7^{-1}d_{47},\\
q(0,19)&=d_{19}^{-1}d_{64-19}=d_{19}^{-1}d_{45},& q(0,21)&=d_{21}^{-1}d_{64-21}=d_{21}^{-1}d_{43},\\
q(0,23)&=d_{23}^{-1}d_{64-23}=d_{23}^{-1}d_{41},& q(0,25)&=d_{25}^{-1}d_{64-25}=d_{25}^{-1}d_{39},\\
q(0,27)&=d_{27}^{-1}d_{64-27}=d_{27}^{-1}d_{37},& q(0,29)&=d_{29}^{-1}d_{64-29}=d_{29}^{-1}d_{35},\\
q(0,31)&=d_{31}^{-1}d_{64-31}=d_{31}^{-1}d_{33}.
\end{align*}
\end{description}
\[
F_0=\prod_{l=1}^{15}\langle q(0,2l+1)\rangle.
\]

По лемме \ref{l:m(k)}
\begin{align*}
D&=\prod_{l=0}^{30}\left\langle d_{2l+1}\right\rangle=\\
&=\langle d_1\rangle\times\langle q(4,1)\rangle\times\prod_{l=0}^1\langle q(3,2l+1)\rangle\times\prod_{l=0}^3\langle q(2,2l+1)\rangle\times\\
&\qquad\quad\times\prod_{l=0}^7\langle q(1,2l+1)\rangle\times\prod_{l=1}^{15}\langle q(0,2l+1)\rangle;\\
F&=\langle d_1^{32}\rangle\times\langle q(4,1)^{16}\rangle\times\prod_{l=0}^1\langle q(3,2l+1)^8\rangle\times\prod_{l=0}^3\langle q(2,2l+1)^4\rangle\times\\
&\qquad\qquad\times\prod_{l=0}^7\langle q(1,2l+1)^2\rangle\times\prod_{l=1}^{15}\langle q(0,2l+1)\rangle;
\intertext{фактор-группа}
\sqrt{F}/F&=\langle d_1^{16}\rangle F\times\langle q(4,1)^8\rangle F\times\prod_{l=0}^1\langle q(3,2l+1)^4\rangle F\times\\
&\qquad\qquad\times\prod_{l=0}^3\langle q(2,2l+1)^2\rangle F\times\prod_{l=0}^7\langle q(1,2l+1)\rangle F.
\end{align*}
\end{proof}

\subsection{Вычисления по модулю $2$ порождающих\\ фактор-группы $\sqrt{F}/F$}

\begin{notation}
Для любого $k\in\{1,\dots,n-3\}$ положим
\[
P_{2^n}(k)=\prod_{j=k-1}^{n-4}d_1^{2^j}\equiv\prod_{j=k-1}^{n-4}d_{2^j}\pmod{2}
\]
по лемме \ref{l:md2}.
\end{notation}

Таким образом, получим.
\begin{description}
\item[Для $16$.]
Для $k\in\{1\}$ имеем
\[
P_{16}(k)=\prod_{j=k-1}d_1^{2^j}\equiv\prod_{j=k-1}d_{2^j}\pmod{2}.
\]
Поэтому
\[
\begin{array}{c|c}
k&1\\ \hline
P(k)&d_1
\end{array}
\]
\item[Для $32$.]
Для любого $k\in\{1,2\}$ имеем
\[
P_{32}(k)=\prod_{j=k-1}^1d_1^{2^j}\equiv\prod_{j=k-1}^1d_{2^j}\pmod{2}.
\]
Поэтому
\[
\begin{array}{c|c|c}
k&1&2\\ \hline
P_{32}(k)&d_1d_2&d_2
\end{array}
\]
\item[Для $64$.]
Для любого $k\in\{1,2,3\}$ имеем
\[
P_{64}(k)=\prod_{j=k-1}^2d_1^{2^j}\equiv\prod_{j=k-1}^2d_{2^j}\pmod{2}.
\]
Поэтому
\[
\begin{array}{c|c|c|c}
k&1&2&3\\ \hline
P_{64}(k)&d_1d_2d_4&d_2d_4&d_4
\end{array}
\]
\item[Для $128$.]
Для любого $k\in\{1,2,3,4\}$ имеем
\[
P_{128}(k)=\prod_{j=k-1}^2d_1^{2^j}\equiv\prod_{j=k-1}^2d_{2^j}\pmod{2}.
\]
Поэтому
\[
\begin{array}{c|c|c|c|c}
k&1&2&3&4\\ \hline
P(k)&d_1d_2d_4d_8&d_2d_4d_8&d_4d_8&d_8
\end{array}
\]
\end{description}

\begin{lemma}\label{l:128dk}
Для любого $k\in\{1,\dots,n-3\}$ имеем
\[
d_1^{-2^{k-1}}\equiv d_{2^{n-3}}P_{2^n}(k)\pmod{2}.
\]
\end{lemma}
\begin{proof}
Так как $d_1^{2^{n-2}}\equiv1\pmod{2}$ по лемме \ref{l:md2}, то по лемме \ref{l:ms2}
\begin{align*}
d_1^{-2^{k-1}}&\equiv d_1^{2^{n-2}-2^{k-1}}=d_1^{2^{k-1}(2^{n-1-k}-1)}=d_1^{2^{k-1}(2^{n-2-k}+\dots+1)}=\\
&=d_1^{2^{n-3}+\dots+2^{k-1}}=d_1^{2^{n-3}}d_1^{2^{n-4}}\dots d_1^{2^{k-1}}\equiv d_1^{2^{n-3}}P(k)\pmod{2}.
\end{align*}
Лемма доказана.
\end{proof}

Таким образом, получим.
\begin{description}
\item[Для $16$.]
Для $k\in\{1\}$ имеем
\[
d_1^{-1}\equiv d_2d_1\pmod{2}.
\]
\item[Для $32$.]
Для любого $k\in\{1,2\}$ имеем
\[
d_1^{-2^{k-1}}\equiv d_{4}P_{32}(k)\pmod{2}.
\]
Подробно
\begin{align*}
k&=1& d_1^{-2^{1-1}}&=d_1^{-1}\equiv d_4P_{32}(1)=d_4d_2d_1=\\
&&&=d_4P_{32}(2)d_1\pmod{2},\\
k&=2& d_1^{-2^{2-1}}&=d_1^{-2}\equiv d_4P_{32}(2)=d_4d_2\pmod{2}.
\end{align*}
\item[Для $64$.]
Для любого $k\in\{1,2,3\}$ имеем
\[
d_1^{-2^{k-1}}\equiv d_8P_{64}(k)\pmod{2}.
\]
Подробно
\begin{align*}
k&=1& d_1^{-2^{1-1}}&=d_1^{-1}\equiv d_8P_{64}(1)=d_8d_4d_2d_1=\\
&&&=d_8P_{64}(2)d_1\pmod{2},\\
k&=2& d_1^{-2^{2-1}}&=d_1^{-2}\equiv d_8P_{64}(2)=d_8d_4d_2=\\
&&&=d_8P_{64}(3)d_2\pmod{2},\\
k&=3& d_1^{-2^{3-1}}&=d_1^{-4}\equiv d_8P_{64}(3)=d_8d_4\pmod{2}.
\end{align*}
\item[Для $128$.]
Для любого $k\in\{1,2,3,4\}$ имеем
\[
d_1^{-2^{k-1}}\equiv d_{16}P_{128}(k)\pmod{2}.
\]
Подробно
\begin{align*}
k&=1& d_1^{-2^{1-1}}&=d_1^{-1}\equiv d_{16}P_{128}(1)=d_{16}d_8d_4d_2d_1=\\
&&&=d_{16}P(2)_{128}d_1\pmod{2},\\
k&=2& d_1^{-2^{2-1}}&=d_1^{-2}\equiv d_{16}P_{128}(2)=d_{16}d_8d_4d_2=\\
&&&=d_{16}P_{128}(3)d_2\pmod{2},\\
k&=3& d_1^{-2^{3-1}}&=d_1^{-4}\equiv d_{16}P_{128}(3)=d_{16}d_8d_4=\\
&&&=d_{16}P_{128}(4)d_4\pmod{2},\\
k&=4& d_1^{-2^{4-1}}&=d_1^{-8}\equiv d_{16}P_{128}(4)=d_{16}d_8\pmod{2}.
\end{align*}
\end{description}

\begin{lemma}\label{l:12832k}
Для любого $k\in\{1,\dots,4\}$ имеем
\[
d_{2^{n-1-k}-1}^{2^{k-1}}\equiv d_{2^{n-2}-2^{k-1}}\equiv d_{2^{k-1}}+r_{2^{k-1}}\pmod{2}.
\]
\end{lemma}
\begin{proof}
По лемме \ref{l:md2}
\[
d_{2^{n-1-k}-1}^{2^{k-1}}\equiv d_{(2^{n-1-k}-1)2^{k-1}}=d_{2^{n-2}-2^{k-1}}.
\]
Далее
\begin{align*}
d_{2^{n-2}-2^{k-1}}&=1+s_{2^{n-2}-2^{k-1}}=
\intertext{по определению $r_{2^{k-1}}$ получим}
&=1+r_{2^{k-1}}-s_{2^{k-1}}\equiv1+s_{2^{k-1}}+r_{2^{k-1}}=\\
&=d_{2^{k-1}}+r_{2^{k-1}}\pmod{2}.
\end{align*}

Лемма доказана.
\end{proof}

Таким образом, получим.
\begin{description}
\item[Для $16$.]
Для $k\in\{1\}$ имеем
\[
d_{2^2-1}^{2^{1-1}}\equiv d_{4-1}\equiv d_1+r_1\pmod{2}.
\]
\item[Для $32$.]
Для любого $k\in\{1,2\}$ имеем
\[
d_{2^{4-k}-1}^{2^{k-1}}\equiv d_{8-2^{k-1}}\equiv d_{2^{k-1}}+r_{2^{k-1}}\pmod{2}.
\]
\item[Для $64$.]
Для любого $k\in\{1,2,3\}$ имеем
\[
d_{2^{5-k}-1}^{2^{k-1}}\equiv d_{16-2^{k-1}}\equiv d_{2^{k-1}}+r_{2^{k-1}}\pmod{2}.
\]
\item[Для $128$.]
Для любого $k\in\{1,2,3,4\}$ имеем
\[
d_{2^{6-k}-1}^{2^{k-1}}\equiv d_{32-2^{k-1}}\equiv d_{2^{k-1}}+r_{2^{k-1}}\pmod{2}.
\]
\end{description}

\begin{lemma}\label{l:128qk}
Для любого $k\in\{1,\dots,n-3\}$ имеем
\[
q(k,1)^{2^{k-1}}\equiv1+d_{2^{k-1}}^{-1}r_{2^{k-1}}\pmod{2}.
\]
\end{lemma}
\begin{proof}
Имеем
\begin{align*}
q(k,1)^{2^{k-1}}&=d_1^{-2^{k-1}}d_{2^{n-1-k}-1}^{2^{k-1}}\equiv d_{2^{k-1}}^{-1}\left(d_{2^{k-1}}+r_{2^{k-1}}\right)=\\
&=1+d_{2^{k-1}}^{-1}r_{2^{k-1}}\pmod{2}.
\end{align*}

Лемма доказана.
\end{proof}

Таким образом, получим.
\begin{description}
\item[Для $16$.]
Для $k\in\{1\}$ имеем
\[
q(1,1)^{2^{1-1}}\equiv1+d_1^{-1}r_1\pmod{2}.
\]
\item[Для $32$.]
Для любого $k\in\{1,2\}$ имеем
\[
q(k,1)^{2^{k-1}}\equiv1+d_{2^{k-1}}^{-1}r_{2^{k-1}}\pmod{2}.
\]
\item[Для $64$.]
Для любого $k\in\{1,2,3\}$ имеем
\[
q(k,1)^{2^{k-1}}\equiv1+d_{2^{k-1}}^{-1}r_{2^{k-1}}\pmod{2}.
\]
\item[Для $128$.]
Для любого $k\in\{1,2,3,4\}$ имеем
\[
q(k,1)^{2^{k-1}}\equiv1+d_{2^{k-1}}^{-1}r_{2^{k-1}}\pmod{2}.
\]
\end{description}

\begin{lemma}\label{l:128rd}
Пусть $\widetilde{R}$, как в лемме {\rm\ref{l:idR}}.
\begin{enumerate}[{\rm1.}]
\item
Для любых целых $j$ и $l$
\[
d_jr_l=r_l+r_{l+j}+r_{l-j}\in\widetilde{R}.
\]
\item
Для любого $k\in\{1,\dots,n-3\}$ имеем
\[
d_{2^{k-1}}^{-1}r_{2^{k-1}}\in\widetilde{R}\pmod{2}.
\]
\item
Для любого целого $l$
\[
d_{2^{n-3}}r_l\equiv r_l\pmod{2}.
\]
\item
Для любого $k\in\{1,\dots,n-3\}$ имеем
\[
q(k,1)^{2^{k-1}}\equiv1+P_{2^n}(k)r_{2^{k-1}}\pmod{2}.
\]
\end{enumerate}
\end{lemma}
\begin{proof}
В самом деле.
\begin{enumerate}[{\rm1.}]
\item
По лемме \ref{l:sdr} для любых целых $j$ и $l$
\[
d_jr_l=r_l+r_{l+j}+r_{l-j}\in\widetilde{R},\text{ где $\widetilde{R}$ как в лемме \ref{l:idR}}.
\]
\item
По лемме \ref{l:128dk} для любого $k\in\{1,\dots,n-3\}$  из первого утверждения этой леммы следует
\[
d_{2^{k-1}}^{-1}r_{2^{k-1}}\in\widetilde{R}\pmod{2}.
\]
\item
Для любого целого $l$ по лемме \ref{l:sdr}
\[
d_{2^{n-3}}r_l=r_l+r_{l+2^{n-3}}+r_{l-2^{n-3}}.
\]
По лемме \ref{l:r2}
\[
r_{l-2^{n-3}}\equiv r_{2^{n-3}-l}\pmod{2}\text{ и }r_{2^{n-3}+l}\equiv r_{2^{n-3}-l}\pmod{2}.
\]
Поэтому
\[
d_{2^{n-3}}r_l=r_l+r_{l+2^{n-3}}+r_{l-2^{n-3}}\equiv r_l+2r_{l-2^{n-3}}\equiv r_l\pmod{2}
\]
\item
Для любого $k\in\{1,\dots,n-3\}$ по лемме \ref{l:128qk} и утверждению 3 этой леммы
\[
q(k,1)^{2^{k-1}}\equiv1+P_{2^n}(k)r_{2^{k-1}}\pmod{2}.
\]
\end{enumerate}

Лемма доказана.
\end{proof}

\begin{lemma}\label{l:128q}
Для $2^n\in\{16,32,64,128\}$ $\longleftrightarrow$ $n\in\{4,5,6,7\}$ имеем следующие соотношения.
\begin{description}
\item[Для $16$.]{\ }
\[
q(1,1)\equiv1+r_1\pmod{2}.
\]
\item[Для $32$.]{\ }
\begin{description}
\item[$k=2.$]
$q(2,1)^2\equiv1+r_2\pmod{2}$.
\item[$k=1.$]
$q(1,1)\equiv1+r_2+r_3\pmod{2}$.
\end{description}
\item[Для $64$.]{\ }
\begin{description}
\item[$k=3.$]
$q(3,1)^4\equiv1+r_4\pmod{2}$.
\item[$k=2.$]
$q(2,1)^2\equiv1+r_4+r_6\pmod{2}$.
\item[$k=1.$]
$q(1,1)\equiv1+r_4+r_6+(r_1+r_3+r_7)\pmod{2}$.
\end{description}
\item[Для $128$.]{\ }
\begin{description}
\item[$k=4.$]
$q(4,1)^8\equiv1+r_8\pmod{2}.$
\item[$k=3.$]
$q(3,1)^4\equiv1+r_8+r_{12}\pmod{2}$.
\item[$k=2.$]
$q(2,1)^2\equiv1+r_8+r_{12}+(r_2+r_6+r_{14})\pmod{2}$.
\item[$k=1.$]
$q(1,1)\equiv1+r_8+r_{12}+(r_2+r_6+r_{14})+(r_3+r_5+r_9+r_{11}+r_{15})\pmod{2}$.
\end{description}
\end{description}
\end{lemma}
\begin{proof}{\ }
\begin{description}
\item[Для $16$.]
Применим лемму \ref{l:128rd} для 16.
По леммам \ref{l:sdr} и \ref{l:r2} получим
\begin{align*}
q(1,1)&\equiv1+P_{16}(1)r_1=1+d_1r_1\equiv1+r_1+r_2+r_0\equiv\\
&\equiv 1+r_1 \pmod{2}.
\end{align*} 
\end{description}
\begin{description}
\item[Для $32$.]
Применим лемму \ref{l:128rd} для 32.
\begin{description}
\item[$k=2.$]
В этом случае по леммам \ref{l:sdr} и \ref{l:r2} получим
\begin{align*}
q(2,1)^2&\equiv1+P_{32}(2)r_2=1+d_2r_2\equiv1+r_2+r_4+r_0\equiv\\
&\equiv 1+r_2 \pmod{2}.
\end{align*} 
\item[$k=1.$]
В этом случае
\begin{align*}
q(1,1)&\equiv1+P_{32}(1)r_1=1+d_2d_1r_1\equiv\\
&\equiv1+d_2(r_1+r_2+r_0)\equiv
\intertext{по предшествующему утверждению}
&\equiv\left(1+r_2\right)+d_2r_1\pmod{2}\equiv\\
&\equiv 1+r_2+(r_1+r_3+r_1)\pmod{2}\equiv\\
&\equiv 1+r_2+r_3 \pmod{2}.
\end{align*}
\end{description}
\end{description}
\begin{description}
\item[Для $64$.]
Применим лемму \ref{l:128rd} для 64.
\begin{description}
\item[$k=3.$]
В этом случае по леммам \ref{l:sdr} и \ref{l:r2} получим
\begin{align*}
q(3,1)^4&\equiv1+P_{64}(3)r_4=1+d_4r_4\equiv\\
&\equiv1+r_4+r_8+r_0\equiv1+r_4\pmod{2}.
\end{align*} 
\item[$k=2.$]
 В этом случае
\begin{align*}
q(2,1)^2&\equiv1+P_{64}(2)r_2=1+d_4d_2r_2\equiv\\
&\equiv1+d_4(r_2+r_4+r_0)\equiv
\intertext{по предшествующему утверждению}
&\equiv1+r_4+d_4r_2\equiv1+r_4+r_2+r_6+r_{-2}\equiv
\intertext{по лемме \ref{l:r2}}
&\equiv1+r_4+r_2+r_6+r_2\equiv1+r_4+r_6\pmod{2}.
\end{align*} 
\item[$k=1.$]
Наконец,
\begin{align*}
q(1,1)&\equiv1+P_{64}(1)r_1=1+d_4d_2d_1r_1\equiv\\
&\equiv1+d_4d_2(r_1+r_2+r_0)\equiv
\intertext{по предшествующему утверждению}
&\equiv\left(1+r_4+r_6\right)+d_4d_2r_1\pmod{2},\\
d_4d_2r_1&\equiv d_4(r_1+r_3+r_1)\equiv d_4r_3\equiv\\
&\equiv r_3+r_7+r_1\pmod{2}.
\end{align*}
Итого
\begin{align*}
q(1,1)&\equiv1+r_4+r_6+(r_3+r_7+r_1)\pmod{2}.
\end{align*}
\end{description}
\end{description}
\begin{description}
\item[Для $128$.]
Применим лемму \ref{l:128rd}.
\begin{description}
\item[$k=4.$]
В этом случае по леммам \ref{l:sdr} и \ref{l:r2} получим
\begin{align*}
q(4,1)^8&\equiv1+P(4)r_8=1+d_8r_8=1+r_8+r_{16}+r_0\equiv\\
&\equiv1+r_8\pmod{2}.
\end{align*}
\item[$k=3.$]
Теперь
\begin{align*}
q(3,1)^4&\equiv1+P(3)r_4=1+d_8d_4r_4\equiv1+d_8(r_4+r_8)\equiv
\intertext{по предшествующему утверждению и лемме \ref{l:sdr}}
&\equiv1+r_8+r_4+r_{12}+r_{-4}\equiv
\intertext{по лемме \ref{l:r2}}
&\equiv1+r_8+r_4+r_{12}+r_4\equiv1+r_8+r_{12}\pmod{2}.
\end{align*}
\item[$k=2.$]
В этом случае
\begin{align*}
q(2,1)^2&\equiv1+P(2)r_2=1+d_8d_4d_2r_2\equiv1+d_8d_4(r_2+r_4)\equiv
\intertext{по предшествующему утверждению}
&\equiv1+r_8+r_{12}+d_8d_4r_2\equiv\\
&\equiv1+r_8+r_{12}+d_8(r_2+r_6+r_{-2})\equiv
\intertext{по лемме \ref{l:r2}}
&\equiv1+r_8+r_{12}+d_8(r_2+r_6+r_2)\equiv\\
&\equiv1+r_8+r_{12}+d_8r_6=\\
&=1+r_8+r_{12}+r_6+r_{14}+r_{-2}\equiv\\
&\equiv1+r_8+r_{12}+r_6+r_{14}+r_2=\\
&=1+r_8+r_{12}+(r_2+r_6+r_{14})+\\
&+(r_3+r_5+r_9+r_{11}+r_{15})\pmod{2}.
\end{align*}
\item[$k=1.$]
Наконец,
\begin{align*}
q(1,1)&\equiv1+P(1)r_1=1+d_8d_4d_2d_1r_1\equiv\\
&\equiv1+d_8d_4d_2(r_1+r_2)\equiv
\intertext{по предшествующему утверждению}
&\equiv\left(1+r_8+r_{12}+(r_2+r_6+r_{14})\right)+\\
&+d_8d_4d_2r_1\pmod{2},\\
d_8d_4d_2r_1&\equiv d_8d_4d_2r_1=d_8d_4(r_1+r_3+r_1)\equiv d_8d_4r_3\equiv\\
&\equiv d_8(r_3+r_7+r_1)\equiv\\
&\equiv(r_3+r_{11}+r_5)+(r_7+r_{15}+r_1)+\\
&+(r_1+r_9+r_7)=\\
&=2r_1+r_3+r_5+2r_7+r_9+r_{11}+r_{15}\equiv\\
&\equiv r_3+r_5+r_9+r_{11}+r_{15}\pmod{2}.
\end{align*}
\end{description}
\end{description}

Лемма доказана.
\end{proof}

\begin{proposition}\label{p:128}
Для $2^n\in\{16,32,64,128\}$ $\longleftrightarrow$ $n\in\{4,5,6,7\}$ имеем следующие соотношения.
\begin{description}
\item[Для $16$.] 
\[
\sqrt{F}/F=\langle d_2\rangle F\times\langle1+r_1\rangle F.
\] 
\item[Для $32$.]
\[
\sqrt{F}/F=\langle d_4\rangle F\times\langle1+r_2\rangle F\times\langle1+r_2+r_1\rangle F\times\langle1+r_2+r_3\rangle F.
\] 
\item[Для $64$.]
\begin{align*}
\sqrt{F}/F&=\langle d_8\rangle F\times\langle1+r_4\rangle F\times\langle1+r_4+r_6\rangle\times\langle1+r_4+r_2\rangle F\times\\
&\times\langle1+r_4+r_6+(r_1+r_3+r_7)\rangle F\times\\
&\times\langle1+r_4+r_2+(r_3+r_5+r_7)\rangle F\times\\
&\times\langle1+r_4+r_2+(r_1+r_3+r_5)\rangle F\times\\
&\times\langle1+r_4+r_6+(r_1+r_5+r_7)\rangle F.
\end{align*}
\item[Для $128$.]
\begin{align*}
&\sqrt{F}/F=\langle d_{16}\rangle F\times\langle1+r_8\rangle F\times\\
&\times\langle1+r_8+r_{12}\rangle F\times\langle1+r_8+r_4\rangle F\times\\
&\times\langle1+r_8+r_{12}+(r_2+r_6+r_{14})\rangle F\times\\
&\times\langle1+r_8+r_4+(r_6+r_{10}+r_{14})\rangle F\times\\
&\times\langle1+r_8+r_4+(r_2+r_6+r_{10})\rangle F\times\\
&\times\langle1+r_8+r_{12}+(r_2+r_{10}+r_{14})\rangle F\times\\
&\hspace*{-3pt}\times\langle1+r_8+r_{12}+(r_2+r_6+r_{14})+(r_3+r_5+r_9+r_{11}+r_{15})\rangle F\times\\ 
&\hspace*{-3pt}\times\langle1+r_8+r_4+(r_6+r_{10}+r_{14})+(r_1+r_5+r_9+r_{13}+r_{15})\rangle F\times\\ 
&\hspace*{-3pt}\times\langle1+r_8+r_4+(r_2+r_6+r_{10})+(r_7+r_9+r_{11}+r_{13}+r_{15})\rangle F\times\\  
&\hspace*{-3pt}\times\langle1+r_8+r_{12}+(r_2+r_{10}+r_{14})+(r_1+r_3+r_9+r_{11}+r_{13})\rangle F\times\\  
&\hspace*{-3pt}\times\langle1+r_8+r_{12}+(r_2+r_{10}+r_{14})+(r_3+r_5+r_7+r_{13}+r_{15})\rangle F\times\\  
&\hspace*{-3pt}\times\langle1+r_8+r_4+(r_2+r_6+r_{10})+(r_1+r_3+r_5+r_7+r_9)\rangle F\times\\  
&\hspace*{-3pt}\times\langle1+r_8+r_4+(r_6+r_{10}+r_{14})+(r_1+r_3+r_7+r_{11}+r_{15})\rangle F\times\\  
&\hspace*{-3pt}\times\langle1+r_8+r_{12}+(r_2+r_6+r_{14})+(r_1+r_5+r_7+r_{11}+r_{13})\rangle F.  
\end{align*}
\end{description}
\end{proposition}
\begin{proof}
\begin{description}
\item[Для $16$.] 
Всё следует из леммы \ref{l:128q}.
\item[Для $32$.]
Из следствия \ref{c:sqrF} и леммы \ref{l:128q} следует, что надо найти только $q(1,3)$.
Имеем
\[
q(1,1)=d_1^{-1}d_7\text{ и }q(1,3)=d_3^{-1}d_5.
\]
Применим автоморфизм $\sigma_3$ поля $\Q_{32}$ из леммы \ref{l:cyc}
\begin{align*}
\sigma_3(d_1)&=\sigma_3(1+\alpha+\alpha^{-1})=1+\alpha^3+\alpha^{-3}=d_3,\\
\sigma_3(d_7)&=\sigma_3(1+\alpha^7+\alpha^{-7})=1+\alpha^{21}+\alpha^{-21}=\\
&=1-\alpha^5-\alpha^{-5}\equiv d_5\pmod{2},\\
\sigma_3(r_2)&=\sigma_3(\alpha^2+\alpha^{-2}+\alpha^6+\alpha^{-6})\equiv\alpha^6+\alpha^{-6}+\alpha^{18}+\alpha^{-18}=\\
&=\alpha^6+\alpha^{-6}-\alpha^2-\alpha^{-2}\equiv r_2\pmod{2},\\
\sigma_3(r_3)&=\sigma_3(\alpha^3+\alpha^{-3}+\alpha^5+\alpha^{-5})\equiv\alpha^9+\alpha^{-9}+\alpha^{15}+\alpha^{-15}\equiv\\
&\equiv-\alpha^{-7}-\alpha^7-\alpha-\alpha^{-1}\equiv r_1\pmod{2}.
\end{align*}
Поэтому
\begin{align*}
\sigma_3(q(1,1))&=\sigma_3(d_1^{-1}d_7)\equiv d_3^{-1}d_5=q(1,3)\pmod{2},\\
\sigma_3(q(1,1))&\equiv\sigma_3(1+r_2+r_3)\equiv1+r_2+r_1\pmod{2}.
\end{align*}
\item[Для $64$.]
Из следствия \ref{c:sqrF} и леммы \ref{l:128q} следует, что надо найти 
\[
 q(2,3)^2, q(1,3), q(1,5), q(1,7).
\]
Имеем по лемме \ref{l:md2}
\begin{align*}
q(2,1)^2&=d_1^{-2}d_7^2\equiv d_2^{-1}d_{14}\pmod2,\\
 q(2,3)^2&=d_3^{-3}d_5^2\equiv d_6^{-1}d_{10}\pmod2.
\end{align*}
Применим автоморфизм $\sigma_3$ поля $\Q_{64}$ из леммы \ref{l:cyc}
\begin{align*}
\sigma_3(d_2)&=\sigma_3(1+\alpha^2+\alpha^{-2})=1+\alpha^6+\alpha^{-6}=d_6,\\
\sigma_3(d_{14})&=\sigma_3(1+\alpha^{14}+\alpha^{-14})=1+\alpha^{42}+\alpha^{-42}=\\
&=1+\alpha^{10}+\alpha^{-10}=d_{10}\pmod{2},\\
\sigma_3(r_4)&=\sigma_3(\alpha^4+\alpha^{-4}+\alpha^{12}+\alpha^{-12})=\\
&=\alpha^{12}+\alpha^{-12}+\alpha^{36}+\alpha^{-36}=\\
&\equiv\alpha^{12}+\alpha^{-12}-\alpha^4-\alpha^{-4}\equiv r_4\pmod{2},\\
\sigma_3(r_6)&=\sigma_3(\alpha^6+\alpha^{-6}+\alpha^{10}+\alpha^{-10})=\\
&=\alpha^{18}+\alpha^{-18}+\alpha^{30}+\alpha^{-30}\equiv\\
&\equiv-\alpha^{14}-\alpha^{-14}-\alpha^{-2}-\alpha^2\equiv r_2\pmod{2}.
\end{align*}
Поэтому
\begin{align*}
\sigma_3(q(1,1)^2)&=\sigma_3(d_2^{-1}d_{14})\equiv d_6^{-1}d_{10}\equiv q(1,3)^2\pmod{2},\\
\sigma_3(q(1,1)^2)&\equiv\sigma_3(1+r_4+r_6)\equiv1+r_4+r_2\pmod{2}.
\end{align*}

Теперь рассмотрим автоморфизмы $\sigma_3$, $\sigma_5$ и $\sigma_7$ поля $\Q_{64}$ из леммы \ref{l:cyc}.
Для $k\in\{1,2,3\}$ имеем
\begin{align*}
\sigma_{2k+1}(d_1)&=\sigma_{2k+1}(1+\alpha+\alpha^{-1})=1+\alpha_{2k+1}+\alpha^{-2k-1}=\\
&=d_{2k+1},\\
\sigma_{2k+1}(d_{15})&=\sigma_{2k+1}(1+\alpha^{15}+\alpha^{-15})=\\
&=1+\alpha^{15(2k+1)}+\alpha^{-15(2k+1)}=\\
&=\begin{cases}
1+\alpha^{45}+\alpha^{-45}&\text{при }2k+1=3,\\
1+\alpha^{75}+\alpha^{-75}&\text{при }2k+1=5,\\
1+\alpha^{105}+\alpha^{-105}&\text{при }2k+1=7
\end{cases}
\equiv\\
&\equiv
\begin{cases}
1-\alpha^{13}-\alpha^{-13}&\text{при }2k+1=3,\\
1+\alpha^{11}+\alpha^{-11}&\text{при }2k+1=5,\\
1-\alpha^9-\alpha^{-9}&\text{при }2k+1=7
\end{cases}
\equiv\\
&\equiv
\begin{cases}
d_{13}&\text{при }2k+1=3,\\
d_{11}&\text{при }2k+1=5,\\
d_9&\text{при }2k+1=7
\end{cases}=\\
&=d_{16-(2k+1)}\pmod{2}.
\end{align*}
Поэтому
\[
q(1,2k+1)=\sigma_{2k+1}(q(1,1)).
\]
Также
\begin{align*}
\sigma_{2k+1}(r_4)&=\sigma_{2k+1}(\alpha^4+\alpha^{-4}+\alpha^{12}+\alpha^{-12})=\\
&=\alpha^{4(2k+1)}+\alpha^{-4(2k+1)}+\alpha^{12(2k+1)}+\alpha^{-12(2k+1)}=\\
&=\begin{cases}
\alpha^{12}+\alpha^{-12}+\alpha^{36}+\alpha^{-36}&\text{при }2k+1=3,\\
\alpha^{20}+\alpha^{-20}+\alpha^{60}+\alpha^{-64}&\text{при }2k+1=5,\\
\alpha^{28}+\alpha^{-28}+\alpha^{84}+\alpha^{-84}&\text{при }2k+1=7
\end{cases}
\equiv\\
&\equiv
\begin{cases}
\alpha^{12}+\alpha^{-12}-\alpha^4-\alpha^{-4}&\text{при }2k+1=3,\\
-\alpha^{-12}-\alpha^{12}+\alpha^{-4}+\alpha^4&\text{при }2k+1=5,\\
-\alpha^{-4}-\alpha^4+\alpha^{-12}+\alpha^{12}&\text{при }2k+1=7
\end{cases}\equiv\\
&\equiv r_4\pmod{2};\\
\sigma_{2k+1}(r_6)&=\sigma_{2k+1}(\alpha^6+\alpha^{-6}+\alpha^{10}+\alpha^{-10})=\\
&=\alpha^{6(2k+1)}+\alpha^{-6(2k+1)}+\alpha^{10(2k+1)}+\alpha^{-10(2k+1)}=\\
&=\begin{cases}
\alpha^{18}+\alpha^{-18}+\alpha^{30}+\alpha^{-30}&\text{при }2k+1=3,\\
\alpha^{30}+\alpha^{-30}+\alpha^{50}+\alpha^{-50}&\text{при }2k+1=5,\\
\alpha^{42}+\alpha^{-42}+\alpha^{70}+\alpha^{-70}&\text{при }2k+1=7
\end{cases}
\equiv\\
&\equiv
\begin{cases}
-\alpha^{14}-\alpha^{-14}-\alpha^{-2}-\alpha^2&\text{при }2k+1=3,\\
-\alpha^{-2}-\alpha^2+\alpha^{-14}+\alpha^{14}&\text{при }2k+1=5,\\
-\alpha^{10}-\alpha^{-10}+\alpha^6+\alpha^{-6}&\text{при }2k+1=7
\end{cases}
\equiv\\
&\equiv
\begin{cases}
r_2&\text{при }2k+1=3,\\
r_2&\text{при }2k+1=5,\\
r_6&\text{при }2k+1=7
\end{cases}.
\end{align*}

Наконец
\begin{align*}
\sigma_{2k+1}(r_1)&=\sigma_{2k+1}(\alpha+\alpha^{-1}+\alpha^{15}+\alpha^{-15})=\\
&=\alpha^{2k+1}+\alpha^{-(2k+1)}+\alpha^{15(2k+1)}+\alpha^{-15(2k+1)}=\\
&=\begin{cases}
\alpha^3+\alpha^{-3}+\alpha^{45}+\alpha^{-45}&\text{при }2k+1=3,\\
\alpha^5+\alpha^{-5}+\alpha^{75}+\alpha^{-75}&\text{при }2k+1=5,\\
\alpha^7+\alpha^{-7}+\alpha^{105}+\alpha^{-105}&\text{при }2k+1=7
\end{cases}
\equiv\\
&\equiv
\begin{cases}
\alpha^3+\alpha^{-3}-\alpha^{13}+\alpha^{-13}&\text{при }2k+1=3,\\
\alpha^5+\alpha^{-5}+\alpha^{11}+\alpha^{-11}&\text{при }2k+1=5,\\
\alpha^7+\alpha^{-7}-\alpha^9+\alpha^{-9}&\text{при }2k+1=7
\end{cases}
\equiv\\
&\equiv
\begin{cases}
r_3&\text{при }2k+1=3,\\
r_5&\text{при }2k+1=5,\\
r_7&\text{при }2k+1=7
\end{cases};\\
\sigma_{2k+1}(r_3)&=\sigma_{2k+1}(\alpha^3+\alpha^{-3}+\alpha^{13}+\alpha^{-13})=\\
&=\alpha^{3(2k+1)}+\alpha^{-3(2k+1)}+\alpha^{13(2k+1)}+\alpha^{-13(2k+1)}=\\
&=\begin{cases}
\alpha^9+\alpha^{-9}+\alpha^{39}+\alpha^{-39}&\text{при }2k+1=3,\\
\alpha^{15}+\alpha^{-15}+\alpha^{65}+\alpha^{-65}&\text{при }2k+1=5,\\
\alpha^{21}+\alpha^{-21}+\alpha^{91}+\alpha^{-91}&\text{при }2k+1=7
\end{cases}
\equiv\\
&\equiv
\begin{cases}
\alpha^9+\alpha^{-7}-\alpha^7-\alpha^{-7}&\text{при }2k+1=3,\\
\alpha^{15}+\alpha^{-15}+\alpha+\alpha^{-1}&\text{при }2k+1=5,\\
-\alpha^{11}+\alpha^{-11}-\alpha^5+\alpha^{-5}&\text{при }2k+1=7
\end{cases}
\equiv\\
&\equiv
\begin{cases}
r_7&\text{при }2k+1=3,\\
r_1&\text{при }2k+1=5,\\
r_5&\text{при }2k+1=7
\end{cases};\\
\sigma_{2k+1}(r_7)&=\sigma_{2k+1}(\alpha^7+\alpha^{-7}+\alpha^9+\alpha^{-9})=\\
&=\alpha^{7(2k+1)}+\alpha^{-7(2k+1)}+\alpha^{9(2k+1)}+\alpha^{-9(2k+1)}=\\
&=\begin{cases}
\alpha^{21}+\alpha^{-21}+\alpha^{27}+\alpha^{-27}&\text{при }2k+1=3,\\
\alpha^{35}+\alpha^{-35}+\alpha^{45}+\alpha^{-45}&\text{при }2k+1=5,\\
\alpha^{49}+\alpha^{-49}+\alpha^{63}+\alpha^{-63}&\text{при }2k+1=7
\end{cases}
\equiv\\
&\equiv
\begin{cases}
-\alpha^{-11}-\alpha^{11}-\alpha^5-\alpha^{-5}&\text{при }2k+1=3,\\
-\alpha^3-\alpha^{-3}-\alpha^{13}+\alpha^{-13}&\text{при }2k+1=5,\\
\alpha^{-15}+\alpha^{15}+\alpha^{-1}+\alpha&\text{при }2k+1=7
\end{cases}
\equiv\\
&\equiv
\begin{cases}
r_5&\text{при }2k+1=3,\\
r_3&\text{при }2k+1=5,\\
r_1&\text{при }2k+1=7
\end{cases}.
\end{align*}
Итак
\begin{align*}
q(1,2k+1)&=\sigma_{2k+1}(q(1,1))=\\
&=
\begin{cases}
1+r_4+r_2+(r_3+r_7+r_5)&\text{при }2k+1=3,\\
1+r_4+r_2+(r_5+r_1+r_3)&\text{при }2k+1=5,\\
1+r_4+r_6+(r_7+r_5+r_1)&\text{при }2k+1=7
\end{cases}=\\
&=
\begin{cases}
1+r_4+r_2+(r_3+r_5+r_7)&\text{при }2k+1=3,\\
1+r_4+r_2+(r_1+r_3+r_5)&\text{при }2k+1=5,\\
1+r_4+r_6+(r_1+r_5+r_7)&\text{при }2k+1=7.
\end{cases}
\end{align*}
\item[Для $128$.]
Из следствия \ref{c:sqrF} и леммы \ref{l:128q} следует, что надо найти
\begin{gather*}
 q(3,3)^4, q(2,3)^2, q(2,5)^2, q(2,7)^2,\\
 q(1,3),q(1,5),q(1,7),q(1,9),q(1,11),q(1,13),q(1,15).
\end{gather*}

Непосредственно из определения следует, что  по модулю $2$ для $128$
\[
q(3,3)^4, q(2,3)^2, q(2,5)^2, q(2,7)^2
\]
совпадают с 
\[
 q(2,3)^2, q(1,3), q(1,5), q(1,7)
\]
для $64$, соответственно.
Более определённо, имеем для $128$
\begin{align*}
q(3,3)^4&=1+r_8+r_4\pmod{2},\\ 
q(2,3)^2&=1+r_8+r_4+(r_6+r_{10}+r_{14})\pmod{2},\\ 
q(2,5)^2&=1+r_8+r_4+(r_2+r_6+r_{10})\pmod{2},\\
q(2,7)^2&=1+r_8+r_{12}+(r_2+r_{10}+r_{14})\pmod{2}.
\end{align*}

Поэтому надо рассмотреть только
\[
 q(1,3),q(1,5),q(1,7),q(1,9),q(1,11),q(1,13),q(1,15).
\]
Теперь рассмотрим автоморфизмы 
\[
\left\{\sigma_{2k+1}\mid k\in\{1,2,3,4,5,6,7\}\right\}
\]
поля $\Q_{128}$ из леммы \ref{l:cyc}.
Для $k\in\{1,2,\dots,7\}$ имеем
\begin{align*}
\sigma_{2k+1}(d_1)&=\sigma_{2k+1}(1+\alpha+\alpha^{-1})=1+\alpha_{2k+1}+\alpha^{-2k-1}=\\
&=d_{2k+1},\\
\sigma_{2k+1}(d_{31})&=\sigma_{2k+1}(1+\alpha^{31}+\alpha^{-31})=\\
&=1+\alpha^{31(2k+1)}+\alpha^{-31(2k+1)}=\\
&=1+\alpha^{32(2k+1)-(2k+1)}+\alpha^{-32(2k+1)+(2k+1)}=\\
&=1+\alpha^{64k+32-(2k+1)}+\alpha^{-64k-32+(2k+1)}=\\
&=1+(-1)^k\alpha^{32-(2k+1)}+(-1)^k\alpha^{-32+(2k+1)}\equiv\\
&\equiv1+\alpha^{32-(2k+1)}+\alpha^{-32+(2k+1)}=\\
&=d_{32-(2k+1)}\pmod{2}.
\end{align*}
Поэтому
\[
q(1,2k+1)\equiv\sigma_{2k+1}(q(1,1))\pmod{2}.
\]

По лемме \ref{l:128q}
\begin{align*}
q(1,1)&\equiv1+r_8+r_{12}+(r_2+r_6+r_{14})+\\
&+(r_3+r_5+r_9+r_{11}+r_{15})\pmod{2}.
\end{align*}
Положим
\begin{align*}
R_0&=1+r_8+r_{12}+(r_2+r_6+r_{14}),\\
R_1&=r_3+r_5+r_9+r_{11}+r_{15}.
\end{align*}
Тогда
\[
q(1,1)\equiv R_0+R_1\pmod{2}.
\]

Теперь по лемме \ref{l:r2}
\begin{align*}
\sigma_3(R_0)&=1+r_{24}+r_{36}+(r_6+r_{18}+r_{42})\equiv\\
&\equiv1+r_8+r_4+(r_6+r_{14}+r_{10})\pmod{2},\\
\sigma_3(R_1)&=r_9+r_{15}+r_5+r_{33}+r_{45}\equiv\\
&\equiv r_9+r_{15}+r_5+r_1+r_{13}\pmod{2}
\intertext{поэтому}
q(1,3)&\equiv\sigma_3(q(1,1))\equiv1+r_8+r_4+(r_6+r_{10}+r_{14})+\\
&+(r_1+r_5+r_9+r_{13}+r_{15})\pmod{2};\\
\sigma_5(R_0)&=1+r_{40}+r_{60}+(r_{10}+r_{30}+r_{70})\equiv\\
&\equiv1+r_8+r_4+(r_{10}+r_2+r_6)\pmod{2},\\
\sigma_5(R_1)&=r_{15}+r_{25}+r_{45}+r_{55}+r_{75}\equiv\\
&\equiv r_{15}+r_7+r_{13}+r_9+r_{11}\pmod{2}
\intertext{поэтому}
q(1,5)&\equiv\sigma_5(q(1,1))\equiv1+r_8+r_{12}+(r_2+r_6+r_{10})+\\
&+(r_7+r_9+r_{11}+r_{13}+r_{15})\pmod{2};\\
\sigma_7(R_0)&=1+r_{56}+r_{84}+(r_{14}+r_{42}+r_{98})\equiv\\
&\equiv1+r_8+r_{12}+(r_{14}+r_{10}+r_2)\pmod{2},\\
\sigma_7(R_1)&=r_{21}+r_{35}+r_{63}+r_{77}+r_{105}\equiv\\ 
&\equiv r_{11}+r_3+r_1+r_{13}+r_9\pmod{2}
\intertext{поэтому}
q(1,7)&\equiv\sigma_7(q(1,1))\equiv1+r_8+r_{12}+(r_2+r_{10}+r_{14})+\\
&+(r_1+r_3+r_9+r_{11}+r_{13})\pmod{2};\\
\sigma_9(R_0)&=1+r_{72}+r_{108}+(r_{18}+r_{54}+r_{126})\equiv\\
&\equiv1+r_8+r_{12}+(r_{14}+r_{10}+r_2)\pmod{2},\\
\sigma_9(R_1)&=r_{27}+r_{45}+r_{81}+r_{99}+r_{135}\equiv\\
&\equiv r_5+r_{13}+r_{15}+r_3+r_7\pmod{2}
\intertext{поэтому}
q(1,9)&\equiv\sigma_9(q(1,1))\equiv1+r_8+r_{12}+(r_2+r_{10}+r_{14})+\\
&+(r_3+r_5+r_7+r_{13}+r_{15})\pmod{2};\\
\sigma_{11}(R_0)&=1+r_{88}+r_{132}+(r_{22}+r_{66}+r_{154})\equiv\\
&\equiv1+r_8+r_4+(r_{10}+r_2+r_6)\pmod{2},\\
\sigma_{11}(R_1)&=r_{33}+r_{55}+r_{99}+r_{121}+r_{165}\equiv\\ 
&\equiv r_1+r_9+r_3+r_7+r_5\pmod{2}
\intertext{поэтому}
q(1,11)&\equiv\sigma_{11}(q(1,1))\equiv1+r_8+r_4+(r_2+r_6+r_{10})+\\
&+(r_1+r_3+r_5+r_7+r_9)\pmod{2};\\
\sigma_{13}(R_0)&=1+r_{104}+r_{156}+(r_{26}+r_{78}+r_{182})\equiv\\
&\equiv1+r_8+r_4+(r_6+r_{14}+r_{10})\pmod{2},\\
\sigma_{13}(R_1)&=r_{39}+r_{65}+r_{117}+r_{143}+r_{195}\equiv\\ 
&\equiv r_7+r_1+r_{11}+r_{15}+r_3\pmod{2}
\intertext{поэтому}
q(1,13)&\equiv\sigma_{13}(q(1,1))\equiv1+r_8+r_{12}+(r_6+r_{10}+r_{14})+\\
&+(r_1+r_3+r_7+r_{11}+r_{15})\pmod{2};\\
\sigma_{15}(R_0)&=1+r_{120}+r_{180}+(r_{30}+r_{90}+r_{210})\equiv\\
&\equiv1+r_8+r_{12}+(r_2+r_6+r_{14})\pmod{2},\\
\sigma_{15}(R_1)&=r_{45}+r_{75}+r_{135}+r_{165}+r_{255}\equiv\\
&\equiv r_{13}+r_{11}+r_7+r_5+r_1\pmod{2}
\intertext{поэтому}
q(1,15)&\equiv\sigma_{15}(q(1,1))\equiv1+r_8+r_{12}+(r_2+r_6+r_{14})+\\
&+(r_1+r_5+r_7+r_{11}+r_{13})\pmod{2}.
\end{align*}

\end{description}

\end{proof}

\section{Основная теорема}

Целью этого раздела является доказательство следующего результата.

\begin{theorem}
Пусть
\[
 2^n\in\{16,32,64,128\}\longleftrightarrow n\in\{4,5,6,7\}.
\]
Тогда
\begin{align*}
E=F=\langle d_1^{2^{n-2}}\rangle&\times \prod_{2l+1\in A_0\setminus\{1\}}\langle d_{2l+1}^{-1}d_{2^{n-1}-(2l+1)}\rangle\times\\
&\times\prod_{k=1}^{n-3}\prod_{2l+1\in A_k}\langle d_{2l+1}^{-2^k}d_{2^{n-1}-(2l+1)}^{2^k}\rangle.
\end{align*}
что равносильно
\begin{align*}
V_1=\langle u_{\chi_1}(d_1^{2^{n-2}})\rangle&\times\prod_{2l+1\in A_0\setminus\{1\}}\langle u_{\chi_1}(d_{2l+1}^{-1})u_{\chi_1}(d_{2^{n-1}-(2l+1)})\rangle\times\\
&\times\prod_{k=1}^{n-3}\prod_{2l+1\in A_k}\langle u_{\chi_1}(d_{2l+1}^{-2^k})u_{\chi_1}(d_{2^{n-1}-(2l+1)}^{2^k}).
\end{align*}
Более развёрнуто получаем.
\begin{description}
\item[Для $16$.]
\[
E=\langle d_1^4\rangle\times\langle d_3^{-1}d_5\rangle\times\langle d_1^{-2}d_3^2\rangle.
\]
\item[Для $32$.]
\begin{align*}
E=\langle d_1^8\rangle&\times\langle d_3^{-1}d_{13}\rangle\times\langle d_5^{-1}d_{11}\rangle\times\langle d_7^{-1}d_9\rangle\times\\
&\times\langle d_1^{-2}d_7^2\rangle\times\langle d_3^{-2}d_5^2\rangle\times\langle d_1^{-4}d_3^4\rangle.
\end{align*}
\item[Для $64$.]
\begin{align*}
E=\langle d_1^{16}\rangle&\times\prod_{l=1}^7\langle d_{2l+1}^{-1}d_{32-(2l+1)}\rangle\times\\
&\times\langle d_3^{-2}d_{13}^2\rangle\times\langle d_5^{-2}d_{11}^2\rangle\times\langle d_7^{-2}d_9^2\rangle\times\\
&\times\langle d_1^{-4}d_7^4\rangle\times\langle d_3^{-4}d_5^4\rangle\times\langle d_1^{-8}d_3^8\rangle.
\end{align*}
\item[Для $128$.]
\begin{align*}
E=\langle d_1^{32}\rangle&\times\prod_{l=1}^{15}\langle d_{2l+1}^{-1}d_{64-(2l+1)}\rangle\times\prod_{l=0}^7\langle d_{2l+1}^{-2}d_{32-(3l+1)}^2\rangle\times\\
&\times\langle d_3^{-4}d_{13}^4\rangle\times\langle d_5^{-4}d_{11}^4\rangle\times\langle d_7^{-4}d_9^4\rangle\times\\
&\times\langle d_1^{-8}d_7^8\rangle\times\langle d_3^{-8}d_5^8\rangle\times\langle d_1^{-16}d_3^{16}\rangle.
\end{align*}
\end{description}
\end{theorem}
\begin{proof}
Достаточно доказать, что
\[
\sqrt{F}\cap E=F.
\]
Будем использовать результаты предложения \ref{p:128}.
\begin{description}
\item[Для $16$.]
Пусть
\[
(1+s_4)^{\delta_0}(1+r_2)^{\delta_1}\in E
\]
для $\{\delta_0,\delta_1\}\subseteq\{0,1\}$.
Тогда
\begin{align*}
(1+s_2)^{\delta_0}(1+r_1)^{\delta_1}&=(1+\delta_0s_2)(1+\delta_1r_1)=\\
&=1+\delta_0s_2+\delta_1r_1+\delta_0\delta_1s_2r_1\equiv
\intertext{по лемме \ref{l:bsp}}
&\equiv1+\delta_0s_2+\delta_1r_1\pmod{2}.
\end{align*}
Так как по определению $E$
\[
\delta_0s_2+\delta_1r_1\equiv0\pmod{2},
\]
а элементы $s_2$ и $r_1$ линейно независимы по модулю $2$, то
\[
\delta_0=\delta_1=0.
\]
Следовательно
\[
\sqrt{F}\cap E=F.
\]
\item[Для $32$.]
Пусть
\[
(1+s_4)^{\delta_0}(1+r_2)^{\delta_1}(1+r_2+r_1)^{\delta_2}(1+r_2+r_3)^{\delta_3}\in E
\]
для $\{\delta_0,\delta_1,\delta_2,\delta_3\}\subseteq\{0,1\}$.
Тогда
\begin{align*}
(1&+s_4)^{\delta_0}(1+r_2)^{\delta_1}(1+r_2+r_1)^{\delta_2}(1+r_2+r_3)^{\delta_3}=\\
&=(1+\delta_0s_4)(1+\delta_1r_2)(1+\delta_2(r_2+r_1))(1+\delta_3(r_2+r_3))\equiv
\intertext{по лемме \ref{l:bsp}}
&\equiv1+\delta_0s_4+\delta_1r_2+\delta_2(r_2+r_1)+\delta_3(r_2+r_3)=\\
&=1+\delta_0s_4+(\delta_1+\delta_2+\delta_3)r_2+\delta_2r_1+\delta_3r_3\pmod{2}.
\end{align*}
Так как по определению $E$
\[
\delta_0s_4+(\delta_1+\delta_2+\delta_3)r_2+\delta_2r_1+\delta_3r_3\equiv0\pmod{2},
\]
а элементы $s_4$, $r_2$, $r_1$ и $r_3$ линейно независимы по модулю $2$, то
\[
\delta_0=\delta_1=\delta_2=\delta_3=0.
\]
Следовательно
\[
\sqrt{F}\cap E=F.
\]
\item[Для $64$.]
Пусть
\begin{multline*}
(1+s_8)^{\delta_0}(1+r_4)^{\delta_1}(1+r_4+r_2)^{\delta_2}(1+r_4+r_6)^{\delta_3}\times\\
\times(1+r_4+r_6+(r_1+r_3+r_7))^{\delta_4}(1+r_4+r_2+(r_3+r_5+r_7))^{\delta_5}\times\\
\times(1+r_4+r_2+(r_1+r_3+r_5))^{\delta_6}(1+r_4+r_6+(r_1+r_5+r_7))^{\delta_7}\in E
\end{multline*}
для $\{\delta_0,\delta_1,\delta_2,\delta_3,\delta_4,\delta_5,\delta_6,\delta_7\}\subseteq\{0,1\}$.
Тогда
\begin{multline*}
(1+s_8)^{\delta_0}(1+r_4)^{\delta_1}(1+r_4+r_2)^{\delta_2}(1+r_4+r_6)^{\delta_3}\times\\
\times(1+r_4+r_6+(r_1+r_3+r_7))^{\delta_4}(1+r_4+r_2+(r_3+r_5+r_7))^{\delta_5}\times\\
\times(1+r_4+r_2+(r_1+r_3+r_5))^{\delta_6}(1+r_4+r_6+(r_1+r_5+r_7))^{\delta_7}=\\
=(1+\delta_0s_8)(1+\delta_1r_4)(1+\delta_2(r_4+r_2))(1+\delta_3(r_4+r_6))\times\\
\times(1+\delta_4(r_4+r_6+(r_1+r_3+r_7)))(1+\delta_5(r_4+r_2+(r_3+r_5+r_7)))\times\\
\times(1+\delta_6(r_4+r_2+(r_1+r_3+r_5)))(1+\delta_7(r_4+r_6+(r_1+r_5+r_7)))=\Delta.
\end{multline*}
Снова применим лемму \ref{l:bsp} и получим
\begin{align*}
\Delta&\equiv1+\delta_0s_8+\delta_1r_4+\delta_2(r_4+r_2)+\delta_3(r_4+r_6)+\\
&+\delta_4(r_4+r_6+(r_1+r_3+r_7))+\delta_5(r_4+r_2+(r_3+r_5+r_7))+\\
&+\delta_6(r_4+r_2+(r_1+r_3+r_5))+\delta_7(r_4+r_6+(r_1+r_5+r_7))=\\
&=1+\delta_0s_8+(\delta_1+\delta_2+\delta_3+\delta_4+\delta_5+\delta_6+\delta_7)r_4+\\
&+(\delta_2+\delta_5+\delta_6)r_2+(\delta_3+\delta_4+\delta_7)r_6+\\
&+(\delta_4+\delta_6+\delta_7)r_1+(\delta_4+\delta_5+\delta_6+\delta_7)r_3+\\
&+(\delta_5+\delta_6+\delta_7)r_5+(\delta_4+\delta_5+\delta_7)r_7\pmod{2}.
\end{align*}
По определению $E$ имеем
\begin{multline*}
\delta_0s_8+(\delta_1+\delta_2+\delta_3+\delta_4+\delta_5+\delta_6+\delta_7)r_4+\\
+(\delta_2+\delta_5+\delta_6)r_2+(\delta_3+\delta_4+\delta_7)r_6+\\
+(\delta_4+\delta_6+\delta_7)r_1+(\delta_4+\delta_5+\delta_6+\delta_7)r_3+\\
+(\delta_5+\delta_6+\delta_7)r_5+(\delta_4+\delta_5+\delta_7)r_7\equiv0\pmod2.
\end{multline*}
Так как элементы $s_8$, $r_4$, $r_2$ и $r_6$ линейно независимы по модулю $2$, то получим систему по модулю $2$
\[
\left\{
\begin{aligned}
\delta_0=0\\
\delta_1+\delta_2+\delta_3+\delta_4+\delta_5+\delta_6+\delta_7=0\\
\delta_2+\delta_5+\delta_6=0\\
\delta_3+\delta_4+\delta_7=0
\end{aligned}
\right.
\longleftrightarrow
\left\{
\begin{aligned}
\delta_0&=0\\
\delta_1&=0\\
\delta_2&=\delta_5+\delta_6\\
\delta_3&=\delta_4+\delta_7
\end{aligned}
\right.\ .
\]
Так как элементы $r_1$, $r_3$, $r_5$ и $r_7$ линейно независимы по модулю $2$, то получим систему по модулю $2$
\begin{multline*}
\left\{
\begin{aligned}
\delta_4+\delta_6+\delta_7=0\\
\delta_4+\delta_5+\delta_6+\delta_7=0\\
\delta_5+\delta_6+\delta_7=0\\
\delta_4+\delta_5+\delta_7=0
\end{aligned}
\right.
\longleftrightarrow
\left\{
\begin{aligned}
\delta_5=0\\
\delta_4+\delta_6+\delta_7=0\\
\delta_6+\delta_7=0\\
\delta_4+\delta_7=0
\end{aligned}
\right.\longleftrightarrow\\
\longleftrightarrow
\left\{
\begin{aligned}
\delta_5=0\\
\delta_4=0\\
\delta_6+\delta_7=0\\
\delta_7=0
\end{aligned}
\right.
\longleftrightarrow
\left\{
\begin{aligned}
\delta_5=0\\
\delta_4=0\\
\delta_6=0\\
\delta_7=0
\end{aligned}
\right.\ .
\end{multline*}
Итак
\[
\delta_0=\delta_1=\delta_2=\delta_3=\delta_4=\delta_5=\delta_6=\delta_7=0,
\]
и потому
\[
\sqrt{F}\cap E=F.
\]
\item[Для $128$.]
Пусть
\begin{multline*}
\Delta=(1+s_{16})^{\delta_0}(1+r_8)^{\delta_1}(1+r_8+r_4)^{\delta_2}(1+r_8+r_{12})^{\delta_3}\times\\
\times(1+r_8+r_{12}+(r_2+r_6+r_{14}))^{\delta_4}(1+r_8+r_4+(r_6+r_{10}+r_{14}))^{\delta_5}\times\\
\times(1+r_8+r_4+(r_2+r_6+r_{10}))^{\delta_6}(1+r_8+r_{12}+(r_2+r_{10}+r_{14}))^{\delta_7}\times\\
\times(1+r_8+r_{12}+(r_2+r_6+r_{14})+(r_3+r_5+r_9+r_{11}+r_{15}))^{\delta_8}\times\\
\times(1+r_8+r_4+(r_6+r_{10}+r_{14})+(r_1+r_5+r_9+r_{13}+r_{15}))^{\delta_9}\times\\
\times(1+r_8+r_4+(r_2+r_6+r_{10})+(r_7+r_9+r_{11}+r_{13}+r_{15}))^{\delta_{10}}\times\\
\times(1+r_8+r_{12}+(r_2+r_{10}+r_{14})+(r_1+r_3+r_9+r_{11}+r_{13}))^{\delta_{11}}\times\\
\times(1+r_8+r_{12}+(r_2+r_{10}+r_{14})+(r_3+r_5+r_7+r_{13}+r_{15}))^{\delta_{12}}\times\\
\times(1+r_8+r_4+(r_2+r_6+r_{10})+(r_1+r_3+r_5+r_7+r_9))^{\delta_{13}}\times\\
\times(1+r_8+r_4+(r_6+r_{10}+r_{14})+(r_1+r_3+r_7+r_{11}+r_{15}))^{\delta_{14}}\times\\
\times(1+r_8+r_{12}+(r_2+r_6+r_{14})+(r_1+r_5+r_7+r_{11}+r_{13}))^{\delta_{15}}\in E
\end{multline*}
для $\left\{\delta_j\mid j\in\{0,1,\dots,15\}\right\}\subseteq\{0,1\}$.
Тогда
\begin{multline*}
\Delta=(1+\delta_0s_{16})(1+\delta_1r_8)(1+\delta_2(r_8+r_4))(1+\delta_3(r_8+r_{12}))\times\\
\times(1+\delta_4(r_8+r_{12}+(r_2+r_6+r_{14})))(1+\delta_5(r_8+r_4+(r_6+r_{10}+r_{14})))\times\\
\times(1+\delta_6(r_8+r_4+(r_2+r_6+r_{10})))(1+\delta_7(r_8+r_{12}+(r_2+r_{10}+r_{14})))\times\\
\times(1+\delta_8(r_8+r_{12}+(r_2+r_6+r_{14})+(r_3+r_5+r_9+r_{11}+r_{15})))\times\\
\times(1+\delta_9(r_8+r_4+(r_6+r_{10}+r_{14})+(r_1+r_5+r_9+r_{13}+r_{15})))\times\\
\times(1+\delta_{10}(r_8+r_4+(r_2+r_6+r_{10})+(r_7+r_9+r_{11}+r_{13}+r_{15})))\times\\
\times(1+\delta_{11}(r_8+r_{12}+(r_2+r_{10}+r_{14})+(r_1+r_3+r_9+r_{11}+r_{13})))\times\\
\times(1+\delta_{12}(r_8+r_{12}+(r_2+r_{10}+r_{14})+(r_3+r_5+r_7+r_{13}+r_{15})))\times\\
\times(1+\delta_{13}(r_4+(r_2+r_6+r_{10})+(r_1+r_3+r_5+r_7+r_9)))\times\\
\times(1+\delta_{14}(r_4+(r_6+r_{10}+r_{14})+(r_1+r_3+r_7+r_{11}+r_{15})))\times\\
\times(1+\delta_{15}(r_8+r_{12}+(r_2+r_6+r_{14})+(r_1+r_5+r_7+r_{11}+r_{13}))).
\end{multline*}
Опять применим лемму \ref{l:bsp} и получим по модулю $2$
\begin{align*}
&\Delta\equiv1+\delta_0s_{16}+\delta_1r_8+\delta_2(r_8+r_4)+\delta_3(r_8+r_{12})+\\
&+\delta_4(r_8+r_{12}+(r_2+r_6+r_{14}))+\\
&+\delta_5(r_8+r_4+(r_6+r_{10}+r_{14}))+\\
&+\delta_6(r_8+r_4+(r_2+r_6+r_{10}))+\\
&+\delta_7(r_8+r_{12}+(r_2+r_{10}+r_{14}))+\\
&+\delta_8(r_8+r_{12}+(r_2+r_6+r_{14})+(r_3+r_5+r_9+r_{11}+r_{15}))+\\
&+\delta_9(r_8+r_4+(r_6+r_{10}+r_{14})+(r_1+r_5+r_9+r_{13}+r_{15}))+\\
&+\delta_{10}(r_8+r_4+(r_2+r_6+r_{10})+(r_7+r_9+r_{11}+r_{13}+r_{15}))+\\
&+\delta_{11}(r_8+r_{12}+(r_2+r_{10}+r_{14})+(r_1+r_3+r_9+r_{11}+r_{13}))+\\
&+\delta_{12}(r_8+r_{12}+(r_2+r_{10}+r_{14})+(r_3+r_5+r_7+r_{13}+r_{15}))+\\
&+\delta_{13}(r_4+(r_2+r_6+r_{10})+(r_1+r_3+r_5+r_7+r_9))+\\
&+\delta_{14}(r_4+(r_6+r_{10}+r_{14})+(r_1+r_3+r_7+r_{11}+r_{15}))+\\
&+\delta_{15}(r_8+r_{12}+(r_2+r_6+r_{14})+(r_1+r_5+r_7+r_{11}+r_{13})).
\end{align*}
По определению $E$ имеем
\begin{multline*}
\delta_0s_{16}+(\delta_1+\delta_2+\dots+\delta_{15})r_8+\\
+(\delta_2+\delta_5+\delta_6+\delta_9+\delta_{11}+\delta_{13}+\delta_{15})r_4+\\
+(\delta_3+\delta_4+\delta_7+\delta_8+\delta_{10}+\delta_{12}+\delta_{14})r_{12}+\\
+(\delta_4+\delta_6+\delta_7+\delta_8+\delta_{10}+\delta_{11}+\delta_{12}+\delta_{13}+\delta_{15})r_2+\\
+(\delta_4+\delta_5+\delta_6+\delta_7+\delta_8+\delta_9+\delta_{10}+\delta_{13}+\delta_{14}+\delta_{15})r_6+\\
+(\delta_5+\delta_6+\delta_7+\delta_9+\delta_{10}+\delta_{11}+\delta_{12}+\delta_{14}+\delta_{15})r_{10}+\\
+(\delta_4+\delta_5+\delta_7+\delta_8+\delta_9+\delta_{11}+\delta_{12}+\delta_{14}+\delta_{15})r_{14}+\\
+(\delta_9+\delta_{11}+\delta_{13}+\delta_{14}+\delta_{15})r_1+\\
+(\delta_8+\delta_{11}+\delta_{12}+\delta_{13}+\delta_{14})r_3+\\
+(\delta_8+\delta_9+\delta_{12}+\delta_{13}+\delta_{15})r_5+\\
+(\delta_{10}+\delta_{12}+\delta_{13}+\delta_{14}+\delta_{15})r_7+\\
+(\delta_8+\delta_9+\delta_{11}+\delta_{13}+\delta_{15})r_9+\\
+(\delta_8+\delta_{10}+\delta_{11}+\delta_{14}+\delta_{15})r_{11}+\\
+(\delta_9+\delta_{10}+\delta_{11}+\delta_{12}+\delta_{15})r_{13}+\\
+(\delta_8+\delta_9+\delta_{10}+\delta_{12}+\delta_{15})r_{15}\equiv0\pmod2.
\end{multline*}
Так как элементы $r_1$, $r_3$, $r_5$, $r_7$, $r_9$,$r_{11}$, $r_{13}$ и $r_{16}$ линейно независимы по модулю $2$, то получим систему по модулю $2$
\[
\left\{
\begin{aligned}
\delta_9+\delta_{11}+\delta_{13}+\delta_{14}+\delta_{15}&=0\\
\delta_8+\delta_{11}+\delta_{12}+\delta_{13}+\delta_{14}&=0\\
\delta_8+\delta_9+\delta_{12}+\delta_{13}+\delta_{15}&=0\\
\delta_{10}+\delta_{12}+\delta_{13}+\delta_{14}+\delta_{15}&=0\\
\delta_8+\delta_9+\delta_{11}+\delta_{13}+\delta_{15}&=0\\
\delta_8+\delta_{10}+\delta_{11}+\delta_{14}+\delta_{15}&=0\\
\delta_9+\delta_{10}+\delta_{11}+\delta_{12}+\delta_{15}&=0\\
\delta_8+\delta_9+\delta_{10}+\delta_{12}+\delta_{15}&=0.
\end{aligned}
\right.
\]
Составим основную матрицу системы и преобразуем её к треугольному виду
\begin{align*}
&\begin{pmatrix}
0&1&0&1&0&1&1&1\\
1&0&0&1&1&1&1&0\\
1&1&0&0&1&1&0&1\\
0&0&1&0&1&1&1&1\\
1&1&0&1&0&1&0&1\\
1&0&1&1&0&0&1&1\\
0&1&1&1&1&0&0&1\\
1&1&1&0&1&0&0&1
\end{pmatrix}
\longleftrightarrow
\begin{pmatrix}
1&0&0&1&1&1&1&0\\
1&1&0&0&1&1&0&1\\
1&1&0&1&0&1&0&1\\
1&0&1&1&0&0&1&1\\
1&1&1&0&1&0&0&1\\
0&1&0&1&0&1&1&1\\
0&1&1&1&1&0&0&1\\
0&0&1&0&1&1&1&1
\end{pmatrix}
\longleftrightarrow\\
&\hspace*{-8pt}\longleftrightarrow
\begin{pmatrix}
1&0&0&1&1&1&1&0\\
0&1&0&1&0&0&1&1\\
0&1&0&0&1&0&1&1\\
0&0&1&0&1&1&0&1\\
0&1&1&1&0&1&1&1\\
0&1&0&1&0&1&1&1\\
0&0&1&0&1&1&1&0\\
0&0&1&0&1&1&1&1
\end{pmatrix}
\longleftrightarrow
\begin{pmatrix}
1&0&0&1&1&1&1&0\\
0&1&0&1&0&0&1&1\\
0&1&0&0&1&0&1&1\\
0&1&1&1&0&1&1&1\\
0&1&0&1&0&1&1&1\\
0&0&1&0&1&1&0&1\\
0&0&1&0&1&1&1&0\\
0&0&1&0&1&1&1&1
\end{pmatrix}
\longleftrightarrow\\
&\hspace*{-8pt}\longleftrightarrow
\begin{pmatrix}
1&0&0&1&1&1&1&0\\
0&1&0&1&0&0&1&0\\
0&0&0&1&1&0&0&0\\
0&0&1&0&0&1&0&0\\
0&0&0&0&0&1&0&0\\
0&0&1&0&1&1&0&1\\
0&0&1&0&1&1&1&0\\
0&0&1&0&1&1&1&1
\end{pmatrix}
\longleftrightarrow
\begin{pmatrix}
1&0&0&1&1&1&1&0\\
0&1&0&1&0&0&1&0\\
0&0&1&0&0&1&0&0\\
0&0&1&0&1&1&0&1\\
0&0&1&0&1&1&1&0\\
0&0&1&0&1&1&1&1\\
0&0&0&1&1&0&0&0\\
0&0&0&0&0&1&0&0
\end{pmatrix}
\longleftrightarrow\\
&\hspace*{-8pt}\longleftrightarrow
\begin{pmatrix}
1&0&0&1&1&1&1&0\\
0&1&0&1&0&0&1&0\\
0&0&1&0&0&1&0&0\\
0&0&0&0&1&0&0&1\\
0&0&0&0&1&0&1&0\\
0&0&0&0&1&0&1&1\\
0&0&0&1&1&0&0&0\\
0&0&0&0&0&1&0&0
\end{pmatrix}
\longleftrightarrow
\begin{pmatrix}
1&0&0&1&1&1&1&0\\
0&1&0&1&0&0&1&0\\
0&0&1&0&0&1&0&0\\
0&0&0&1&1&0&0&0\\
0&0&0&0&1&0&0&1\\
0&0&0&0&1&0&1&0\\
0&0&0&0&1&0&1&1\\
0&0&0&0&0&1&0&0
\end{pmatrix}
\longleftrightarrow\\
&\hspace*{-8pt}\longleftrightarrow
\begin{pmatrix}
1&0&0&1&1&1&1&0\\
0&1&0&1&0&0&1&0\\
0&0&1&0&0&1&0&0\\
0&0&0&1&1&0&0&0\\
0&0&0&0&1&0&0&1\\
0&0&0&0&0&0&1&1\\
0&0&0&0&0&0&1&0\\
0&0&0&0&0&1&0&0
\end{pmatrix}
\longleftrightarrow
\begin{pmatrix}
1&0&0&1&1&1&1&0\\
0&1&0&1&0&0&1&0\\
0&0&1&0&0&1&0&0\\
0&0&0&1&1&0&0&0\\
0&0&0&0&1&0&0&1\\
0&0&0&0&0&1&0&0\\
0&0&0&0&0&0&1&1\\
0&0&0&0&0&0&1&0
\end{pmatrix}
\longleftrightarrow\\
&\hspace*{-8pt}\longleftrightarrow
\begin{pmatrix}
1&0&0&1&1&1&1&0\\
0&1&0&1&0&0&1&0\\
0&0&1&0&0&1&0&0\\
0&0&0&1&1&0&0&0\\
0&0&0&0&1&0&0&1\\
0&0&0&0&0&1&0&0\\
0&0&0&0&0&0&1&1\\
0&0&0&0&0&0&0&1
\end{pmatrix}\ .
\end{align*}
Отсюда следует
\[
\delta_8=\delta_9=\delta_{10}=\delta_{11}=\delta_{12}=\delta_{13}=\delta_{14}=\delta_{15}=0,
\]
и система становится такой
\begin{multline*}
\delta_0s_{16}+(\delta_1+\delta_2+\delta_3+\delta_4+\delta_5+\delta_6+\delta_7)r_8+\\
+(\delta_2+\delta_5+\delta_6)r_4+(\delta_3+\delta_4+\delta_7)r_{12}+\\
+(\delta_4+\delta_6+\delta_7)r_2+(\delta_4+\delta_5+\delta_6+\delta_7)r_6+\\
+(\delta_5+\delta_6+\delta_7)r_{10}+(\delta_4+\delta_5+\delta_7)r_{14}\equiv0\pmod2,
\end{multline*}
которая совпадает с системой для $64$.
Поэтому
\[
\delta_0=\delta_1=\delta_2=\delta_3=\delta_4=\delta_5=\delta_6=\delta_7=0,
\]
что даёт
\[
\sqrt{F}\cap E=F.
\]
\end{description}
\end{proof}

\end{selectlanguage}

\begin{thebibliography}{99}
\bibitem{aleev1}
\textbf{Р.~Ж.~Алеев,}
\textit{Единицы полей характеров и центральные единицы целочисленных групповых колец конечных групп,} Матем. труды. 2000.  Т.~3, вып.~1. С.~3--37.
\bibitem{aleev2}
\textbf {Р.~Ж.~Алеев,}
\textit{Центральные элементы целочисленных групповых колец,} Алгебра и Логика. 2000. Т.~39, вып.~5. С.~513--525.
 \bibitem{aleev3+}
{\bf Р.~Ж.~Алеев,} \textsf{Центральные единицы целочисленных групповых колец конечных групп.}~--- Дисс. д-ра физ.-мат. наук,  Челябинск, 2000.~--- 355~с.
\bibitem{amgod}
\textbf{Р.~Ж.~Алеев, О.~В.~Митина, А.~Д.~Годова,}
\emph{Круговые единицы сравнимые с $1$ по модулю $2$ для $2$-круговых полей,}
Международная конференция МАЛЬЦЕВСКИЕ ЧТЕНИЯ 16–20 ноября 2020 г.
Тезисы докладов.\\
\begin{small}
\texttt{http://www.math.nsc.ru/conference/malmeet/20/maltsev20.pdf}\\
\end{small}
с. 134.
\bibitem{amkhan}
\textbf{Р.~Ж.~Алеев, О.~В.~Митина,  Т.~А.~Ханенко,}
\textit{Нахождение единиц целочисленных групповых колец циклических групп порядков $16$ и $32$,} Челяб. физ.-мат. журн. 2016. Т.~1, вып.~4. С.~30--55.
\bibitem{amkhan1}
\textbf{Р.~Ж.~Алеев, О.~В.~Митина,  Т.~А.~Ханенко,}
\textit{Описание групп единиц целочисленного группового кольца циклической группы порядка $16$,} Тр. Ин-та математики и механики УрО РАН. 2017. Т.~23, вып.~4. С.~32--42.
\bibitem{amkhan2}
\textbf{Р.~Ж.~Алеев, О.~В.~Митина,  Т.~А.~Ханенко,}
\textit{Локальные единицы целочисленного группового кольца циклической группы порядка $64$ для характера с полем характера $Q_{64}$,} 
Челяб. физ.-мат. журн. 2018. Т.~3, вып.~3.  C.~253--275.
\bibitem{amkhr}
\textbf{Р.~Ж.~Алеев, О.~В.~Митина, Е.~А.~Христенко,}
\textit{Сравнение по модулю $2$ круговых единиц в полях $Q_{16}$ и $Q_{32}$,} Челяб. физ.-мат. журн. 2016. Т.~1, вып.~4. С.~8--29.
\bibitem{amukh}
\textbf{Р.~Ж.~Алеев, И.~Р.~Мухамадеева,}
\textit{Вычисление квантовых факториалов и к ним обратных,} Челяб. физ.-мат. журн. 2016. Т.~1, вып.~1. С.~6--15.
\bibitem{bor}
\textbf{ З.~И.~Боревич,  И.~Р.~Шафаревич,}
\textsf{Теория чисел,} М.: Наука, 1985.
\bibitem{vdv}
\textbf{ Б.~Л.~ван дер Варден,}
\textsf{Алгебра:} --- 2-е изд. --- М.: Наука. Гл. ред. физ.-мат. лит-ры, 1979. --- 624~с.
\bibitem{gauss}
\textbf{К.~Ф.~Гаусс,}
\textsf{Труды по теории чисел.} Общая редакция И.М. Виноградова. Комментарии Б.Н. Делоне. Перевод В.Б. Демьянова.
Москва: Издательство Академии Наук СССР, 1959. --- Серия «Классики науки». --- 979~с.
\bibitem{leng}
\textbf{С.~Ленг,}
 \textsf{Алгебраические числа:} Пер. с англ.--- М.: Мир --- 1966. --- 225~с.
\bibitem{aleev3}
\textbf{R.~\v{Z}.~Aleev,}
\textit{Higman's central unit theory, units of integral group rings of finite cyclic groups and Fibonacci numbers,} 
Intern. J. of Algebra and Comp.,1994  Vol.~4, No.~3(1994), pp.~309--358.
\bibitem{fk}
\textbf{T. Fukuda, K. Komatsu,}
\textit{Weber's class number problem in the cyclotomic $\mathbb{Z}_2$-extension of $\mathbb{Q}$, III.,}
Int. J. Number Theory 7, No. 6, 1627-1635 (2011).
\bibitem{gap}
\textsf{The GAP Group, GAP - Groups, Algorithms, and Programming,}
 Version 4.11.0; 2020\\
(\texttt{http://www.gap-system.org})
\bibitem{has}
\textbf{H. Hasse,}
\textsf{Number theory},
 Berlin and New York : Springer-Verlag. --- 1980. ---  xvii + 638 pp.
\bibitem{li}
\textbf{J.~J.~ Liang,}
\textit{On the integral basis of the maximal real subfield of a cyclotomic field,}
Journ. f\"{u}r die Reine und Angew. Math., Band~286/287, 1976, pp.~223--226.
\bibitem{lin}
\textbf{F.~J.~van der Linden,}
\textit{Class Number Computations of Real Abelian Number Fields,}
Mathematics of Computations, 39, № 160 (Oct. 1982), p. 693--707.
\bibitem{mas}
\textbf{J.~ M.~Masley,}
\textit{Solution of small class number problems for cyclotomic fields,}
Compositio Mathematica, Vol. 33, Fasc. 2, 1976, P. 179--186.
\bibitem{mil1}
\textbf{J.~C.~Miller,}
\textit{Class numbers of totally real fields and applications to the Weber class number problem,}
Acta Arithmetica 164, no. 4 (2014), 381--397.
\bibitem{mil2}
\textbf{J.~C.~Miller,}
\textit{Class numbers of real cyclotomic fields of composite conductor,}
LMS Journal of Computation and Mathematics 17, Special Issue A (2014), 404-417.
\bibitem{sin}
\textbf{W.~Sinnott,}
\textit{On the Stickelberger ideal and circular units of a cyclotomic field,}
Ann. of Math., Vol.~108, no.~1(1978), pp.~107--134
\end{thebibliography}
\end{document}